\documentclass[reqno]{amsart}
\usepackage{graphicx}
\usepackage{amscd}
\usepackage{amsmath}
\usepackage{amsfonts}
\usepackage{amssymb}
\usepackage{cite}
\usepackage{xcolor}
\usepackage{tabularx,multirow}
\usepackage[caption=false]{subfig}
\def\R {\mathbb{R}}
\def\C {\mathcal{C}}
\def\S {\mathcal{S}}
\def\P{\mathcal{P}}
\def\L{\mathcal{L}}
\def\a{\mathfrak{a}}
\def\b{\mathfrak{b}}
\def\s{\mathfrak{s}}
\def\z{\mathfrak{z}}

\newtheorem{proposition}{Proposition}[section]
\newtheorem{theorem}[proposition]{Theorem}
\newtheorem{corollary}{Corollary}[section]
\newtheorem{lemma}{Lemma}[section]
\theoremstyle{definition}
\newtheorem{definition}{Definition}[section]
\newtheorem{remark}{Remark}[section]
\numberwithin{equation}{section}

\begin{document}

\title[Contact Muskat Problem]
{On a local solvability of the contact Muskat problem}

\author[ N. Vasylyeva]
{ Nataliya Vasylyeva}

\address{Institute of Applied Mathematics and Mechanics of NAS of Ukraine
\newline\indent
G.Batyuka st.\ 19, 84100 Sloviansk, Ukraine; and
\newline\indent
Dipartimento di Matematica, Politecnico di Milano
\newline\indent
Piazza Leonardo da Vinci 32, 20133 Milano, Italy}
\email[N.Vasylyeva]{nataliy\underline{\ }v@yahoo.com}


\begin{abstract}
In the paper, we discuss the two-dimensional contact Muskat problem
with zero surface tension on a free boundary. The initial shape of
the unknown interface is a smooth simple curve which forms acute
corners $\delta_{0}$ and $\delta_{1}$ with  fixed boundaries. Under
suitable assumptions on the given data, the one-to-one local
classical solvability of this problem is proved. We also describe
the sufficient conditions on the data in the model which provide
the existence of the "waiting time" phenomenon.
\end{abstract}

\maketitle

\section{Introduction}
\label{s1}

\noindent The Muskat problem proposed by Morris Muskat \cite{Mu} in
1934 is a classical model describing the viscous displacement in a
two-phase fluid system occupied a two-dimensional porous medium. As
noted in \cite{Mu,MW}, this problem describes  the encroachment of
water into oil sand and deals with the secondary phase of the oil
recovery process, where water injection is utilized to enlarge the
pressure in the oil reservoir and to move the oil to the production
well. The fluids' motion is subjected by the experimental Darcy law.
The boundary between the two liquid regions is an unknown needing to
be searched. This moving (unknown) interface is usually called a
free boundary and, accordingly, the Muskat problem is a free
boundary problem.

For the last 70 years, thanks to various applications in
hydrodynamics, chemistry and oil industry, the Muskat model and the
associated problems have attracted a wide scientific interests among
the mathematical, engineering and chemical community. As mentioned
in \cite{GGHP}, there is an enormous literature on these problems in
various geometries and physical settings, where analytical and
numerical investigations are carried out via various techniques and
approaches.

In this paper, we focus on  the two-dimensional  contact Muskat
problem in the case of the  zero surface tension (ZST) of the moving
interface. Let $\Omega\subset\R^{2}$ be a bounded domain with a
smooth boundary
 $\partial \Omega\in\C^{2+\beta},$ $\beta\in(0,1).$ For each $\tau\in[0,T],$ a simple
  curve $\Gamma (\tau)\subseteq\bar{\Omega}$ splits the domain $\Omega$ onto two sub-domains $\Omega_{1}(\tau)$ and $\Omega_{2}(\tau)$ such that
\[
\Omega=\Omega_{1}(\tau)\cup\Omega_{2}(\tau)\cup\Gamma(\tau),\quad
\Omega_{1}(\tau)\cap\Omega_{2}(\tau) =\emptyset,\quad
\partial\Omega_{1}(\tau)\cap\partial\Omega_{2}(\tau)=\Gamma(\tau)
\]
and, besides, $\Gamma(\tau)$ intersects with $\partial\Omega$ at two
corner points $\mathbf{A}_{0}$ and $\mathbf{A}_{1}$, i.e.
\[
\partial\Omega\cap\Gamma(\tau)=\{\mathbf{A}_{0},\mathbf{A}_{1}\}
\]
(see Figure \ref{fig:1} for geometric settings).
\begin{figure}[t]
\centering \subfloat[\label{fig:1a}] {
\includegraphics[width=0.2\linewidth]{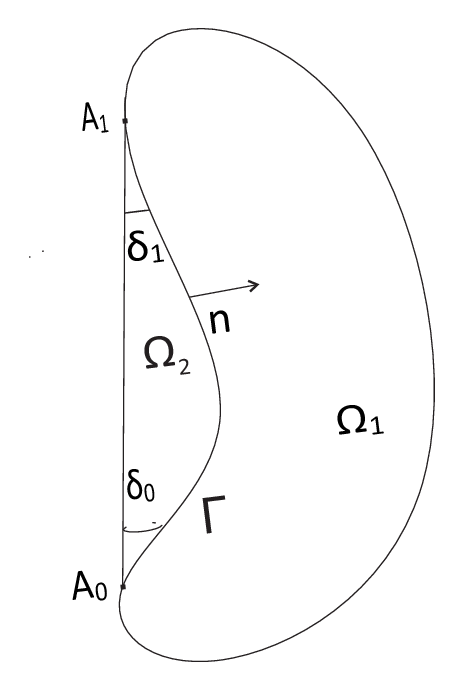}
} \subfloat[\label{fig:1b}] {
\includegraphics[width=0.2\linewidth]{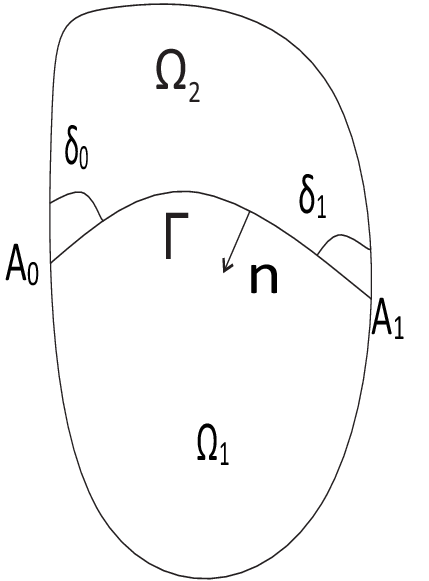}
} \caption{Typical configurations of the initial domains
$\Omega_{i}$ and $\Gamma$.} \label{fig:1}
\end{figure}

For fixed positive quantities $T,$ $k_{1}$ and $k_{2},$ $k_{1}\neq
k_{2}$,  we look for unknown fluid pressures
$\mathrm{P}_{1}=\mathrm{P}_{1}(y_{1},y_{2},\tau),$
$\mathrm{P}_{2}=\mathrm{P}_{2}(y_{1},y_{2},\tau)$ in the
corresponding liquid domains $\Omega_{1}(\tau)$ and
$\Omega_{2}(\tau)$, and the unknown interface $\Gamma(\tau)$
satisfying the system
\begin{equation}\label{1.1}
\begin{cases}
\Delta_{y}\mathrm{P}_{i}=0\quad\text{in}\quad\Omega_{i}(\tau),\quad \tau\in(0,T), \quad i=1,2,\\
\mathrm{P}_{1}-\mathrm{P}_{2}=0\quad\text{on}\quad \Gamma(\tau),\quad\tau\in[0,T],\\
V_{n}=-k_{1}\frac{\partial \mathrm{P}_{1}}{\partial n_{\tau}}=-k_{2}\frac{\partial \mathrm{P}_{2}}{\partial n_{\tau}} \quad\text{on}\quad \Gamma(\tau),\quad\tau\in[0,T],\\
\mathrm{P}_{1}=\mathrm{p}_{1}(y)\quad\text{on}\quad\partial\Omega_{1}(\tau)\backslash\Gamma(\tau),\quad\tau\in[0,T],\\
\mathrm{P}_{2}=\mathrm{p}_{2}(y)\quad\text{on}\quad\partial\Omega_{2}(\tau)\backslash\Gamma(\tau),\quad\tau\in[0,T],\\
\Gamma(0)=\Gamma\quad\text{is given}.
\end{cases}
\end{equation}
Here $\Delta_{y}=\frac{\partial^{2}}{\partial
y_{1}^{2}}+\frac{\partial^{2}}{\partial y_{2}^{2}}$, $n_{\tau}$
denotes the outward normal to $\Gamma(\tau)$ directed in
$\Omega_{1}(\tau)$, the functions $\mathrm{p}_{1}$ and
$\mathrm{p}_{2}$ are prescribed. Finally, the symbol $V_{n}$ stands
for the velocity of the free boundary $\Gamma(\tau)$ in the
direction of $n_{\tau}$.

It is apparent that the mathematical statement of the contactless
Muskat problem  is the similar to \eqref{1.1} where $\Gamma(\tau)$
is assumed to be a simple smooth or nonsmooth  closed curve
separating $\Omega$ into two corresponding subdomains. It is worth
noting that the accounting  surface tension of the interface
$\Gamma(\tau)$ in the $n$-dimensional Muskat model is achieved via
introducing the term $\sigma\mathrm{H}_{\Gamma(\tau)}$ in the
right-hand side of the first condition on $\Gamma(\tau)$ in
\eqref{1.1}, that is
\[
\mathrm{P}_{1}-\mathrm{P}_{2}=\sigma\mathrm{H}_{\Gamma(\tau)}\quad\text{on}\quad\Gamma(\tau),\,
\tau\in[0,T],
\]
where $\sigma$ is the coefficient of the surface tension, while
$\mathrm{H}_{\Gamma(\tau)}$ denotes the $(n-1)$-fold mean curvature.

The study of the (contactless) Muskat problem has a long history. We
do not provide, here, a complete survey on the outcomes related to
this and associated problems, but rather present some of them. The
first result on the classical solvability on a $2$-dim. domain in
the case of the nonzero surface tension (NZST) was established in
\cite{HTY}. The stability of a circular steady-state in the
unbounded domain is discussed in \cite{FT}. In particular, the
authors state that the equilibrium is not asymptotically stable. The
one-valued classical solvability of the Muskat problem in a
horizontally periodic geometry is proved in \cite{EEM} in the case
of presence of the surface tension and of gravity. Besides, the
authors obtain the exponential stability of certain flat equilibria
and, utilizing the bifurcation theory, they find finger shaped
(unstable) steady-states. We refer \cite{EEM,EMM,EMW}, where
employing the abstract parabolic theory, the authors refine and
extend the aforementioned results. In \cite{MM}, the well-posedness
of the Muskat problem with positive $\sigma$ including the criterion
of the global solvability is discussed in the subcritical Sobolev
spaces. In fine, we quote \cite{BV1}, where the existence of the
waiting time in the two-dimensional Muskat problem with NZST stated
in domains with singular boundaries is established. We recall that
the waiting time phenomenon in the Muskat model and the associated
problems means the existence of a positive time $T^{*}$, such that
the geometry of the free boundary at singularities preserves and/or
these singular points do not move for each $t\in[0,T^{*}]$.

Coming to the Muskat problem in the ZST case, it is well-known, that
this model can be ill-posed. Namely, this option occurs if the
Rayleigh-Taylor condition is not fulfilled. Other words, it is
happen if either the more viscous liquid displaces the less viscous
fluid, or the heavier liquid is located above the lighter one (see
for details \cite{Am}). In \cite{Y}, the first local classical
solvability was established by Newtonian iteration technique. We
refer works \cite{BCG,CCFG,CCGS,CGZ,CGS,CGO,EMM2,Mat,SCH,V5,Y1},
where various aspects associated with the equivalence of statement,
local and global well-posedness, regularity, breakdown of
smoothness, finite time turning and stability shifting are discussed
with employing energy approach, methods of complex analysis, the
Cauchy-Kowalewski theorem, fixed-point arguments, the abstract
well-posedness result based  on continuous maximal regularity. The
evolution of the singularities being located on the initial shape of
the moving boundaries are described in \cite{APW,BV2,GGHP,GGNP}. For
further acquaintance with results, we refer readers to the papers
\cite{CCFG,Mat,Mat2} and monographs \cite{PS}, where a brief
overview (including a historic survey) of the Muskat problem is
done.

As for the study of the contact free boundary problems to Laplace
equations similar to \eqref{1.1}, we mention \cite{BF,MV,V3}, where
the contact Hele-Shaw free boundary problems (sometimes calling as
the one-phase Muskat problem) with and without surface tension are
discussed. Namely, in these works, the unknown interface and the
fixed boundaries form either acute angles \cite{BF,V3} or the right
angle \cite{MV}. Besides, the authors proved the local classical
solvability of these problem and establish the existence of the
waiting time under certain assumptions on the given data in these
models.

To the author's best knowledge, there are no works in the literature
addressing the study of the evolution of the moving interface in the
contact Muskat problem \eqref{1.1}. The aim of this article is to
fill this gap, providing a (locally in time) well-posedness result
along with the regularity of solution in the weighted H\"{o}lder
spaces. Moreover, under certain assumptions on the initial geometry
of the domains and the given functions and parameters in the model,
we analyze   the existence of the waiting time phenomenon in
\eqref{1.1}.

In the course of this investigation, we exploit the following
technique including the $4^{th}$ main steps. In the first stage,
introducing the special type vector field and utilizing the
Hanzawa-type transformation (accounting the nonsmooth boundaries of
$\Omega_{i}(0)$), we reduce problem \eqref{1.1} to a nonlinear
problem defined in fixed domains $\Omega_{i}(0)\times(0,T),$
$i=1,2$. The we linearize this nonlinear problem on the initial data
$\mathrm{P}_{i}(y,0)$ and on the special constructed function
$\rho(y,t)$ related with the initial shape of the interface
$\Gamma.$ After that, we solve the nonclassical linear interface
problem to Poisson equations stated in $\Omega_{i}(0)\times(0,T)$
having nonsmooth boundaries $\partial\Omega_{i}(0)\times(0,T)$. The
nonclassical character of this problem is related with a dynamic
boundary condition as the one of the transmission conditions on
$\Gamma\times[0,T]$. To prove the one-valued classical  solvability
of this linear problem in the weighted H\"{o}lder classes (which is
more natural in the case of singular boundaries), we collect the
continuation approach with studying the corresponding model problems
providing the \textit{a priori} estimates. In fine, coming to the
nonlinear problem and utilizing contraction mapping principle, we
establish the local classical one-to-one solvability of \eqref{1.1}.
Further, appealing to the properties of the weighted H\"{o}lder
classes and employing the classical solvability of the contact
Muskat problem, we arrive at the existence of the waiting time
phenomenon under  assumptions on the given data providing the local
well-posedness. The later, in particular, means that
$$\delta_{0},\delta_{1}\in(0,\pi/4), \quad 0<k_{2}<k_{1}\quad \frac{\partial P_{i}(y,0)}{\partial n_{\tau}}|_{\Gamma}<0,\quad i=1,2.$$

\subsection*{Outline of the Paper} The paper is organized as
follows. In Section \ref{s2}, we introduce the functional setting
and some notations. Here, we also describe the key properties of the
weighted H\"{o}lder classes and the  special  transformations which
will be main tools in the analysis of the linear problems in
Sections \ref{s4}-\ref{s5}. The principal results of these paper are
stated in Theorems \ref{t3.1} and \ref{t6.2} and Corollary
\ref{3.1}, where theorems concern with the solvability of
\eqref{1.1}, while Corollary \ref{3.1} deals with the existence of
the waiting time phenomenon. In Section \ref{s3}, we reformulate
\eqref{1.1} in the form of the nonlinear problem in the fixed
domains and, then, making main assumptions, we establish the local
one-valued classical solvability of \eqref{1.1} (Theorem \ref{3.1})
in the case of $\delta_{0},\delta_{1}\in(0,\pi/4)$  being the
rational part of $\pi$. The proof of Theorem \ref{t3.1} is carried
out in Sections \ref{s4}-\ref{s5}. In particular, Sections \ref{s4}
is devoted to the analysis of the nonclassical transmission problems
with the dynamic boundary conditions in  plane corners. Actually,
the principal results of this section stated in Theorems
\ref{t4.1}-\ref{t4.3} have the independent significance in the
theory of transmission boundary value problems in nonsmooth domains.
 The local solvability of the corresponding
nonlinear problem is discussed in Section \ref{s5}. Theorem
\ref{t6.1} concerning with the  solvability of \eqref{1.1} in the
case of arbitrary $\delta_{i}\in(0,\pi/4)$, $i=1,2,$ is stated and
proved in Section \ref{s6}.


\section{Functional spaces and notations}
\label{s2}

\noindent
Throughout this work, the symbol $C$ will denote a generic positive constant, depending only on the structural quantities of the model.
 We will carry out our analysis in the framework of the weighted H\"{o}lder spaces (which is first  introduced in \cite{BV4}).

To introduce these classes, we first consider a domain $D\subset
\R^{2}$ with a boundary $\partial D$ having a finite number
 of the corner points
$\mathbf{A}_{0},\mathbf{A}_{1},...,\mathbf{A}_{\mathrm{q}}$.
Besides, for each fixed $T>0$, we denote $D_{T}=D\times(0,T)$,
$\partial D_{T}=\partial D\times [0,T]$.

\noindent Till the end of the paper,
\[
\alpha,\beta\in(0,1),\quad s\in\R
\]
are arbitrarily but fixed. For each points $y,\bar{y}\in\bar{D},$ we
put
\begin{align*}
r_{i}(y)&=|\mathbf{A}_{i}-y|,\quad r_{i}(\bar{y})=|\mathbf{A}_{i}-\bar{y}|,\quad i=0,1,...,\mathrm{q},\\
r(y)&=\min\{r_{0}(y),r_{1}(y),...,r_{\mathrm{q}}(y)\},\quad r(\bar{y})=\min\{r_{0}(\bar{y}),r_{1}(\bar{y}),...,r_{\mathrm{q}}(\bar{y})\},\\
r(y,\bar{y})&=\min\{r(y),r(\bar{y})\}.
\end{align*}

Denoting
\begin{align*}
\langle v\rangle_{s,D}^{(\beta)}&=\sup\Big\{r^{-s+\beta}(y,\bar{y})\frac{|v(y)-v(\bar{y})|}{|y-\bar{y}|^{\beta}}:\quad y,\bar{y}\in\bar{D},\quad 0<|y-\bar{y}|<r(y,\bar{y})/2\Big\},\\
\langle v\rangle_{y,s,D_T}^{(\beta)}&=\sup\Big\{r^{-s+\beta}(y,\bar{y})\frac{|v(y,t)-v(\bar{y},t)|}{|y-\bar{y}|^{\beta}}:\quad (y,t),(\bar{y},t)\in\bar{D}_{T},\quad 0<|y-\bar{y}|<r(y,\bar{y})/2\Big\},\\
\langle v\rangle_{t,s,D_T}^{(\alpha)}&=\sup\Big\{r^{-s}(y)\frac{|v(y,t_{1})-v(y,t_{2})|}{|t_{1}-t_{2}|^{\alpha}}:\quad (y,t_{1}),(y,t_{2})\in\bar{D}_{T},\quad t_{1}\neq t_{2}\Big\},\\
[v]_{s,D_T}^{(\beta,\alpha)}&=\sup\Big\{r^{-s+\beta}(y,\bar{y})\frac{|v(y,t_{1})-v(\bar{y},t_{1})-v(y,t_{2})+v(\bar{y},t_{2})|}{|y-\bar{y}|^{\beta}|t_{1}-t_{2}|^{\alpha}}:\quad (y,\bar{y})\in\bar{D},\quad t_{1},t_{2}\in[0,T],\\
&\qquad\qquad t_{1}\neq t_{2},\quad 0<|y-\bar{y}|<r(y,\bar{y})/2\Big\},
\end{align*}
we define spaces $E_{s}^{l+\beta}(\bar{D})$ and
$E_{s}^{l+\beta,\alpha,\beta}(\bar{D}_{T})$ for integer nonnegative
$l$.
\begin{definition}\label{d2.1}
Functions $v=v(y)$ and $u=u(y,t)$ belong to  classes $E_{s}^{l+\beta}(\bar{D})$ and $E_{s}^{l+\beta,\alpha,\beta}(\bar{D}_{T})$ if the norms here below are finite
\begin{align*}
\|v\|_{E_{s}^{l+\beta}(\bar{D})}&=\sum_{|\iota|=0}^{l}\underset{\bar{D}}{\sup}\,
r^{-s+|\iota|}(y)|\mathcal{D}_{y}^{\,\iota}v(y)|+
\sum_{|\iota|=l}\langle \mathcal{D}_{y}^{\,\iota}v\rangle_{s-l,D}^{(\beta)},\quad |\iota|=\iota_{1}+\iota_{2},\\
\|u\|_{E_{s}^{l+\beta,\alpha,\beta}(\bar{D}_T)}&=\sum_{|\iota|=0}^{l}\Big\{\underset{\bar{D}_T}{\sup}\,
r^{-s+|\iota|}(y)|\mathcal{D}_{y}^{\,\iota}u(y,t)|+ \langle
\mathcal{D}_{y}^{\,\iota}u\rangle_{y,s-|\iota|,D_T}^{(\beta)}
+\langle
\mathcal{D}_{y}^{\,\iota}u\rangle_{t,s-|\iota|,D_T}^{(\alpha)}
+[\mathcal{D}_{y}^{\,\iota}u]^{(\beta,\alpha)}_{s-|\iota|,D_{T}}
\Big\}.
\end{align*}
\end{definition}
\noindent The spaces $E_{s}^{l+\beta}(\partial D)$ and
$E_{s}^{l+\beta,\alpha,\beta}(\partial D_{T})$ are defined in a
similar way.

If the domain $D$ has a smooth boundary (in the case of the absence
of any corner points), then the classes $E_{s}^{l+\beta}(\bar{D})$
and $E_{s}^{l+\beta,\alpha,\beta}(\bar{D}_{T})$ boil down to the
usual H\"{o}lder spaces $\C^{l+\beta}(\bar{D})$ and
$\C^{l+\beta,\alpha,\beta}(\bar{D}_{T})$.
\begin{definition}\label{d2.2}
For any real number $\hat{s}$ and integer $l\geq 1$, we define the
Banach space
$\mathcal{E}_{s,\hat{s}}^{l+\beta,\alpha,\beta}(\bar{D}_{T})$
consisting of all functions $u:$ $r^{\hat{s}}(y)u\in
E_{s}^{l+\beta,\alpha,\beta}(\bar{D}_{T})$, $\frac{\partial
u}{\partial t}\in E_{s-1}^{l-1+\beta,\alpha,\beta}(\bar{D}_{T})$,
with the finite norm
\[
\|u\|_{\mathcal{E}_{s,\hat{s}}^{l+\beta,\alpha}(\bar{D}_{T})}=\|r^{\hat{s}}u\|_{E_{s}^{l+\beta,\alpha,\beta}(\bar{D}_{T})}+
\|\partial u/\partial t\|_{E_{s-1}^{l-1+\beta,\alpha,\beta}(\bar{D}_{T})}.
\]
\end{definition}
\noindent The space
$\mathcal{E}_{s,\hat{s}}^{l+\beta,\alpha,\beta}(\partial D_{T})$ is
introduced with the same manner.

Finally, in the spaces $E^{l+\beta,\alpha,\beta}_{s}(\bar{D}_{T})$ and
$\mathcal{E} ^{l+\beta,\alpha,\beta}_{s,\hat{s}}(\bar{D}_{T})$  we secrete the subspaces
 \begin{align*}
\underset{0}{E}\, ^{l+\beta,\alpha,\beta}_{s}(\bar{D}_{T})&=\{u\in E^{l+\beta,\alpha,\beta}_{s}(\bar{D}_{T}):\, \mathcal{D}_{y}^{\iota}u(y,0)=0,\,  |\iota|=0,1,...,l\},\\
\underset{0}{\mathcal{E}}\, ^{l+\beta,\alpha,\beta}_{s,\hat{s}}(\bar{D}_{T})&=
\Big\{u\in \mathcal{E}^{l+\beta,\alpha,\beta}_{s,\hat{s}}(\bar{D}_{T}):\, \mathcal{D}_{y}^{\iota}u(y,0)=0,\, |\iota|=0,1,...,l,\, \mathcal{D}_{y}^{j}\frac{\partial u}{\partial t}(y,0)=0,\, |j|=0,1,...,l-1\Big\},
\end{align*}

At this point, we recall some useful properties of these spaces that
will be exploited in the  further analysis. In particular, the
straightforward calculations provide the following claim.
\begin{corollary}\label{c2.1}
Let $D$ be a bounded domain. For any  positive  $s_{1}$, we assume that $v\in E_{s_{1}}^{l+\beta}(\bar{D})$ and $u\in E_{s_{1}}^{l+\beta,\alpha,\beta}(\bar{D}_{T}).$  Then for all $s\in[0,s_{1})$, there hold
\[
\|v\|_{E_{s}^{l+\beta}(\bar{D})}\leq C\|v\|_{E_{s_1}^{l+\beta}(\bar{D})}\quad \text{and}\quad
\|u\|_{E_{s}^{l+\beta,\alpha,\beta}(\bar{D}_{T})}\leq C\|u\|_{E_{s_1}^{l+\beta,\alpha,\beta}(\bar{D}_T)}
\]
with the positive quantity $C$ depending only on $|D|^{s_{1}-s}$, where $|D|$ is the Lebesgue measure of $D$.
\end{corollary}

At last, we give equivalent definitions of the weighted H\"{o}lder spaces in the case of  $D$ being a plane corner. To this end, for some fixed $\delta\in[0,\pi/2),$ we define the plane corners:
\begin{align}\label{2.1}\notag
G_{1}&=\{(y_{1},y_{2}):\quad y_{1}>0,\quad y_{2}<y_{1}\tan\delta\},\quad G_{1,T}=G_{1}\times (0,T),\\
G_{2}&=\{(y_{1},y_{2}):\quad y_{1}>0,\quad y_{2}>
y_{1}\tan\delta\},\quad G_{2,T}=G_{2}\times (0,T),\\\notag
g&=\{(y_{1},y_{2}):\quad y_{1}\geq 0,\quad y_{2}=
y_{1}\tan\delta\},\quad g_{T}=g\times[0,T],
\end{align}
and introduce new independent variables
\begin{equation}\label{2.2}
x_{1}=\ln(y_{1}^{2}+y_{2}^{2})^{1/2}\quad\text{and}\quad x_{2}=\arctan(y_{2}/y_{1}).
\end{equation}

It is apparent that, this mapping  transforms the plane corners
$G_{i},$ $i=1,2,$ to the strips $B_{i}:$
\begin{align*}
B_{1}&=\{(x_{1},x_{2}):\quad x_{1}\in\R,\quad -\pi/2<x_{2}< \delta\},\quad B_{1,T}=B_{1}\times (0,T),\\
B_{2}&=\{(x_{1},x_{2}):\quad x_{1}\in\R,\quad
\delta<x_{2}<\pi/2\},\quad B_{2,T}=B_{2}\times (0,T).
\end{align*}
Besides, the image of $g$ after transformation \eqref{2.2} is
\[
b=\{(x_{1},x_{2}):\quad x_{1}\in\R,\quad x_{2}=\delta\}.
\]
Then, direct calculations yield  the following  assertion.
\begin{corollary}\label{c2.2}
For $i=1,2,$ there are  equivalences:
\begin{align*}
v_{i}&\in E_{s}^{l+\beta}(\bar{G}_{i})\quad\text{iff}\quad v_{i}(y(x))e^{-sx_{1}}\in\C^{l+\beta}(\bar{B}_{i}),\\
u_{i}&\in
E_{s}^{l+\beta,\alpha,\beta}(\bar{G}_{i,T})\quad\text{iff}\quad
u_{i}(y(x),t)e^{-sx_{1}}\in\C^{l+\beta,\alpha,\beta}(\bar{B}_{i,T}).
\end{align*}
\end{corollary}
\noindent The similar results hold for the functions $v_{i}(y)$ and
$u_{i}(y,t)$ defined on $g$ and $g_{T}$, respectively.

Along the paper for $p>0$, we will also encounter the  weighted
Sobolev spaces  $W^{l,p}_{s}(D)$ with the norms
\[
\|v\|_{l,p,s;D}=\Big(\sum_{|\iota|\leq
l}\int_{D}r^{p(s+|\iota|-l)}(y)|\mathcal{D}_{y}^{\iota}v(y)|^{p}dy\Big)^{1/p}.
\]
The definition of these spaces and their properties in the case of
noninteger $l$ can be found, for instance, in \cite{Ad}.

In fine, we complete this section with the main notations using
throughout this paper.

\noindent$\bullet$ $\Omega_{1}=\Omega_{1}(0),$
$\Omega_{2}=\Omega_{2}(0),$ $\Gamma=\Gamma(0),$
$\Gamma_{1}=\partial\Omega_{1}\backslash \Gamma,$
$\Gamma_{2}=\partial\Omega_{2}\backslash \Gamma$.

 For each fixed
$T>0,$
\begin{align*}
\Omega_{1,T}&=\Omega_{1}\times (0,T),\quad
\Omega_{2,T}=\Omega_{2}\times (0,T),\quad
\Gamma_{T}=\Gamma\times[0,T],\\
 \Gamma_{1,T}&=\Gamma_{1}\times[0,T], \quad \Gamma_{2,T}=\Gamma_{2}\times[0,T].
 \end{align*}

\noindent$\bullet$ $\delta_{0}$ denotes the angle between $\Gamma$
and the positive $y_{2}-$axis at $\mathbf{A}_{0}$

\noindent$\bullet$ $\delta_{1}$ is the angles between $\Gamma$ and
the negative oriented $y_{2}-$axis at $\mathbf{A}_{1}$.

\noindent$\bullet$  For any fixed positive $\varepsilon$ and $j=1,2,$ the symbols $\mathcal{O}_{j\varepsilon}(\mathbf{A}_{0})$ and $\mathcal{O}_{j\varepsilon}(\mathbf{A}_{1})$ stand for the balls with the centers in $\mathbf{A}_{0}$ and $\mathbf{A}_{1}$, respectively, and with the radius $j\varepsilon$.

\noindent$\bullet$ $\omega$ denotes some parameter along $\Gamma$,
e.g. the arc length of $\Gamma$, $\omega\in \mathrm{W}$.

\noindent$\bullet$ $n=n(\omega)$ is the vector of outward normal to
$\Gamma$ directed in $\Omega_{1}$, while
$\mathfrak{l}=\mathfrak{l}(\omega)$ stands for the
$\C^{l+\beta}$-vector field ($l\geq 2$) on $\Gamma$, which is
transversal to $\Gamma$ such that $|\mathfrak{l}|=1$ and
\begin{equation}\label{0.1*}
\mathfrak{l}=
\begin{cases}
(0;1)\qquad\quad\text{in}\quad \Gamma\cap\bar{\mathcal{O}}_{\varepsilon}(\mathbf{A}_{1}),\\
(0;-1)\qquad\text{ in}\quad \Gamma\cap\bar{\mathcal{O}}_{\varepsilon}(\mathbf{A}_{0}),\\
n\qquad\qquad\quad\text{outside of }\quad
\{\Gamma\cap\bar{\mathcal{O}}_{2\varepsilon}(\mathbf{A}_{1})\}\cup\{\Gamma\cap\bar{\mathcal{O}}_{2\varepsilon}(\mathbf{A}_{0})\}.
\end{cases}
\end{equation}


\section{Main results}
\label{s3}
\subsection{Main Assumptions in \eqref{1.1}}\label{s3.1}
We first specify
 the geometrical configuration of the
initial domains $\Omega_{1}$ and $\Omega_{2}$. For simplicity
consideration, we will focus on
 the domains shown on Figure \ref{fig:1a}. It is worth noting
that, our analysis and results (may be with slightly modifications)
hold in the case of $\Omega_{i}$ given on Figure \ref{fig:1b}. First, we state
main hypothesis on the geometry of the domains and on the given
functions.

\begin{description}
\item[h1 (Smoothness of the given boundaries)] For integer $l\geq
3$ and $\beta\in(0,1)$, we assume that
\[
\Gamma_{1},\, \Gamma_{2},\,  \Gamma \quad\text{belong to}\quad
\C^{l+\beta}.\]

\smallskip
\item[h2 (Geometric configuration of $\Gamma_{i}$)] For some given
positive numbers $\a,$ $\a_{1}$ and $\a_{2}$, $0<\a<\a_{1}$, there
hold
\[
\Gamma_{2}=\{(y_{1},y_{2}):\quad y_{1}=0,\quad y_{2}\in(0,\a)\}
\]
and
\[
\{(y_{1},y_{2}):\quad y_{1}=0,\quad y_{2}\in(\a,\a_{1}]\}\cup
\{(y_{1},y_{2}):\quad y_{1}=0,\quad
y_{2}\in[-\a_{2},0)\}\subset\Gamma_{1}.
\]

\smallskip
\item[h3\, (Configuration of $\Gamma$)] We designate   points on
$\Gamma$ by
$$\mathbf{m}(\omega)=\{m_{1}(\omega),m_{2}(\omega)\},\quad \omega\in \mathrm{W},$$
 and, besides, this boundary near corner points
$\mathbf{A}_{0}=(0;0)$ and $\mathbf{A}_{1}=(0,\a)$ is prescribed
with
\[
y_{2}=\begin{cases} \varphi_{0}(y_{1})\qquad\text{if}\quad
(y_{1},y_{2})\in\overline{\Gamma\cap
\mathcal{O}_{2\varepsilon}(\mathbf{A}_{0})},\\
 \varphi_{1}(y_{1})\qquad\text{if}\quad
(y_{1},y_{2})\in\overline{\Gamma\cap
\mathcal{O}_{2\varepsilon}(\mathbf{A}_{1})},
\end{cases}
\]
where $\varepsilon$ is some fixed positive numbers,
$\varepsilon\in(0,\min\{\tfrac{\a_{1}-\a}{6};\tfrac{\a}{6};\tfrac{\a_{2}}{6}\})$,
and $\varphi_{0},$ $\varphi_{1}$ are given smooth functions
satisfying the following relations
\[
\varphi_{0}(0)=0,\quad \varphi'_{0}(0)=\cot\delta_{0}>0\quad
\text{and}\quad \varphi_{1}(0)=\mathfrak{a},\quad
\varphi'_{1}(0)=-\cot\delta_{1}<0
\]
with $\delta_{0},\delta_{1}\in(0,\pi/2)$.

\smallskip

\item[h4 (Condition of the well-posedness to \eqref{1.1})] We require
that
\[
k=\frac{k_{2}}{k_{1}}\in(0,1),\quad\text{and}\quad
\frac{\partial\mathcal{U}_{1,0}}{\partial n},\,
\frac{\partial\mathcal{U}_{2,0}}{\partial
n}<0\quad\text{on}\quad\Gamma.
\]

\smallskip
\item[h5 (Conditions on the angles)] We assume that
  $\delta_{0},$ $\delta_{1}\in(0,\pi/4)$ are rational part of $\pi$. That is,
  for any fixed
  $\mathfrak{q}_{0},\mathfrak{q}_{1},\mathfrak{p}_{0},\mathfrak{p}_{1}\in\mathbb{N}$,
  $\mathfrak{p}_{i}>2\mathfrak{q}_{i}$,
there are representations
  \[
\delta_{0}=\pi\Big(\frac{1}{2}-\frac{\mathfrak{q}_{0}}{\mathfrak{p}_{0}}\Big)\qquad\text{and}\qquad
\delta_{1}=\pi\Big(\frac{1}{2}-\frac{\mathfrak{q}_{1}}{\mathfrak{p}_{1}}\Big),
  \]
  where  $\frac{\mathfrak{q}_{0}}{\mathfrak{p}_{0}},$ $\frac{\mathfrak{q}_{1}}{\mathfrak{p}_{1}}$ are
  irreducible fractions.

\item[h6 (Conditions on the given functions)] We require that
$\mathrm{p}_{1}(y)$ and  $\mathrm{p}_{2}(y)$ are positive functions,
and $\mathrm{p}_{1}\in E^{3+\beta}_{s^{*}}(\bar{\Gamma}_{1})$ and
$\mathrm{p}_{2}\in E^{3+\beta}_{s^{*}}(\bar{\Gamma}_{2})$ with
$$s^{*}\in\Big(\max\Big\{\frac{13}{4},\frac{2\pi}{\pi-2\delta_{1}},\frac{2\pi}{\pi-2\delta_{0}}\Big\},4\Big)
\backslash\Big\{\frac{\pi}{2\delta_{0}},\frac{\pi}{2\delta_{1}}\Big\}.$$
\end{description}


\subsection{Reformulating problem \eqref{1.1}}\label{s3.2}
At this point,  we start with reformulating problem \eqref{1.1} in
more convenient form. Actually, following \cite{H,BV2,V3}, we reduce
the contact Muskat problem (1.1) with moving interface
$\Gamma(\tau)$
 to a nonlinear problem in  the domain with fixed boundaries.
To this end, working within the assumptions above, we introduce
unknown function $\s(\omega,t)$ satisfying relations
\begin{equation}\label{3.0*}
|\s(\omega,t)|<\b_{0}/4\quad\text{and}\quad \s(\omega,0)=0
\end{equation}
with enough small positive number $\b_{0}.$

Then, appealing to \eqref{0.1*}, we define the free boundary
$\Gamma(\tau)$ in the form
\begin{equation}\label{3.1}
\Gamma(\tau)=\{(y,\tau):\,
y(\omega,\tau)=\mathbf{m}(\omega)+\mathfrak{l}\,\s(\omega,\tau),\,
\tau\in[0,T]\}
\end{equation}
for each fixed $T>0$.

\noindent It is worth noting that, the restriction \eqref{3.0*} on
the function $\s(\omega,\tau)$ are related with the the local
classical solvability of \eqref{1.1}.

Next, we describe the mapping which reduces problem \eqref{1.1} with
the moving boundary $\Gamma(\tau)$ to the nonlinear problem in the
domain with fixed boundaries. To this end,  introducing the
nonintersecting $\omega-$lines:
\[
\mathbf{m}(\omega)+\eta\mathfrak{l},\quad|\eta|<2\b_{0},
\]
we define the mapping $(\omega,\eta)\to y=y(\omega,\eta)$ acting via
the rule
\[
y=(y_{1},y_{2})=\mathbf{m}(\omega)+\eta\mathfrak{l}
\]
and being a diffeomorphism  from
$\mathrm{M}=\mathrm{W}\times(-\b_{0},\b_{0})$ onto
$\mathrm{N}=\{y:\, y=\mathbf{m}(\omega)+\eta\mathfrak{l},\,
(\omega,\eta)\in \mathrm{M}\}$.

\noindent Denoting the inverse mapping by
$\Sigma:\mathrm{N}\to\mathrm{M},$ we have
\[
\Sigma:\, y\to(\omega(y),\eta(y)).
\]
After that, setting
\begin{equation}\label{3.0}
\Phi_{\s}(y,\tau)=\eta(y)-\s(\omega(y),\tau),\quad
(y,\tau)\in\mathrm{N}\times[0,T],
\end{equation}
we arrive at the equation
\[
\Phi_{\s}(y,\tau)=0,
\]
which identifies the free boundary $\Gamma(\tau)$.

Let $\chi(\lambda)$ be a smooth cut-off function possessing the
properties
\begin{equation}\label{3.1*}
\chi\in\C_{0}^{\infty}(\R),\quad 0\leq\chi\leq
1,\quad|\chi'|\leq\frac{C_{0}}{\b_{0}}, \quad \chi(\lambda)=
\begin{cases}
1,\qquad|\lambda|<\varepsilon_{1}\b_{0},\\
0,\qquad |\lambda|>2\varepsilon_{1}\b_{0},
\end{cases}
\end{equation}
with some positive $C_{0}$ and
 $\varepsilon_{1}$ such that
\begin{equation}\label{3.10}
\varepsilon_{1}\b_{0}<\frac{\varepsilon}{2}
\end{equation}
with $\varepsilon$ is chosen in (h3).

After that, utilizing the coordinates $(\omega,\eta)$, we define the
diffeomorphism
\[
\mathfrak{e}_{\s}:(x,t)\to(y,\tau)
\]
from $\mathbb{X}_{T}=\R^{2}\times[0,T]$ onto
$\mathbb{Y}_{T}=\R^{2}\times[0,T]$ by setting
\begin{equation}\label{3.2}
\begin{cases}
t=\tau,\\
\omega(y)=\omega(x),\,\eta(y)=\lambda(x)+\chi(\lambda)\s(\omega,\tau)\quad\text{if}\quad
(\omega(x),\lambda(x))\in\mathrm{N},\\
y=x\qquad\qquad\qquad\qquad\qquad\qquad\qquad\qquad\text{otherwise}.
\end{cases}
\end{equation}
Here, $\omega(x)$ and $\lambda(x)$ are the coordinates in
$\mathbb{X}_{T}$, which are similar to the coordinates $\omega(y)$
and $\eta(y)$ in $\mathbb{Y}_{T}$.

It is apparent that $\mathfrak{e}^{-1}_{\s}$ maps $\Omega_{i}(\tau)$
onto $\Omega_{i},$ $i=1,2,$ and $\Gamma(\tau)$ onto $\Gamma$ for
each $\tau\in[0,T]$. Besides, the moving boundary is given by
\[
\mathfrak{e}_{\s}(\{\lambda(x)=0\}).
\]
Bearing in mind \eqref{3.2}, we set
\begin{align}\label{3.00}\notag
\Psi_{1}&=\Psi_{1}(x,t)=\mathrm{p}_{1}(y)\circ\mathfrak{e}_{\s}(x,t),\quad
\Psi_{2}=\Psi_{2}(x,t)=\mathrm{p}_{2}(y)\circ\mathfrak{e}_{\s}(x,t),\\
\mathcal{U}_{1}&=\mathcal{U}_{1}(x,t)=\mathrm{P}_{1}(y,\tau)\circ\mathfrak{e}_{\s}(x,t),\quad
\mathcal{U}_{2}=\mathcal{U}_{2}(x,t)=\mathrm{P}_{2}(y,\tau)\circ\mathfrak{e}_{\s}(x,t),\\\notag
\nabla_{x}&=\Big(\frac{\partial}{\partial
x_{1}},\frac{\partial}{\partial x_{2}}\Big),\quad
\nabla_{\s}=(J_{\s}^{*})^{-1}\nabla_{x},
\end{align}
where $J_{\s}$ is the Jacobi matrix of the mapping
$y=\mathfrak{e}_{\s}(x,t)$.
\begin{remark}\label{r3.0}
In fact, relations  \eqref{3.1*}, \eqref{3.10}, \eqref{3.00} tell
that we need in $\mathrm{p}_{1}$ and $\mathrm{p}_{2}$ defined in
$\{x_{1}=0,\,
x_{2}\in(0,\varepsilon]\cup[\mathfrak{a}-\varepsilon,\mathfrak{a})\}$
 and
 $\{x_{1}=0,\,
x_{2}\in[-\varepsilon,0)\cup(\mathfrak{a},\mathfrak{a}+\varepsilon]\}$,
respectively. The  assumption (h6) on $\mathrm{p}_{1}$ and
$\mathrm{p}_{2}$ allow to extend (in the same classes) them in
arbitrary way. To this end, for example, it is possible to exploit
the even extensions of $\mathrm{p}_{1}$ and $\mathrm{p}_{2}$ in the
segments
$(0,\varepsilon/2]\cup[\mathfrak{a}-\varepsilon/2,\mathfrak{a})$ and
$[-\varepsilon/2,0)\cup(\mathfrak{a},\mathfrak{a}+\varepsilon/2]$,
correspondingly, and then to utilize  the cut-off functions.
\end{remark}

At this point, performing the change of variables \eqref{3.2} in
\eqref{1.1}, we arrive at the relations
\[
\begin{cases}
\nabla_{\s}^{2}\mathcal{U}_{i}(x,t)=0\qquad\qquad\quad\text{in}\quad
\Omega_{i,T},\quad i=1,2,\\
\mathcal{U}_{1}(x,t)-\mathcal{U}_{2}(x,t)=0\quad\text{ on}\quad\Upsilon_{T},\\
\mathcal{U}_{1}(x,t)=\Psi_{1}(x,t)\qquad\quad\text{on}\quad\Gamma_{1,T},\\
\mathcal{U}_{2}(x,t)=\Psi_{2}(x,t)\qquad\quad\text{on}\quad\Gamma_{2,T}.
\end{cases}
\]
In order to rewrite two last conditions on the free boundary in
\eqref{1.1}, we exploit the definition of the function
$\Phi_{\s}(y,\tau)$ (see \eqref{3.0}) to derive the following
equalities satisfying on $\Gamma_{T}$:
\begin{align*}
n_{\tau}&=\frac{\nabla_{y}\Phi_{\s}}{|\nabla_{y}\Phi_{\s}|},\quad
V_{n}=-\frac{\frac{\partial\Phi_{\s}}{\partial
\tau}}{|\nabla_{y}\Phi_{\s}|}=\frac{\frac{\partial
\s(\omega,\tau)}{\partial\tau}}{|\nabla_{y}\Phi_{\s}|},\\
\langle\nabla_{y}\mathrm{P}_{i},\nabla_{y}\Phi_{\s}\rangle&=
\langle\nabla_{\s}\mathcal{U}_{i},\nabla_{\s}\Phi_{\s}\rangle=
\S(\omega,\s,\tfrac{\partial\s}{\partial\omega})\tfrac{\partial\mathcal{U}_{i}}{\partial\lambda}+
\S_{1}(\omega,\s,\tfrac{\partial\s}{\partial\omega})\tfrac{\partial\mathcal{U}_{i}}{\partial\omega},\quad
i=1,2.
\end{align*}
Here, the symbol $\langle\cdot,\cdot\rangle$ stands for the inner
product, and $\S(\omega,\s,\tfrac{\partial\s}{\partial\omega})$,
$\S_{1}(\omega,\s,\tfrac{\partial\s}{\partial\omega})$ are smooth
functions calculated via formulas:
\begin{equation}\label{3.12}
\S(\omega,\s,\tfrac{\partial\s}{\partial\omega})=\langle\nabla_{\s}\lambda,\nabla_{\s}\lambda\rangle,\quad
\S_{1}(\omega,\s,\tfrac{\partial\s}{\partial\omega})=\langle\nabla_{\s}\omega,\nabla_{\s}\lambda\rangle.
\end{equation}
Exploiting these relations, we rewrite the remaining conditions on
the free boundary in the form
\begin{align*}
\frac{\partial\s(\omega,t)}{\partial
t}&=-k_{1}[\S(\omega,\s,\tfrac{\partial\s}{\partial\omega})\tfrac{\partial\mathcal{U}_{1}}{\partial\lambda}+
\S_{1}(\omega,\s,\tfrac{\partial\s}{\partial\omega})\tfrac{\partial\mathcal{U}_{1}}{\partial\omega}]\\
&=-k_{2} [
\S(\omega,\s,\tfrac{\partial\s}{\partial\omega})\tfrac{\partial\mathcal{U}_{2}}{\partial\lambda}+
\S_{1}(\omega,\s,\tfrac{\partial\s}{\partial\omega})\tfrac{\partial\mathcal{U}_{2}}{\partial\omega}].
\end{align*}
Summing up our arguments, we reformulate problem \eqref{1.1} in new
unknown functions and new variables. Namely, we look for unknowns
$\mathcal{U}_{1}:\Omega_{1,T}\to\R$,
$\mathcal{U}_{2}:\Omega_{2,T}\to\R$ and $\s:\Gamma_{T}\to\R$ by the
conditions
\begin{equation}\label{3.3}
\begin{cases}
\nabla_{\s}^{2}\mathcal{U}_{i}(x,t)=0\qquad\qquad\qquad\text{in}\quad\Omega_{i,T},\quad
i=1,2,\\
\mathcal{U}_{1}(x,t)-\mathcal{U}_{2}(x,t)=0\qquad\quad\text{on}\quad\Gamma_{T},\\
\frac{\partial\s(\omega,t)}{\partial
t}=-k_{1}[\S(\omega,\s,\tfrac{\partial\s}{\partial\omega})\tfrac{\partial\mathcal{U}_{1}}{\partial\lambda}+
\S_{1}(\omega,\s,\tfrac{\partial\s}{\partial\omega})\tfrac{\partial\mathcal{U}_{1}}{\partial\omega}]\\
\qquad\quad=-k_{2} [
\S(\omega,\s,\tfrac{\partial\s}{\partial\omega})\tfrac{\partial\mathcal{U}_{2}}{\partial\lambda}+
\S_{1}(\omega,\s,\tfrac{\partial\s}{\partial\omega})\tfrac{\partial\mathcal{U}_{2}}{\partial\omega}]
\quad\text{on}\quad\Gamma_{T},\\
\mathcal{U}_{i}(x,t)=\Psi_{i}(x,t)\qquad\quad\text{on}\quad\Gamma_{i,T},\,
i=1,2,\\
\s(\omega,0)=0\qquad\qquad\qquad\text{on}\quad \Gamma.
\end{cases}
\end{equation}
Relations \eqref{3.1}-\eqref{3.3} suggest that the initial
distribution of the pressure
\[
\mathcal{U}_{1,0}=\mathcal{U}_{1}(x,0)=\mathrm{P}_{1}(y,\tau)\circ\mathfrak{e}_{\s}(x,t)|_{t=0}\quad\text{and}\quad
\mathcal{U}_{2,0}=\mathcal{U}_{2}(x,0)=\mathrm{P}_{2}(y,\tau)\circ\mathfrak{e}_{\s}(x,t)|_{t=0},
\]
solves the transmission problem
\begin{equation}\label{3.4}
\begin{cases}
\Delta_{x}\mathcal{U}_{i,0}(x)=0\qquad\qquad\text{in}\quad\Omega_{i},\quad
i=1,2,\\
\mathcal{U}_{1,0}(x)-\mathcal{U}_{2,0}=0\qquad\text{ on}\quad\Gamma,\\
k_{1}\frac{\partial\mathcal{U}_{1,0}}{\partial
n}=k_{2}\frac{\partial\mathcal{U}_{2,0}}{\partial n} \qquad\text{ on}\quad\Gamma,\\
\mathcal{U}_{1,0}(x)=\mathrm{p}_{1}(x)\qquad\quad\text{on}\quad\Gamma_{1},\\
\mathcal{U}_{2,0}(x)=\mathrm{p}_{2}(x)\qquad\quad\text{on}\quad\Gamma_{2}.
\end{cases}
\end{equation}

\subsection{Main Result}\label{s3.3}
We mention that the solvability of \eqref{3.3} are discussed in the
weighted H\"{o}lder space $E^{2+\beta,\beta,\beta}_{s+2}$. Thus,
before stating the main results, we need in the last assumptions in
this model concerns to the value $s$. To this end, we first need in
the following key magnitudes related with the parameters in
\eqref{3.3}. Denoting
\begin{align*}
q^{*}&=\frac{1+k}{1-k},\qquad\mathfrak{x}_{0}=\frac{\mathfrak{p}_{0}}{2\mathfrak{q}_{0}},\qquad
\mathfrak{x}_{1}=\frac{\mathfrak{p}_{1}}{2\mathfrak{q}_{1}},\\
\alpha_{0}&=\frac{\partial \mathcal{U}_{1,0}}{\partial r_{0}(y)}
\Big(\frac{\partial\mathcal{U}_{1,0}}{\partial
n}\Big)^{-1}\Big|_{y=\mathbf{A}_{0}},\quad\qquad
\alpha_{1}=-\frac{\partial \mathcal{U}_{1,0}}{\partial r_{1}(y)}
\Big(\frac{\partial\mathcal{U}_{1,0}}{\partial
n}\Big)^{-1}\Big|_{y=\mathbf{A}_{1}},
\end{align*}
we define values $\Theta_{i}$ and $Q_{i}$ along with the functions
$S^{+}_{i},$ $S^{-}_{i},$ $i=0,1,$ via formulas
\begin{align*}
Q_{i}&=\frac{\sin\delta_{i}}{1-k}\sqrt{(1-k)^{2}+(2k\alpha_{i}+(1+k)\cot\delta_{i})^{2}},\\
\sin\Theta_{i}&=\frac{1-k}{\sqrt{(1-k)^{2}+(2k\alpha_{i}+(1+k)\cot\delta_{i})^{2}}},\qquad
\cos\Theta_{i}=-\frac{2k\alpha_{i}+(1+k)\cot\delta_{i}}{\sqrt{(1-k)^{2}+(2k\alpha_{i}+(1+k)\cot\delta_{i})^{2}}},\\
S_{i}^{+}(\z)&=\sin(\z-\delta_{i})+Q_{i}\sin(\z\mathfrak{x}_{i}-\Theta_{i}),\qquad\qquad
S_{i}^{-}(\z)=\sin\z-q^{*}\sin\z\mathfrak{x}_{i},\quad\z\in\mathbb{C}.
\\
\end{align*}
In further analysis, in the case of $Q_{i}\neq 1$, we need in the
zeros of these functions in the strip
$Re\,\z\in[0,4\pi\mathfrak{q}_{i}),$ $i=0,1$. It is worth noting
that the properties of $S_{i}^{\pm}$ are studied and described in
detail in \cite[Section 5]{V4}. For readers' convenience, we report
them (rewritten in our notations) in Corollaries
\ref{c4.2}-\ref{c4.4} in Section \ref{s4.1}. In particular, there
are $K_{i}-$zeros and $2\mathfrak{p}_{i}-$zeros of $S_{i}^{+}(\z)$
and $S_{i}^{-}(\z)$, respectively, in the strip
$Re\,\z\in[0,4\pi\mathfrak{q}_{i}),$ $i=0,1,$ which are real and
nondecreasing. Denoting these zeros by

\noindent$\bullet$ $\mathfrak{h}^{+}_{l},$ $l=0,1,...,K_{0}-1,$ and
$\mathfrak{h}^{-}_{j}$, $j=0,1,...,2\mathfrak{p}_{0}-1,$ for
$S_{0}^{+}$ and $S_{0}^{-}$, correspondingly,

\noindent$\bullet$ $\mathfrak{f}^{+}_{l},$
$j=0,1,...,2\mathfrak{p}_{1}-1,$ and $\mathfrak{f}^{-}_{j}$,
$l=0,1,...,K_{1}-1,$
 for  $S_{1}^{+}$ and $S_{1}^{-}$,
correspondingly,

\noindent we select integer quantities
$l_{i}^{*}\in\{1,2,...,K_{i}-1\},$ $i=0,1,$ satisfying the following
inequalities
\begin{align*}
&\frac{\mathfrak{h}^{+}_{l^{*}_{0}-1}}{\pi-2\delta_{0}}<\min\Big\{3,\frac{2\pi}{\pi-2\delta_{0}},\frac{2\pi}{\pi-2\delta_{1}}\Big\}\leq
\frac{\mathfrak{h}^{+}_{l^{*}_{0}}}{\pi-2\delta_{0}},\\
&\frac{\mathfrak{f}^{+}_{l^{*}_{1}-1}}{\pi-2\delta_{1}}<\min\Big\{3,\frac{2\pi}{\pi-2\delta_{0}},\frac{2\pi}{\pi-2\delta_{1}}\Big\}\leq
\frac{\mathfrak{f}^{+}_{l^{*}_{1}}}{\pi-2\delta_{1}}.
\end{align*}
At last, we introduce the quantities
\begin{align}\label{0.1}\notag
\underline{\mathfrak{h}}&=\max\Big\{-\frac{\mathfrak{h}^{-}_{j}}{\pi-2\delta_{0}}+(s^{*}-2)(m-d_{0}),\quad
m\in\mathbb{K}_{0}^{-},\quad j\in\{1,2,...,l_{0}^{*}+2\},\quad
d_{0}\in[0,1]\Big\},\\\notag
\underline{\mathfrak{f}}&=\max\Big\{-\frac{\mathfrak{f}^{-}_{j}}{\pi-2\delta_{1}}+(s^{*}-2)(m-d_{0}),\quad
m\in\mathbb{K}_{1}^{-},\quad j\in\{1,2,...,l_{1}^{*}+2\},\quad
d_{0}\in[0,1]\Big\},\\
\overline{\mathfrak{h}}&=\max\Big\{\frac{\mathfrak{h}^{+}_{l}}{\pi-2\delta_{0}}-(s^{*}-2)(m-1),\quad
m\in\mathbb{K}_{0}^{+},\quad l\in\{0,1,...,l_{0}^{*}-1\}\Big\},\\
\notag
\overline{\mathfrak{f}}&=\max\Big\{\frac{\mathfrak{f}^{+}_{l}}{\pi-2\delta_{1}}-(s^{*}-2)(m-1),\quad
m\in\mathbb{K}_{1}^{+},\quad l\in\{0,1,...,l_{1}^{*}-1\}\Big\},
\end{align}
where   sets
$\mathbb{K}^{+}_{i}\subset\{0,1,...,\mathcal{K}^{+}_{i}\}$ and
$\mathbb{K}^{-}_{i}\subset\{0,1,...,\mathcal{K}^{-}_{i}\}$ with
\begin{align*}
\mathcal{K}^{+}_{0}&=\max\Big\{m\in\mathbb{N}:\,m<1+\frac{\mathfrak{h}_{l_{0}^{*}-1}^{+}-\mathfrak{h}_{1}^{-}}{(s^{*}-2)(\pi-2\delta_{0})}\Big\},\quad
\mathcal{K}^{+}_{1}=\max\Big\{m\in\mathbb{N}:\,m<1+\frac{\mathfrak{f}_{l_{1}^{*}-1}^{+}-\mathfrak{f}_{1}^{-}}{(s^{*}-2)(\pi-2\delta_{1})}\Big\},\\
\mathcal{K}^{-}_{0}&=\max\Big\{m\in\mathbb{N}:\,m<1+\frac{\mathfrak{h}_{l_{0}^{*}}^{+}+\mathfrak{h}_{l_{0}^{*}+2}^{-}}{(s^{*}-2)(\pi-2\delta_{0})}\Big\},\quad
\mathcal{K}^{-}_{1}=\max\Big\{m\in\mathbb{N}:\,m<1+\frac{\mathfrak{f}_{l_{1}^{*}}^{+}+\mathfrak{f}_{l_{1}^{*}+2}^{-}}{(s^{*}-2)(\pi-2\delta_{1})}\Big\},
\end{align*}
are built with the same way as the construction of the sets
$\mathbb{M}^{\pm}_{i}$ (see \eqref{4.0*}) described in Section
\ref{s4.1}.

It is worth noting that the existence of values $l_{i}^{*}$,
$\underline{\mathfrak{h}},$ $\underline{\mathfrak{f}},$
$\overline{\mathfrak{h}},$ $\overline{\mathfrak{f}}$ are provided by
Corollaries \ref{c4.2}--\ref{c4.4}.

\noindent Now we  complete  stipulating    assumptions in the model
\eqref{3.3} by means of the requirement on the weight $s$.
\begin{description}
\item[h7 (Condition of the weight $s$)] We require that
\[
s+2\in\begin{cases}
(\max\{2,\mathfrak{h}^{*},\mathfrak{f}^{*}\},\min\{3,\tfrac{2\pi}{\pi-2\delta_{0}},\tfrac{2\pi}{\pi-2\delta_{1}}\})\backslash
\{\tfrac{\pi}{2\delta_{0}},\tfrac{\pi}{2\delta_{1}}\}\qquad
\text{if}\quad Q_{0},Q_{1}\neq 1,\\
\\
(\max\{2,\mathfrak{f}^{*}\},\min\{3,\tfrac{2\pi}{\pi-2\delta_{0}},\tfrac{2\pi}{\pi-2\delta_{1}}\})\backslash
\{\tfrac{\pi}{2\delta_{0}},\tfrac{\pi}{2\delta_{1}}\}\qquad\quad
\text{ if}\quad Q_{0}=1,Q_{1}\neq 1,\\
\\
(\max\{2,\mathfrak{h}^{*}\},\min\{3,\tfrac{2\pi}{\pi-2\delta_{0}},\tfrac{2\pi}{\pi-2\delta_{1}}\})\backslash
\{\tfrac{\pi}{2\delta_{0}},\tfrac{\pi}{2\delta_{1}}\}\qquad\quad
\text{if}\quad Q_{0}\neq 1,Q_{1}= 1,\\
\\
(2,\min\{3,\tfrac{2\pi}{\pi-2\delta_{0}},\tfrac{2\pi}{\pi-2\delta_{1}}\})\backslash
\{\tfrac{\pi}{2\delta_{0}},\tfrac{\pi}{2\delta_{1}}\}\qquad\qquad\qquad\quad
\text{if}\quad Q_{0},Q_{1}= 1,\\
\end{cases}
\]
with
\[
\mathfrak{h}^{*}=\max\{\underline{\mathfrak{h}},\overline{\mathfrak{h}}\}\qquad\text{and}\qquad
\mathfrak{f}^{*}=\max\{\underline{\mathfrak{f}},\overline{\mathfrak{f}}\}.
\]
\end{description}


Now we are ready to state the main results of this paper.
\begin{theorem}\label{t3.1}
Under assumptions (h1)-(h7), problem \eqref{3.3} admits a unique
local classical solution $(\mathcal{U}_{1},\mathcal{U}_{2},\s)$ in
some small time interval $[0,T^{*}]$. Besides, this solution has the
regularity
\[
\mathcal{U}_{1}\in
E_{s+2}^{2+\beta,\beta,\beta}(\bar{\Omega}_{1,T^{*}}),\quad
\mathcal{U}_{2}\in
E_{s+2}^{2+\beta,\beta,\beta}(\bar{\Omega}_{2,T^{*}}),\quad
\s\in\mathcal{E}^{2+\beta,\beta,\beta}_{s+2,s^{*}-1}(\Gamma_{T^{*}}),
\]
the free boundary $\Gamma(\tau)$ is defined via \eqref{3.1} for each
$\tau\in[0,T^{*}]$, and the initial distribution of the pressure
$(\mathcal{U}_{1,0},\mathcal{U}_{2,0})$ is a unique classical
solution of \eqref{3.4}.
\end{theorem}
\begin{remark}\label{r3.2}
Recalling that in the Muskat model
\[
k_{1}=\frac{\bar{k}}{\mu_{1}}\quad\text{and}\quad
k_{2}=\frac{\bar{k}}{\mu_{2}}
\]
with positive constants $\bar{k},$ $\mu_{i}$ having the sense of the
permeability of the porous media and the viscosity of the fluids in
$\Omega_{i}(\tau),$ respectively,
 condition (h4) tells that a more viscous fluid is
displaced by a less viscous fluid. Thus, the contact Muskat problem
\eqref{1.1} (or its reformulated form \eqref{3.3}) without surface
tension is well-posed (see e.g. \cite{Am}).
\end{remark}
\noindent Appealing to the definition of the weighted H\"{o}lder 
spaces and bearing in mind of Theorem \ref{t3.1}, we arrive at the following behavior
\[
|\mathcal{U}_{1}|,|\mathcal{U}_{2}|\lesssim\begin{cases}
r_{0}^{s+2}(x)\quad\text{if}\quad x\to\mathbf{A}_{0},\\
r_{1}^{s+2}(x)\quad\text{if}\quad x\to\mathbf{A}_{1},
\end{cases}
\quad
\text{and}\quad 
|\mathfrak{s}_{t}|,|\mathfrak{s}|\lesssim
\begin{cases}
r_{0}^{s+1}(x)\quad\text{if}\quad x\to\mathbf{A}_{0},\\
r_{1}^{s+1}(x)\quad\text{if}\quad x\to\mathbf{A}_{1},
\end{cases}
\]
which in turn provides
the following property so-called as "waiting time"
phenomenon.
\begin{corollary}\label{c3.1}
If assumptions of Theorem \ref{t3.1} hold, then for each
$t\in[0,T^{*}]$, the corner points $\mathbf{A}_{0}$ and
$\mathbf{A}_{1}$ do not move and the geometry of the initial shape
of the moving boundary near these points is preserved. Other words,
Theorem \ref{t3.1} provides the existence of "waiting time" in the
contact Muskat problem \eqref{1.1} in the case of zero surface
tension.
\end{corollary}

The proof of Theorem \ref{t3.1} is carried out in Sections
\ref{s4}--\ref{s5}. In Section \ref{s6}, we discuss how restriction
(h5) on the form of $\delta_{0}$ and $\delta_{1}$ can be removed,
that we extend the results of  Theorem \ref{t3.1}  to the case of
arbitrary $\delta_{0},\delta_{1}\in(0,\pi/4)$.

\section{A Transmission Problem with a Dynamic Boundary Condition in $(G_{1}\cup G_{2})_{T}$}
\label{s4}

\noindent In this section, we discuss the one-to-one classical
solvability of a nonclassical transmission problem to Poisson
equations in $(G_{1}\cup G_{2})_{T}$ with $\delta$ being a rational
part of $\pi$ (see \eqref{2.1}) with a transmission condition
containing the time derivative of the unknown function. We recall
that this kind of boundary conditions is often called a dynamic
boundary condition. This problem is a key tool in Section \ref{s5.3}
where we   discuss the solvability of  the linear problem
corresponding to nonlinear \eqref{3.3}.

 For any fixed
$\delta\in[0,\pi/2]$ having the form
\begin{equation}\label{4.1}
\delta=\frac{q\pi}{p}\quad\text{with}\quad
\frac{q}{p}\quad\text{being  irreducible fraction,}\quad
q,p\in\mathbb{N},\quad p>2q,
\end{equation}
we focus on the nonclassical transmission problem in the unknowns
$u_{1}=u_{1}(y,t):G_{1,T}\to\R$, $u_{2}=u_{2}(y,t):G_{2,T}\to\R$,
$\varrho=\varrho(y,t):g_{T}\to\R,$
\begin{equation}\label{4.2}
\begin{cases}
\Delta_{y} u_{i}=f_{0,i}(y,t)\qquad\text{in}\quad G_{i,T},\quad
i=1,2,\\
u_{1}-u_{2}=a_{0}r_{0}^{s^{*}-1}\varrho\qquad\text{on}\quad
g_{T},\\
a_{1}\frac{\partial\varrho}{\partial t}-\Big(\frac{\partial
u_{1}}{\partial n}-\frac{\partial u_{2}}{\partial
n}\Big)-a_{2}\Big(\frac{\partial u_{1}}{\partial
r_{0}}-\frac{\partial u_{2}}{\partial r_{0}}\Big)=f_{1}(y,t)\qquad
\text{on}\quad g_{T},\\
\frac{\partial u_{1}}{\partial n}-k\frac{\partial u_{2}}{\partial
n}-ka_{3}\Big(\frac{\partial u_{1}}{\partial r_{0}}-\frac{\partial
u_{2}}{\partial r_{0}}\Big)=f_{2}(y,t)\qquad
\text{on}\quad g_{T},\\
u_{1}=0\quad\text{on}\quad \partial G_{1,T}\backslash
g_{T},\\
u_{2}=0\quad\text{on}\quad \partial G_{2,T}\backslash
g_{T},\\
u_{i}(y,0)=0\qquad\text{in}\quad \bar{G}_{i},\, i=1,2.
\end{cases}
\end{equation}
Here $n$ denotes the outward normal to $g$ directed in $G_{1}$,
$f_{0,i},$ $i=1,2,$ and $f_{j}$ are prescribed below functions,
while $a_{l},$ $l=0,1,2,3,$ $k,$ $s^{*}$ are given real numbers
satisfying the inequalities:
\begin{equation}\label{4.3}
a_{0}<0,\quad a_{1},a_{2}>0,\quad 0<k<1,\quad  s^{*}\geq 3.
\end{equation}

It is worth noting that in this section we analyze model \eqref{4.2}
in the more general assumptions on the data than it is required in
Section \ref{s5.3}. Moreover, in this section (in contrast with the
remaining  part of the paper), we discuss global classical
solvability of \eqref{4.2}, i.e. solvability for any but fixed $T$.

\subsection{Assumptions on the given data in the model}\label{s4.1}

To state the hypothesis in \eqref{4.1}, we need in the following
values: $\theta_{1},\theta_{2},$ $q_{2},q_{1}$ defined as
\begin{align*}
q_{1}&=\frac{p}{2q},\quad
q_{2}=\sqrt{\frac{(1-k)^{2}+(a_{2}(1+k)+2ka_{3})^{2}}{(1-k)^{2}(1+a_{2}^{2})}},\\
\sin\theta_{1}&=\frac{1}{\sqrt{1+a_{2}^{2}}},\qquad
\cos\theta_{1}=\frac{a_{2}}{\sqrt{1+a_{2}^{2}}},\\
\sin\theta_{2}&=\frac{1-k}{\sqrt{(1-k)^{2}+(2ka_{3}+(1+k)a_{2})^{2}}},\quad
\cos\theta_{2}=-\frac{2ka_{3}+(1+k)a_{2}}{\sqrt{(1-k)^{2}+(2ka_{3}+(1+k)a_{2})^{2}}}.
\end{align*}
It is worth noting that in the particular case of
$a_{2}=\cot\delta_{i}$ and $a_{3}=\alpha_{i}$, the introduced
quantities boil down with $\mathfrak{x}_{i},Q_{i},\Theta_{i}$.

The straightforward technical calculations provide the following
properties of these values.
\begin{corollary}\label{c4.1}
Let $k\in(0,1)$ and \eqref{4.1}, \eqref{4.3} hold. Then $q^{*}>1,$
$\theta_{1}\in(0,\pi/2),$ $\theta_{2}\in(0,\pi)$ and, besides,
\begin{align*}
\theta_{2}&\in(0,\pi/2]\quad\text{if}\quad 2ka_{3}+(1+k)a_{2}\leq
0,\\
\theta_{2}&\in(\pi/2,\pi)\quad\text{if}\quad 2ka_{3}+(1+k)a_{2}>0.
\end{align*}
Moreover,

\noindent$\bullet$ $q_{2}>1$ if the one of the following conditions
holds:

\noindent(i) either $a_{3}\geq 0$,

\noindent(ii) or $a_{2}>-a_{3}>0$,

\noindent(iii) or $0<a_{2}<-k a_{3}$;

\noindent$\bullet$ $q_{2}=1$ if either $a_{2}=-ka_{3}>0$ or
$a_{2}=-a_{3}>0$;

\noindent$\bullet$ $q_{2}<1$ if $a_{3}<0$ and $
0<-ka_{3}<a_{2}<-a_{3}. $
\end{corollary}
Further, keeping in mind these quantities, we introduce the
functions:
\begin{equation}\label{4.0}
S^{+}=S^{+}(\mathfrak{z})=\sin(\z-\theta_{1})+q_{2}\sin(q_{1}\z-\theta_{2})\quad\text{and}\quad
S^{-}=S^{-}(\mathfrak{z})=\sin\z-q^{*}\sin q_{1}\z
\end{equation}
for $\z\in\mathbb{C}$.

The properties of these functions are the key technical tool in this
section, which are established in \cite[Section 5]{V4} and reported
here below in very particular form tailored for our goals.
\begin{corollary}\label{c4.2}
Let $k\in(0,1)$ and assumptions \eqref{4.1}, \eqref{4.3} hold. Then
 $S^{-}$ is the entire odd periodic function with the period
 $\mathbb{T}_{q}=4q\pi$. Besides,

 \noindent(i) all zeros $\z_{i}^{-}$ of $S^{-}$ in the strip $Re\z\in[0,4q
 \pi)$ are real and strictly increasing, and their number equals to
 $2p$, that is
 \[
0\leq\z_{0}^{-}<\z_{1}^{-}<...<\z_{2p-1}^{-}<4q\pi,\quad
\z_{i}^{-}\in\Big[\frac{\pi(2i-1)}{2q_{1}},\frac{\pi(2i+1)}{2q_{1}}\Big],\quad
i=0,1,...,2p-1;
 \]

\noindent(ii) all zeros of $S^{-}$ in $\mathbb{C}$ are real and
computed via formulas:
\[
\z_{i,n}^{-}=\z_{i}^{-}+n\mathbb{T}_{q},\quad
\z_{i,-n}^{-}=-\z_{i,n}^{-},\quad n\in\mathbb{N},\quad
i=0,1,...,2p-1;
\]

\noindent(iii) the following factorization holds in $\mathbb{C},$
\[
S^{-}(\z)=(1-q_{1}q^{*})\z\prod_{i=1}^{2p-1}\frac{(\z_{i}^{-}-\z)(\z_{i}^{-}+\z)}{(\z_{i}^{-})^{2}}
\prod_{n=1}^{\infty}\prod_{i=0}^{2p-1}\frac{(\z_{i,n}^{-}-\z)(\z_{i,n}^{-}+\z)}{(\z_{i,n}^{-})^{2}};
\]
\end{corollary}
\begin{corollary}\label{c4.3}
Let $q_{2}\neq 1$ and assumptions of Corollary \ref{c4.2} hold. Then
$S^{+}$ is the entire periodic function with the period
$\mathbb{T}_{q}=4q\pi$, which has only real zeros in $\mathbb{C}$
computed via formulas
\[
\z_{i,n}^{+}=\z^{+}_{i}+n\mathbb{T}_{q},\quad
\z_{i,-n}^{+}=\z^{+}_{i}-n\mathbb{T}_{q},\quad n\in\mathbb{N},\quad
i=0,1,...,K^{+}-1,
\]
where $\z_{i}^{+}$ are the zeros  in the strip $Re\,\z\in[0,4q\pi)$,
and $K^{+}=2p$ if $q_{2}>1$, while
\[
K^{+}\in
\begin{cases}
[4q,2p]\quad\text{if}\quad q_{2}<1,\, q_{2}\neq 1/q_{1},\\
(4q,6p)\quad\text{if}\quad q_{2}=1/q_{1}.
\end{cases}
\]
Moreover,

\noindent$\bullet$ if  $q_{2}\neq 1/q_{1}$, then $\z_{i}^{+}$,
$i=0,1,...,K^{+}-1,$ are simple positive and strictly increasing
such that
\[
\z_{i}^{+}\in
\begin{cases}
\Big[\frac{\pi(2i-1)+2\theta_{2}}{2q_{1}},\frac{\pi(2i+1)+2\theta_{2}}{2q_{1}}\Big]\quad\text{if}\quad
q_{2}>1,\\
\,\\
 \bigcup_{l=1_{1}}^{l=l_{2}}\mathfrak{G}_{l}\qquad\qquad\qquad\text{if}\quad
q_{2}<1\quad\text{and}\quad q_{2}\neq 1/q_{1}.
\end{cases}
\]
with
\begin{align*}
l_{1}&=0\qquad\text{and}\quad l_{2}=2q\qquad\quad\text{if}\quad
\theta_{1}<\arcsin q_{2}<\pi+\theta_{1},\\
l_{1}&=0\qquad\text{and}\quad l_{2}=2q-1\quad\text{if}\quad
\theta_{1}-\pi<\arcsin q_{2}<\theta_{1},\\
l_{1}&=-1\quad\text{ and}\quad l_{2}=2q\qquad\text{ if}\quad
\pi-\theta_{1}<\arcsin q_{2}<\pi,
\end{align*}
and
\[
\mathfrak{G}_{l}=[\theta_{1}-\arcsin q_{2}+2\pi l,\theta_{1}+\arcsin
q_{2}+2\pi l]\cup[\theta_{1}+\pi-\arcsin q_{2}+2\pi
l,\theta_{1}+\pi+\arcsin q_{2}+2\pi l];
\]

\noindent$\bullet$ if $q_{2}=1/q_{1}$ then all $\z_{i}^{+}$ are
positive nondecreasing and some of them have the third order;
$\z_{i}^{+}\in\cup_{l=l_{1}}^{l_{2}}\mathfrak{G}_{l},$
$i=0,1,...,K^{+}-1$;

\noindent$\bullet$ there is the following factorization in
$\mathbb{C},$
$$S^{+}(\z)=-[\sin\theta_{1}+q_{2}\sin\theta_{2}]\prod_{i=0}^{K^{+}-1}\frac{[\z_{i}^{+}-z]}{\z_{i}^{+}}\prod_{n=1}^{\infty}
\prod_{i=0}^{K^{+}-1}\frac{[\z_{i,n}^{+}-z][-\z_{i,-n}^{+}+z]}{(-\z_{i,-n}^{+})\z_{i,n}^{+}}.$$
\end{corollary}

Coming to $S^{+}$ in the case of $q_{2}=1$ and collecting the direct
calculations with outcomes in \cite[Table 3]{V4}, we claim.
\begin{corollary}\label{c4.4}
Let $q_{2}=1$ and assumptions of Corollary \ref{c4.2} hold. Then
$S^{+}$ admits the factorization
\[
S^{+}(\z)=(q_{1}+1)\frac{\sin\theta_{1}}{\theta_{1}}\Big(\z-\frac{2\theta_{1}}{q_{1}+1}\Big)\prod_{n=1}^{+\infty}\frac{
\Big[2\frac{n\pi+\theta_{1}}{q_{1}+1}-\z\Big]\Big[2\frac{n\pi-\theta_{1}}{q_{1}+1}+\z\Big]}
{4\Big(\frac{n\pi+\theta_{1}}{q_{1}+1}\Big)\Big(\frac{n\pi-\theta_{1}}{q_{1}+1}\Big)},
\quad \z\in\mathbb{C}. \]
\end{corollary}

At this point, appealing to Corollaries \ref{c4.2}-\ref{c4.3}, we
define two integer numbers $i_{1}^{*},i_{2}^{*}\in[1,K^{+}-1]$ such
that
\begin{equation}\label{4.4}
\frac{\z^{+}_{i_{1}^{*}-1}}{2\delta}<3\leq\frac{\z^{+}_{i_{1}^{*}}}{2\delta}\quad\text{and}\quad
\frac{\z^{+}_{i_{2}^{*}-1}}{2\delta}<\min\Big\{3,\frac{\pi}{\delta}\Big\}\leq\frac{\z^{+}_{i_{2}^{*}}}{2\delta}.
\end{equation}

In light the introduced values, we construct the bounded integer
sets for $j=1,2:$
\begin{equation}\label{4.0*}
\mathbb{M}^{-}_{j}\subset\{0,1,...,M_{j}^{-}\}\quad\text{and}\quad
\mathbb{M}^{+}_{j}\subset\{0,1,...,M_{j}^{+}\}
\end{equation}
with
\[
M_{j}^{-}=\max\Big\{m\in\mathbb{N}:\quad
m<1+\frac{\z^{+}_{i^{*}_{j}}+\z^{-}_{i^{*}_{j}+2}}{2\delta(s^{*}-2)}\Big\}
\quad\text{and}\quad
 M_{j}^{+}=\max\Big\{m\in\mathbb{N}:\quad
m<1+\frac{\z^{+}_{i^{*}_{j}-1}-\z^{-}_{1}}{2\delta(s^{*}-2)}\Big\}
\]
via utilizing the following rule:

\noindent$\bullet$ the nonnegative integer $m$ belongs to the set
$\mathbb{M}^{-}_{j}$ if $m\leq M_{j}^{-}$ and the inequality holds
\[
-\frac{\z^{-}_{i}}{2\delta}+(s^{*}-2)(m-d_{0})<
\begin{cases}
3\qquad\qquad\quad\text{ if}\quad j=1,\\
\min\{3,\pi/\delta\}\quad\text{if}\quad j=2
\end{cases}
\]
for some $i\in\{1,2,...,i^{*}_{j}+2\}$ and any $d_{0}\in[0,1]$;

\noindent$\bullet$ the nonnegative integer $m$ belongs to the set
$\mathbb{M}^{+}_{j}$ if $m\leq M_{j}^{+}$ and the inequality holds
\[
\frac{\z^{+}_{i}}{2\delta}-(s^{*}-2)(m-1)<
\begin{cases}
3\qquad\qquad\quad\text{ if}\quad j=1,\\
\min\{3,\pi/\delta\}\quad\text{if}\quad j=2
\end{cases}
\]
for some $i\in\{0,1,...,i^{*}_{j}-1\}$.

In fine, we  introduce the  magnitudes  playing the key role in the
assumption on the weight $s$. To this end, we set for $j=1,2,$
\begin{align*}
\underline{\z_{j}}&=\max\Big\{\frac{-\z^{-}_{i}}{2\delta}+(s^{*}-2)(m-d_{0})\quad\text{for
each}\quad
m\in\mathbb{M}_{j}^{-},\, i\in\{1,2,...,i^{*}_{j}+2\},\, d_{0}\in[0,1]\Big\},\\
\overline{\z_{j}}&=\max\Big\{\frac{\z^{+}_{i}}{2\delta}-(s^{*}-2)(m-1)\quad\text{for
each}\quad m\in\mathbb{M}_{j}^{+},\,
i\in\{0,1,...,i^{*}_{j}-1\}\Big\},
\end{align*}
and define
\begin{equation}\label{4.5}
\z_{j}^{*}=\max\{\underline{\z_{j}},\overline{\z_{j}}\}.
\end{equation}

Now we are ready to state the remaining assumptions in the model.
\begin{description}
\item[h8 (Condition on the function $f_{1}$)] For $s+2\in(\max\{2,\z^{*}_{1}\},3)$, we require that
\[
f_{1}\in\underset{0}{E}\,_{s+1}^{1+\beta,\beta,\beta}(g^{+}_{T})
\]
and
\[
f_{1}\equiv 0\quad\text{if}\quad |y|>R_{0}
\]
for some positive $R_{0}$.

\item[h9 (Condition on the right-hand sides)]We assume that $\delta\in(\pi/4,\pi/2),$
$$s+2\in(\max\{2,\z^{*}_{2}\},\min\{3,\pi/\delta\})\backslash\{\pi/(\pi-2\delta)\},$$
and
\[
f_{0,1}\in\underset{0}{E}\,_{s}^{\beta,\beta,\beta}(\bar{G}_{1,T}),\quad
f_{0,2}\in\underset{0}{E}\,_{s}^{\beta,\beta,\beta}(\bar{G}_{2,T}),\quad
f_{1}, f_{2}\in\underset{0}{E}\,_{s+1}^{1+\beta,\beta,\beta}(g_{T}).
\]
Besides,
\[
f_{0,1}, f_{0,2},f_{1},f_{2}\equiv 0\quad\text{if}\quad |y|>R_{0}
\]
for some positive $R_{0}$.
\end{description}

\subsection{Main Results to \eqref{4.2}}\label{s4.2}

The main results of Section \ref{s4} are follows.
\begin{theorem}\label{t4.1}
Let  $T>0$ be arbitrary but fixed and let
\eqref{4.1},\eqref{4.3},\eqref{4.5} hold. Under assumption (h8),
problem \eqref{4.2} with $f_{0,1},f_{0,2},f_{2}\equiv 0$ admits a
unique classical solution $(u_{1},u_{2},\varrho)$ in $(G_{1}\cup
G_{2})_{T}$ satisfying the regularity
\begin{equation}\label{4.6}
u_{1}\in\underset{0}{E}\,_{s+2}^{2+\beta,\beta,\beta}(\bar{G}_{1,T}),\quad
u_{2}\in\underset{0}{E}\,_{s+2}^{2+\beta,\beta,\beta}(\bar{G}_{2,T}),\quad
\varrho\in\underset{0}{\mathcal{E}}\,_{s+2,s^{*}-1}^{2+\beta,\beta,\beta}(g_{T}).
\end{equation}
Besides, the estimates hold
\begin{align*}
&\|u_{1}\|_{E_{s+2}^{2+\beta,\beta,\beta}(\bar{G}_{1,T})}+
\|u_{2}\|_{E_{s+2}^{2+\beta,\beta,\beta}(\bar{G}_{2,T})}+
\|\varrho\|_{\mathcal{E}_{s+2,s^{*}-1}^{2+\beta,\beta,\beta}(g_{T})}\leq
C\|f_{1}\|_{E_{s+1}^{1+\beta,\beta,\beta}(g_{T})},\\
&\|u_{1}\|_{E_{s+2}^{1+\beta,\beta,\beta}(\bar{G}_{1,T})}+
\|u_{2}\|_{E_{s+2}^{1+\beta,\beta,\beta}(\bar{G}_{2,T})}+
\|r_{0}^{s^{*}-1}\varrho\|_{E_{s+2}^{1+\beta,\beta,\beta}(g_{T})}\leq
CT^{\beta^{*}-\beta}\|f_{1}\|_{E_{s+1}^{1+\beta,\beta,\beta}(g_{T})}
\end{align*}
with any $\beta^{*}\in(\beta,1)$.
\end{theorem}

The following claim concerns with the solvability of \eqref{4.2} in
the case of the inhomogeneous right-hand sides.
\begin{theorem}\label{t4.2}
Let $T>0$ be arbitrary but fixed,   $q_{2}\neq 1$,
$\delta\in(\pi/4,\pi/2)$ and let \eqref{4.1}, \eqref{4.3},
\eqref{4.5} hold. Under assumption (h9), problem \eqref{4.2} admits
a unique classical solution $(u_{1},u_{2},\varrho)$ in $(G_{1}\cup
G_{2})_{T}$ satisfying the regularity \eqref{4.6}. Besides, the
bound holds
\begin{align*}
&\|u_{1}\|_{E_{s+2}^{2+\beta,\beta,\beta}(\bar{G}_{1,T})}+
\|u_{2}\|_{E_{s+2}^{2+\beta,\beta,\beta}(\bar{G}_{2,T})}+
\|\varrho\|_{\mathcal{E}_{s+2,s^{*}-1}^{2+\beta,\beta,\beta}(g_{T})}\\
&\leq
C[\|f_{1}\|_{E_{s+1}^{1+\beta,\beta,\beta}(g_{T})}+\|f_{2}\|_{E_{s+1}^{1+\beta,\beta,\beta}(g_{T})}
+ \|f_{0,1}\|_{E_{s}^{\beta,\beta,\beta}(\bar{G}_{1,T})}+
\|f_{0,2}\|_{E_{s}^{\beta,\beta,\beta}(\bar{G}_{2,T})}].
\end{align*}
\end{theorem}

Coming to the case of $q_{2}=1$, we claim
\begin{theorem}\label{t4.3}
Let $\delta\in(\pi/4,\pi/2)$ and
$s+2\in\Big(2,\min\Big\{3,\frac{\pi}{\delta},
2(s^{*}-2)+\frac{\z_{1}^{-}}{2\delta}\Big\}\Big)$. We assume that
the right-hand sides in \eqref{4.2} meet the requirements of Theorem
\ref{t4.2}. Then problem \eqref{4.2} admits a unique classical
solution $(u_{1},u_{2},\varrho)$ in $(G_{1}\cup G_{2})_{T}$
satisfying the regularity \eqref{4.6}. Besides, the bound of Theorem
\ref{t4.2} holds.
\end{theorem}

The proofs of these theorems given in the next subsections suggest
 the following claim.
\begin{corollary}\label{c4.0}
The results of Theorems \ref{t4.1}-\ref{t4.3} hold if the quantity
$k$ in the boundary condition on $g_{T}$ in \eqref{4.2} is replaced
by $k\,\epsilon$ with any $\epsilon\in(0,1)$.
\end{corollary}

It is worth noting that Theorems \ref{t4.1}-\ref{t4.3} describe the
following behavior of the solution of \eqref{4.2} at the singular
point $(0,0)$:
\[
|u_{1}|, |u_{2}|\lesssim r_{0}^{2+s}(y),\qquad
|\varrho_{t}|,|\varrho|\lesssim r_{0}^{1+s}(y)\quad\text{as}\quad
r_{0}(y)\to 0.
\]

The proof of Theorem \ref{t4.3} is similar to arguments leading to
Theorems \ref{t4.1} and \ref{t4.2}. Thus, we verify (here) only the
 first two theorems.


\subsection{Proof of Theorem \ref{t4.1}}\label{s4.3}
Thanks to the first and the second transmission condition on $g_{T}$
and taking into account that we are working in the framework of
classical solvability,  we transform system \eqref{4.1}
 to the transmission problem in unknowns $u_{1}$ and $u_{2}$:
\begin{equation}\label{4.7}
\begin{cases}
\Delta_{y} u_{i}=0\qquad\qquad\quad\text{in}\quad G_{i,T},\quad
i=1,2,\\
u_{i}(y,0)=0\qquad\qquad\text{in}\quad
\bar{G}_{i},\, i=1,2,\\
\frac{a_{1}}{|a_{0}|}r_{0}^{1-s^{*}}\Big[\frac{\partial
u_{1}}{\partial t}-\frac{\partial u_{2}}{\partial
t}\Big]+\Big(\frac{\partial u_{1}}{\partial n}-\frac{\partial
u_{2}}{\partial n}\Big)+a_{2}\Big(\frac{\partial u_{1}}{\partial
r_{0}}-\frac{\partial u_{2}}{\partial r_{0}}\Big)=-f_{1}(y,t)\qquad
\text{on}\quad g_{T},\\
\frac{\partial u_{1}}{\partial n}-k\frac{\partial u_{2}}{\partial
n}-ka_{3}\Big(\frac{\partial u_{1}}{\partial r_{0}}-\frac{\partial
u_{2}}{\partial r_{0}}\Big)=0\qquad
\text{on}\quad g_{T},\\
u_{1}=0\quad\text{on}\quad (\partial G_{1}\backslash
g)_{T},\\
u_{2}=0\quad\text{on}\quad (\partial G_{2}\backslash g)_{T}.
\end{cases}
\end{equation}
Here, we set $r_{0}(y)=r_{0}$.

 Since we search a classical solution to \eqref{4.2}, we are
left to prove Theorem \ref{t4.1} in the case of problem \eqref{4.7}.
To this end, we exploit the following strategy. In the first step,
utilizing Fourier and Laplace transformations, we build the integral
representation of the solution. Then, using this explicit form of
$u_{i},$ we verify the corresponding bounds which also allow to
derive the uniqueness of the obtained solutions. In further
analysis, for simplicity, we set $\frac{a_{1}}{|a_{0}|}=1$.

\noindent\textit{Stage 1: Explicit form of $u_{i}$.} Performing the
change of variables \eqref{2.2} and introducing new unknown
functions
\[
v_{1}(x,t)=e^{-(s+2)x_{1}}u_{1}(x,t),\qquad
v_{2}(x,t)=e^{-(s+2)x_{1}}u_{2}(x,t),
\]
we rewrite \eqref{4.7} in more suitable form
\begin{equation}\label{4.8}
\begin{cases}
\Delta v_1+2(s+2)\frac{\partial v_1}{\partial
x_1}+(s+2)^{2}v_1=0\qquad \text{in}\quad B_{1,T},
\\
\Delta v_2+2(s+2)\frac{\partial v_2}{\partial
x_1}+(s+2)^{2}v_2=0\qquad\text{in}\quad B_{2,T},\\
e^{(2-s^{*})x_1}\frac{\partial(v_1-v_2)}{\partial t} -
\frac{\partial(v_1-v_2)}{\partial x_2}+a_2\Big[
\frac{\partial (v_1-v_2)}{\partial x_1}+(s+2)(v_1-v_2)\Big]=-e^{-(s+1)x_1}f_{1}(x,t)\quad\text{on}\quad b_{T},\\
\frac{\partial v_1}{\partial x_2}-k\frac{\partial v_2}{\partial
x_2}+ka_3
\Big[\frac{\partial (v_1-v_2)}{\partial x_1}+(s+2)(v_1-v_2)\Big]=0\qquad\qquad\qquad\qquad\text{on}\quad b_{T},\\
v_1(x_1,-\pi/2,t)=0\qquad\text{and}\qquad
v_2(x_1,\pi/2,t)=0\quad\qquad\qquad \text{for}\quad x_1\in\R,\,
t\in[0,T],\\
v_1(x_1,x_2,0)=0\qquad\quad\text{ in}\quad \bar{B}_{1},\quad \qquad
v_2(x_1,x_2,0)=0\qquad\text{ in}\quad \bar{B}_{2}.
\end{cases}
\end{equation}
Denote  by $\tilde{v}(\zeta,x_{2},t)$ the Fourier transform of
$v(x_{1},x_{2},t)$ and by $\hat{v}(x_{1},x_{2},\mu)$ the Laplace
transform of this function, and then we use  $*$ instead of
$\widehat{\widetilde{\quad}}$. After that, performing these
transformations in \eqref{4.8} and setting
$\mathfrak{r}=i\zeta+s+2$, we end up with the problem
\begin{equation*}\label{4.9}
\begin{cases}
\frac{\partial^{2}v^{*}_{1}}{\partial
x_2^{2}}+\mathfrak{r}^{2}v^{*}_{1}=0\quad
x_2\in(-\pi/2,\delta),\quad \tilde{v}_1(\zeta,x_2,0)=0\quad
 x_2\in[-\pi/2,\delta],\\
\frac{\partial^{2}v^{*}_{2}}{\partial
x_2^{2}}+\mathfrak{r}^{2}v^{*}_{2}=0\quad x_2\in(\delta,\pi/2),
\quad\tilde{v}_2(\zeta,x_2,0)=0,\quad x_2\in[\delta,\pi/2]\\
\mu[v_1^{*}(\zeta-i(s^{*}-2),\delta,\mu)-v_2^{*}(\zeta-i(s^{*}-2),\delta,\mu)]
-\frac{\partial}{\partial x_2}[v_1^{*}(\zeta,\delta,\mu)-v_2^{*}(\zeta,\delta,\mu)]\\
+a_2\mathfrak{r}[v_1^{*}(\zeta,\delta,\mu)-v_2^{*}(\zeta,\delta,\mu)]=-f^{*}_{1}(\zeta-i(1+s),\mu),\\
\frac{\partial v_1^{*}}{\partial
x_2}(\zeta,\delta,\mu)-k\frac{\partial v_2^{*}}{\partial
x_2}(\zeta,\delta,\mu)
+ka_3\mathfrak{r}[v_1^{*}(\zeta,\delta,\mu)-v_2^{*}(\zeta,\delta,\mu)]=0,\\
v_1^{*}(\zeta,-\pi/2,\mu)=0,\quad v_2^{*}(\zeta,\pi/2,\mu)=0.
\end{cases}
\end{equation*}
Clearly, selecting
\begin{equation*}
v_1^{\star}(\zeta,x_2,\mu)=\mathcal{M}_1(\zeta,
\mu)\sin\mathfrak{r}(x_2+\pi/2),\quad
v_2^{\star}(\zeta,x_2,\mu)=\mathcal{M}_2(\zeta,
\mu)\sin\mathfrak{r}(x_2-\pi/2),
\end{equation*}
with unknown functions $\mathcal{M}_1(\zeta, \mu)$ and
$\mathcal{M}_2(\zeta,\mu)$ specifying below, we satisfy the
equations and boundary conditions on $x_{2}=\pm\pi/2$. Then,
substituting these $v_{i}^{*}$ to the transmission condition and
denoting
\[
\mathcal{N}(\zeta)=\frac{\cos\mathfrak{r}(\delta-\pi/2)
+a_3\sin\mathfrak{r}(\delta-\pi/2)}{\cos\mathfrak{r}(\delta+\pi/2)
+ka_3\sin\mathfrak{r}(\delta+\pi/2)},
\]
 we end up with  the system to find $\mathcal{M}_1$ and $\mathcal{M}_2$
\begin{equation*}
\begin{cases}
\mathcal{M}_1(\zeta,
\mu)=k\mathcal{N}(\zeta)\mathcal{M}_2(\zeta,\mu),\\
\mu\mathcal{M}_2(\zeta-i(s^{\star}-2),\mu)[k\mathcal{N}(\zeta-i(s^{\star}-2))\sin\mathfrak{r}(\delta+\pi/2)-\sin\mathfrak{r}(\delta-\pi/2)]\\
-\mathfrak{r}\mathcal{M}_2(\zeta,\mu)\{a_2\sin\mathfrak{r}(\delta-\pi/2)-\cos\mathfrak{r}(\delta-\pi/2)
\\
+k\mathcal{N}(\zeta)[\cos\mathfrak{r}(\delta+\pi/2)-a_2\sin\mathfrak{r}(\delta+\pi/2)]\}=-f_{1}^{*}(\zeta-i(1+s),\mu).
\end{cases}
\end{equation*}
The first equation in this system tells that we are left to find the
function $\mathcal{M}_{2}$. Denoting
\[
\mathcal{N}_1(\zeta)=k\mathcal{N}\sin\mathfrak{r}(\delta+\pi/2)-\sin\mathfrak{r}(\delta-\pi/2),
\]
we rewrite the second equation in the form
\begin{align*}
&\mathfrak{r}\mathcal{M}_2(\zeta,\mu)\{k\mathcal{N}(\zeta)\cos\mathfrak{r}(\delta+\pi/2)
-a_2\mathcal{N}_1(\zeta)-\cos\mathfrak{r}(\delta-\pi/2)\}\\
& -
\mu\mathcal{M}_2(\zeta-i(s^{\star}-2),\mu)\mathcal{N}_1(\zeta-i(s^{\star}-2))
=f_{1}^{*}(\zeta-i(s+1),\mu),
\end{align*}
which in turn is reduced to a functional equation in the unknown
function
$v(\zeta,\mu)=\mathcal{M}_{2}(\zeta,\mu)\mathcal{N}_{1}(\zeta),$
\[
\mu
v(\zeta-i(s^{\star}-2),\mu)-\mathfrak{r}\mathcal{G}(\zeta)v(\zeta,\mu)
=-f_{1}^{*}(\zeta-i(1+s),\mu).
\]
Here, we put
\[
\mathcal{G}(\zeta)=\frac{k\mathcal{N}(\zeta)\cos\mathfrak{r}(\delta+\pi/2)-\cos\mathfrak{r}(\delta-\pi/2)}{\mathcal{N}_1(\zeta)}-a_2.
\]
In fine, introducing the new variable
\[
\zeta=-i(s^{*}-2)\nu,
\]
 and new functions
 \[
V(\nu,\mu)=v(-i(s^{*}-2)\nu,\mu),\qquad
f^{*}(\nu,\mu)=-f_{1}^{*}(-i(s^{*}-2)\nu-i(1+s),\mu),
 \]
 we rewrite the equation above in the form
 \begin{equation}\label{4.10}
\mu
V(\nu+1,\mu)-[s+2+\nu(s^{*}-2)]\mathcal{G}(-i(s^{*}-2)\nu)V(\nu,\mu)=f^{*}(\nu,\mu).
 \end{equation}
Appealing to the values $\theta_{1},\theta_{2},q_{2},q^{*}$ (see
Subsection \ref{s4.1}) and performing the direct calculations, we
conclude that
\[
\mathcal{G}(-i(s^{*}-2)\nu)=-\frac{\sqrt{1+a_{2}^{2}}S^{+}(2\delta(s+2+(s^{*}-2)\nu))}{S^{-}(2\delta(s+2+(s^{*}-2)\nu))}
\]
with $S^{+}$ and $S^{-}$ given with \eqref{4.0}. Moreover, taking
into account the inequality $q_{2}\neq 1$ and applying Corollaries
\ref{c4.2} and \ref{c4.3}, we factorize  the function $\mathcal{G}$.
 Namely, putting
\[
\hat{s}=\frac{2+s}{s^{*}-2},
\]
we get
\begin{align*}
\mathcal{G}(-i(s^{*}-2)\nu)&=\frac{(\sin\theta_{1}+q_{2}\sin\theta_{2})
\prod_{i=1}^{2p-1}(\z_{i}^{-})^{2}}{\sin\theta_{1}(q_{1}q_{2}^{*}-1)2\delta(s^{*}-2)[\nu+\hat{s}]\prod_{i=0}^{K^{+}-1}\z_{i}^{+}
}\\
& \times
\frac{\prod_{i=0}^{K^{+}-1}[\z_{i}^{+}-2\delta(s^{*}-2)(\nu+\hat{s})]}
{\prod_{i=1}^{2p-1}[\z_{i}^{-}-2\delta(s^{*}-2)(\nu+\hat{s})][\z_{i}^{-}+2\delta(s^{*}-2)(\nu+\hat{s})]}
\\
& \times \prod_{n=1}^{\infty}
\frac{\prod_{i=0}^{K^{+}-1}[\z_{i,n}^{+}-2\delta(s^{*}-2)(\nu+\hat{s})][-\z_{i,n}^{+}+2\delta(s^{*}-2)(\nu+\hat{s})]}
{\prod_{i=0}^{2p-1}[\z_{i,n}^{-}-2\delta(s^{*}-2)(\nu+\hat{s})][\z_{i,n}^{-}+2\delta(s^{*}-2)(\nu+\hat{s})]}
\frac{
\prod_{i=0}^{2p-1}(\z_{i,n}^{-})^{2}}{\prod_{i=0}^{K^{+}-1}\z_{i,n}^{+}(-\z_{i,-n}^{+})
}.
\end{align*}
Clearly,
\[
\mathcal{G}(0)=-\frac{\sin(2\delta(s+2)-\theta_{1})+q_{2}\sin(\pi(s+2)-\theta_{2})}{\sin\theta_{1}(\sin
2\delta(s+2)-q^{*}\sin\pi(s+2))}.
\]
Thanks to assumption (h8) on $s$ and \cite[Proposition 2]{V1}, the
following inequality holds
\[
\sin 2\delta(s+2)-q_{2}^{*}\sin\pi(s+2)\neq 0,
\]
which provides the well-posedness of  $\mathcal{G}(0)$. Denoting
\begin{align*}
Z^{+}_{i,n}&=\frac{\z_{i,n}^{+}-2\delta(2+s)}{2\delta(s^{*}-2)},\quad
Z_{i,0}^{+}=Z_{i}^{+},\quad
Z^{+}_{i,-n}=\frac{-\z_{i,-n}^{+}+2\delta(2+s)}{2\delta(s^{*}-2)},\quad
Z_{i,-0}^{+}=Z_{i,-}^{+},\\
Z^{-}_{i,n}&=\frac{\z_{i,n}^{-}-2\delta(2+s)}{2\delta(s^{*}-2)},\quad
Z_{i,0}^{-}=Z_{i}^{-},\quad
Z^{-}_{i,-n}=\frac{\z_{i,n}^{-}+2\delta(2+s)}{2\delta(s^{*}-2)},\quad
Z_{i,-0}^{-}=Z_{i,-}^{-},
\end{align*}
we end up with the desired factorization of the function
$\mathcal{G}$,
\begin{align*}\label{4.12}
\mathcal{G}(-i(s^{*}-2)\nu)&=\frac{\mathcal{G}(0)\hat{s}
\prod_{i=1}^{2p-1}Z_{i,-}^{-}Z_{i}^{-}\prod_{i=0}^{K^{+}-1}[Z_{i}^{+}-\nu]}
{[\nu+\hat{s}]\prod_{i=0}^{K^{+}-1}Z_{i}^{+}
\prod_{i=1}^{2p-1}[Z_{i}^{-}-\nu][Z_{i,-}^{-}+\nu]}
\\
& \times \prod_{n=1}^{\infty}
\frac{\prod_{i=0}^{K^{+}-1}[Z_{i,n}^{+}-\nu][Z_{i,-n}^{+}+\nu]\prod_{i=0}^{2p-1}Z_{i,n}^{-}Z_{i,-n}^{-}}
{\prod_{i=0}^{2p-1}[Z_{i,n}^{-}-\nu][Z_{i,n}^{-}+\nu]\prod_{i=0}^{K^{+}-1}Z_{i,n}^{+}Z_{i,-n}^{+}
}.
\end{align*}
We mention that equation \eqref{4.10} with the more complex
coefficient than $\mathcal{G}$ is studied in \cite{V4}. Indeed, the
explicit solution of this equation along with asymptotic behavior of
the solution are obtained in \cite[Sections 1-5]{V4}. Thus, we
rewrite these outcomes to equation \eqref{4.10} in our notations.
Coming to the construction of the solution,  \cite[Theorems
2.1-2.3]{V4} tell that it is represented as a composition of a
solution $V_{0}$ to the homogeneous equation and a partial solution
of inhomogeneous one.

At this point, we  aim to obtain the explicit form of $V_{0}$ and to
describe its properties. To this end, utilizing \cite[Theorems
2.1-2.2]{V4} (bearing in mind  the obtained factorization) and
setting
\begin{align*}
&\mathcal{L}_{0}(\nu)=\frac{\prod_{i=1}^{2p-1}\Gamma(Z_{i}^{-}-\nu+1)}{\prod_{i=0}^{K^{+}-1}
\Gamma(1+Z_{i}^{+}-\nu)\prod_{i=1}^{2p-1}\Gamma(Z^{-}_{i,-}+\nu)},\qquad
\mathfrak{d}=\frac{\mathcal{G}(0)\hat{s}
\prod_{i=1}^{2p-1}Z_{i,-}^{-}Z_{i}^{-}}
{\prod_{i=0}^{K^{+}-1}Z_{i}^{+}},\\
&\mathcal{R}=\sum_{i=0}^{K^{+}-1}[Z_{i,-n}^{+}(\ln\,Z_{i,-n}^{+}-1)-Z_{i,n}^{+}(\ln\,Z_{i,n}^{+}-1)]
+
\sum_{i=0}^{2p-1}[Z_{i,-n}^{-}(\ln\,Z_{i,-n}^{-}-1)-Z_{i,n}^{-}(\ln\,Z_{i,n}^{-}-1)],\\
&\mathcal{L}(\nu)=\prod_{n=1}^{\infty} e^{\mathcal{R}}
\bigg(\frac{\prod_{i=0}^{2p-1}Z_{i,n}^{-}Z_{i,-n}^{-}}{\prod_{i=0}^{K^{+}-1}Z_{i,n}^{+}Z_{i,-n}^{+}}
\bigg)^{p-\frac{1}{2}}
\frac{\prod_{i=0}^{2p-1}\Gamma(Z_{i,n}^{-}-\nu+1)\prod_{i=0}^{K^{+}-1}
\Gamma(Z_{i,-n}^{+}+\nu)}{\prod_{i=0}^{K^{+}-1}
\Gamma(1+Z_{i,n}^{+}-\nu)\prod_{i=0}^{2p-1}\Gamma(Z^{-}_{i,-n}+\nu)},\\
&\mathcal{P}(\nu)=\frac{\prod_{i=1}^{i^{*}_{1}+2}\sin\pi(1+Z_{i}^{-}-\nu)}
{\prod_{i=0}^{i^{*}_{1}-1}\sin\pi(1+Z_{i}^{+}-\nu)},
\end{align*}
we claim.
\begin{proposition}\label{p4.1}
The function
\[
V_{0}(\nu,\mu)=(\mathfrak{d})^{\nu-1/2}(\mu(s^{*}-2))^{1/2-\nu}\P(\nu)\L_{0}(\nu)\L(\nu)
\] solves homogenous equation \eqref{4.10}, if
the following inequalities hold for all $m\in\mathbb{N}\cup\{0\}$
and $l,n\in\mathbb{N}$:
\begin{equation*}\label{4.13}
\begin{cases}
Re\,\nu\neq 1+Z_{i}^{+}-l,\qquad i=0,1,...,i^{*}_{1}-1,\\
Re\,\nu\neq 1+Z_{i}^{-}+m,\qquad i=i^{*}_{1}+3,...,2p-1,\\
Re\,\nu\neq Z_{i,n}^{-}+1+m,\qquad i=0,1,...,2p-1,\\
Re\,\nu\neq -Z_{i,-n}^{+}-m,\qquad i=0,1,...,K^{+}-1.
\end{cases}
\end{equation*}
The functions $\L_{0}(\nu)$ and $\mathcal{L}(\nu)$ have no poles and
the infinite product $\L(\nu)$ converges for any complex $\nu$
satisfying \eqref{4.13}. Moreover, for $|Im\,\nu|\to+\infty$ and
$Re\,\nu$ satisfying \eqref{4.13}, there is the asymptotic
\begin{align*}\label{4.13*}\notag
\mathcal{P}(\nu)&\approx e^{2\pi|Im\,\nu|},\\
(\mathfrak{d})^{\nu-1/2}\L_{0}(\nu)\L(\nu)&\approx
[(\hat{s}+\nu)\mathcal{G}(-i(s^{*}-2)\nu)]^{\nu-1/2}\exp\{C_{1}\ln\,\nu+C_{2}\nu+O(1)\}
\end{align*}
with real constants $C_{1}$ and $C_{2}$.
\end{proposition}
Taking into account the explicit form of $V_{0}$ and performing the
straightforward calculations, we obtain the following properties of
$V_{0}$, which play a crucial role in the finding solution of
\eqref{4.10} with $f^{*}\neq 0$.
\begin{lemma}\label{l4.1}
The function $V_{0}$ introducing in Proposition \ref{p4.1} possesses
the following properties:

\noindent(i) $V_{0}(\nu,\mu)$ does not have any poles if
\[
-Z^{+}_{K^{+}-1,-1}<Re\,\nu<Z^{*}\quad\text{and}\quad Re\,\nu\neq
1-l+Z_{i}^{+},\quad l\in\mathbb{N},\, i=0,1,...,i^{*}_{1}-1,
\]
where
\[
Z^{*}=\begin{cases} 1+Z_{i^{*}_{1}+2}^{-}\qquad\text{if}\quad
2p>i_{1}^{*}+3,\\
4\pi q\qquad\qquad\quad \text{if}\quad 2p\leq i_{1}^{*}+3;
\end{cases}
\]

\noindent(ii) $V_{0}(\nu,\mu)$ does not vanish if $Re\,\nu$
satisfies the inequalities
\[
-Z^{-}_{1,-}<Re\,\nu<1+Z^{+}_{i^{*}_{1}}\quad\text{and}\quad
Re\,\nu\neq 1-m+Z_{i}^{-},\quad m\in\mathbb{N},\,
i=1,...,i^{*}_{1}+2;
\]

\noindent(iii) if $Re\,\nu$ meets the requirements stated in (ii),
then there is the bound
\[
\frac{1}{|V_{0}(\nu,\mu)|}\leq
C|\nu|^{-C_{1}-Re\,\nu+1/2}\exp\{-[\pi-\theta_{2}]|Im\,\nu|\}\quad\text{as}\quad
|Im\,\nu|\to+\infty.
\]
\end{lemma}
In fine, collecting Lemma \ref{l4.1} and Proposition \ref{p4.1}
 with \cite[Theorem 2.2]{V4}, we construct the solution of
inhomogeneous equation \eqref{4.10}. To this end, for any fixed
$d_0\in[0,1]$,  we introduce the contour $\ell_{d_{0}}$ in the
complex plane $\xi\in\mathbb{C}$:

\noindent$\bullet$ $\ell_{d_{0}}=\{Re\, \xi=-d_{0},\, Im\,
\xi\in\R\}$ if $d_0\in(0,1)$;

\noindent$\bullet$ if $d_0=1$, then $\ell_{d_{0}}=\ell_{1}$ consists
of three parts: the half-circle $\{|\xi+1|=d_1,\, Re\, \xi>-1\}$
with a small positive number $d_1$, $0<d_1<d_0/8,$ and  the
intervals: $\{Re\, \xi=-1,\, Im\, \xi\in(-\infty,-d_1)\}$ and
$\{Re\, \xi=-1,\, Im\, \xi\in(d_1,+\infty)\}$;

\noindent$\bullet$ the contour $\ell_{0}$ (i.e. $d_0=0$) is built
via $\ell_{1}$ after its shifting to the right-hand side on $Re\,
\xi=1$.
\begin{lemma}\label{l4.2}
Let
\[
-Z_{1,-}^{-}<Re\,\nu<Z_{i^{*}_{1}}^{+}\quad\text{and}\quad
Re\,\nu\neq d_{0}+Z_{i,-}^{-}-m\quad i=1,2,...,i_{1}^{*}+2,\quad
m\in\mathbb{N}
\]
for any $d_{0}\in[0,1]$. Then the function
\[
V(\nu,\mu)=\frac{i}{2}\int_{\ell_{d_{0}}}\frac{V_{0}(\nu,\mu)f_{1}^{*}(i\nu(2-s^{*})-i(s^{*}-2)\xi-i(s+1),\mu)[\cot\pi\xi+i]\sin^{2}\pi
d}{V_{0}(\nu+1+\xi,\mu)\sin^{2}\pi(\xi+d)} d\xi
\]
with $\xi=-d_{0}+iz,$ $z\in\R$ and some fixed $d,$ $0<d-d_{0}<1$,
solves \eqref{4.10}.

If, in additionally,
\[
Re\, \nu\neq 1-m+Z_{i}^{+},\quad i=0,1,...,i_{1}^{*}-1,\,
m\in\mathbb{N},
\]
then the function $\frac{V_{0}(\nu,\mu)}{V_{0}(\nu+1+\xi,\mu)}$ is
analytic in $\nu$.
\end{lemma}
At this point, we are ready to build solutions $v_{1}(x,t)$ and
$v_{2}(x,t)$. To this end, setting
\[
\mathcal{N}_{2}(\mathfrak{r})=\frac{\cot\mathfrak{r}(\delta-\pi/2)+a_{3}}{\cot\mathfrak{r}(\delta+\pi/2)+k\a_{3}},
\]
and substituting $V(\nu,\mu)$ in the expressions of
$\mathcal{M}_{1}$ and $\mathcal{M}_{2}$, we arrive at
\begin{align*}
v_{1}^{*}(\zeta,x_{2},\mu)&=\frac{ik\mathcal{N}_{2}(i\zeta+s+2)\sin(i\zeta+s+2)(x_{2}+\pi/2)}
{[k\mathcal{N}_{2}(i\zeta+s+2)+1]\sin(i\zeta+s+2)(\delta+\pi/2)}\\
& \times
\int_{\ell_{d_{0}}}\frac{V_{0}(i\zeta(s^{*}-2)^{-1},\mu)f_{1}^{*}(\zeta-i(s^{*}-2)\xi-i(s+1),\mu)[\cot\pi\xi+i]\sin^{2}\pi
d}{V_{0}(1+\xi+i\zeta(s^{*}-2)^{-1},\mu)\sin^{2}\pi(\xi+d)}
d\xi,\\
v_{2}^{*}(\zeta,x_{2},\mu)&=\frac{i\sin(i\zeta+s+2)(x_{2}-\pi/2)}
{[k\mathcal{N}_{2}(i\zeta+s+2)+1]\sin(i\zeta+s+2)(\delta-\pi/2)}\\
& \times
\int_{\ell_{d_{0}}}\frac{V_{0}(i\zeta(s^{*}-2)^{-1},\mu)f_{1}^{*}(\zeta-i(s^{*}-2)\xi-i(s+1),\mu)[\cot\pi\xi+i]\sin^{2}\pi
d}{V_{0}(1+\xi+i\zeta(s^{*}-2)^{-1},\mu)\sin^{2}\pi(\xi+d)} d\xi,
\end{align*}
At last, appealing to the easily verified equality
\[
\frac{1}{2i\pi}\int_{-i\infty}^{+\infty}\mu^{-d_{0}+iz}d\mu=\frac{t^{d_{0}-1-iz}}{\Gamma(d_{0}-iz)}
\]
and performing the inverse Laplace and Fourier transformations in
$v_{1}^{*}$ and $v_{2}^{*}$, we arrive at the integral
representation of solution to \eqref{4.8}. Namely, denoting
\begin{align*}
f(x_{1},t)&=f_{1}(x_{1},t)e^{-(1+s)x_{1}},\\
\mathcal{H}(z,t,x_{1})&=\frac{\exp\{-i(s^{*}-2)(\ln(s^{*}-2)t+x_{1})\}
\sin^{2}\pi d}{\sin\pi(d_{0}-iz)\Gamma(d_{0}-iz)\sin^{2}\pi(d-d_{0}+iz)},\\
\mathcal{A}_{1}(\zeta,x_{1},z)&=\frac{\mathcal{P}(i\zeta/(s^{*}-2))\mathcal{L}_{0}(i\zeta/(s^{*}-2))
\mathcal{L}(i\zeta/(s^{*}-2))\mathfrak{d}^{-iz}} {
\mathcal{P}(1-d_{0}+iz+i\zeta/(s^{*}-2))\mathcal{L}_{0}(1-d_{0}+iz+i\zeta/(s^{*}-2))
\mathcal{L}(1-d_{0}+iz+i\zeta/(s^{*}-2))}\\
& \times
\frac{ik\mathcal{N}_{2}(i\zeta+s+2)e^{ix_{1}(\zeta+(s^{*}-2)z)}}{[k\mathcal{N}_{2}(i\zeta+s+2)+1]},\\
\mathcal{A}_{2}(\zeta,x_{1},z)&=\frac{\mathcal{P}(i\zeta/(s^{*}-2))\mathcal{L}_{0}(i\zeta/(s^{*}-2))
\mathcal{L}(i\zeta/(s^{*}-2))\mathfrak{d}^{-iz}} {
\mathcal{P}(1-d_{0}+iz+i\zeta/(s^{*}-2))\mathcal{L}_{0}(1-d_{0}+iz+i\zeta/(s^{*}-2))
\mathcal{L}(1-d_{0}+iz+i\zeta/(s^{*}-2))}\\
& \times
\frac{ie^{ix_{1}(\zeta+(s^{*}-2)z)}}{[k\mathcal{N}_{2}(i\zeta+s+2)+1]},\\
\mathcal{B}_{1}(\zeta,x_{2})&=\frac{\sin(i\zeta+s+2)(x_{2}+\pi/2)}
{\sin(i\zeta+s+2)(\delta+\pi/2)},\quad
\mathcal{B}_{2}(\zeta,x_{2})=\frac{\sin(i\zeta+s+2)(x_{2}-\pi/2)}
{\sin(i\zeta+s+2)(\delta-\pi/2)},
\end{align*}
and exploiting Lemma \ref{l4.2}, we claim the following.
\begin{lemma}\label{l4.3}
Let $f\in\C^{1+\beta,\beta,\beta}(b_{T}^{+})$,  and \eqref{4.1},
\eqref{4.3} hold. We assume that $Im\,\zeta$ satisfies the following
inequalities
\[
\begin{cases}
-Z_{1,-}^{-}<-\frac{Im\,\zeta}{s^{*}-2}< Z_{i^{*}_{1}}^{+},\\
-\frac{Im\,\zeta}{s^{*}-2}\neq d_{0}+Z_{i,-}^{-}-m,\quad
i=1,2,...,i_{1}^{*}+2,\\
-\frac{Im\,\zeta}{s^{*}-2}\neq 1+Z_{j}^{+}-m,\quad
j=0,1,...,i_{1}^{*}-1
\end{cases}
\]
for each $m\in\mathbb{N}$ and any $d_{0}\in[0,1]$. Then problem
\eqref{4.8} admits solution
\begin{align}\label{4.14}\notag
v_{1}(x_{1},x_{2},t)&=const.\int_{0}^{t}\frac{d\tau}{[(s^{*}-2)(t-\tau)]^{1-d_{0}}}\int_{-\infty}^{\infty}d\varsigma
f(x_{1}-\varsigma,\tau)e^{(s^{*}-2)d_{0}(x_{1}-\varsigma)}\\\notag &
\times\int_{-\infty}^{+\infty}dz\mathcal{H}(z,t-\tau,x_{1})\int_{-\infty}^{+\infty}d\zeta\,\mathcal{A}_{1}(\zeta,\varsigma,z)\mathcal{B}_{1}(\zeta,x_{2}),\\
v_{2}(x_{1},x_{2},t)&=const.\int_{0}^{t}\frac{d\tau}{[(s^{*}-2)(t-\tau)]^{1-d_{0}}}\int_{-\infty}^{\infty}d\varsigma
f(x_{1}-\varsigma,\tau)e^{(s^{*}-2)d_{0}(x_{1}-\varsigma)}\\\notag &
\times\int_{-\infty}^{+\infty}dz\mathcal{H}(z,t-\tau,x_{1})\int_{-\infty}^{+\infty}d\zeta\,\mathcal{A}_{2}(\zeta,\varsigma,z)\mathcal{B}_{2}(\zeta,x_{2}).
\end{align}
\end{lemma}

Clearly, conditions on $Im\,\zeta$ stated in Lemma \ref{l4.3}
dictate the restrictions on the weight $s$. Namely, the integrands
$\mathcal{A}_{i}$ and $\mathcal{B}_{i}$ in \eqref{4.14} should be
left analytic in $\zeta$ if, in particulary, $Im\,\zeta=0$. This
fact leads to the inequalities
\[
Z_{i^{*}_{1}}^{+}>0, \quad Z_{1,-}^{-}>0,\quad 0\neq
d_{0}+Z_{i,-}^{-}-m,\quad 0\neq 1-m+Z_{j}^{+},
\]
for $i=1,2,...,i_{1}^{*}+2,$ $j=0,1,...,i_{1}^{*}-1$ and each
$d_{0}\in[0,1]$, which in turn provide
\begin{equation}\label{4.15}
\begin{cases}
-\frac{\z^{-}_{1}}{2\delta}<s+2<\frac{\z^{+}_{i^{*}_{1}}}{2\delta},\\
s+2\neq (s^{*}-2)(m-d_{0})-\frac{\z^{-}_{i}}{2\delta},
\quad i=1,2,...,i_{1}^{*}+2,\quad d_{0}\in[0,1],\\
s+2\neq (s^{*}-2)(1-m)+\frac{\z^{+}_{j}}{2\delta},\quad
j=0,1,...,i_{1}^{*}-1,\quad m\in\mathbb{N}.
\end{cases}
\end{equation}
Here, we appealed to Corollary \ref{c4.2}, which, in particulary,
provides the bound
\[
\frac{\z^{-}_{1}}{2\delta}<1/2.
\]
Obviously, if $s$ meets the requirements of Theorem \ref{t4.1}, then
\eqref{4.15} holds. Here,  we restrict ourself with verification of
this fact in the case of $q_{2}>1$ and $\theta_{2}>\pi/2$ (which
corresponds to (ii) in Corollary \ref{c4.1}). The remaining cases
are analyzed with the similar arguments. To this end, exploiting
Corollary \ref{c4.3} and conditions \eqref{4.1} and \eqref{4.4}, we
easily deduce that
\[
i_{1}^{*}=
\begin{cases}
3\qquad\text{if}\quad \frac{\z_{2}^{+}}{2\delta}<3,\\
2\qquad\text{if}\quad \frac{\z_{2}^{+}}{2\delta}=3
\end{cases}
\]
It is apparent that in the studied case $K^{+}-1=2p-1\geq 5$.

Further, setting, for simplicity,  $i_{1}^{*}=3$, we  rewrite
\eqref{4.15} in the form
\begin{equation}\label{4.16*}
\begin{cases}
-\frac{\z^{-}_{1}}{2\delta}<s+2<\frac{\z^{+}_{3}}{2\delta},\\
s+2\neq (s^{*}-2)(m_{1}-d_{0})-\frac{\z^{-}_{i}}{2\delta},
\quad i=1,2,...,i_{1}^{*}+2,\quad d_{0}\in[0,1],\\
s+2\neq (s^{*}-2)(1-m_{2})+\frac{\z^{+}_{j}}{2\delta},\quad
j=0,1,...,i_{1}^{*}-1,\quad m\in\mathbb{N},
\end{cases}
\end{equation}
where $i=1,2,...,5,$ $j=0,1,2,$ and integer $m_{1},$ $m_{2}$ satisfy
the inequalities
\begin{equation}\label{4.16}
0<m_{1}<1+\frac{\z_{3}^{+}+\z^{-}_{5}}{2\delta(s^{*}-2)},\qquad
0<m_{2}<1+\frac{\z_{2}^{+}-\z^{-}_{1}}{2\delta(s^{*}-2)}.
\end{equation}
Then, keeping in mind these relations and coming to the definition
of  sets $\mathbb{M}_{1}^{+}$ and $\mathbb{M}_{1}^{-}$ (see
\eqref{4.0*}), we yield
\[
\mathbb{M}_{1}^{-}\subset\{0,1,...,M_{1}^{-}\}\quad\text{with
integer}\quad M_{1}^{-}\in[4,11];\quad
\mathbb{M}_{1}^{+}\subset\{0,1,...,M_{1}^{+}\}\quad\text{with
integer}\quad M_{1}^{+}\in[1,6].
\]
We notice that the restrictions on $M_{1}^{+}$ and $M_{1}^{-}$
follow from Corollaries \ref{c4.2}-\ref{c4.3} and inequalities
\eqref{4.16}.

Thus, appealing to \eqref{4.5}, we compute $\z_{1}^{*}$ by the
values
\begin{align*}
\underline{\z_{1}}&=\max\Big\{-\frac{\z_{i}^{-}}{2\delta}+[s^{*}-2][m_{1}-d_{0}],\,
m_{1}\in\mathbb{M}_{1}^{-},\, i=1,2,...,5,\,
d_{0}\in[0,1]\Big\},\\
\overline{\z_{1}}&=\max\Big\{\frac{\z_{i}^{+}}{2\delta}-[s^{*}-2][m_{2}-1],\,
m_{2}\in\mathbb{M}_{1}^{+},\, i=0,1,2\Big\}.
\end{align*}

At last, collecting  \eqref{4.16*}, \eqref{4.16} with the definition
of $\mathbb{M}_{1}^{-}$ and $\mathbb{M}_{1}^{+}$, we conclude that
the value $s$ satisfying assumption (h8) fits to conditions
\eqref{4.15}. Thus, the first stage in our arguments is completed.

\smallskip
\noindent\textit{Stage 2: Estimates of $u_{i}$.} It is worth noting
that the estimates of the solution $(u_{1},u_{2})$ and, accordingly,
$\varrho$ are simple consequence of the corresponding bounds of
$v_{1}$ and $v_{2}$ constructed in Lemma \ref{l4.3}.
\begin{lemma}\label{l4.4}
Let assumptions of Theorem \ref{t4.1} holds. Then there are the
estimates
\begin{align*}
&\sum_{i=1}^{2}[\|v_{i}\|_{\C^{2+\beta,\beta,\beta}(\bar{B}_{i,T})}+\|e^{(2-s^{*})x_{1}}\tfrac{\partial
v_{i}}{\partial t}\|_{\C^{1+\beta,\beta,\beta}(b_{T})}]\leq
C\|f\|_{\C^{1+\beta,\beta,\beta}(b_{T})},\\
&\sum_{i=1}^{2}\|v_{i}\|_{\C^{2+\beta,\beta,\beta}(\bar{B}_{i,T})}\leq
C
T^{\beta^{*}-\beta}R_{0}^{\beta^{*}(s^{*}-2)}\|f\|_{\C^{1+\beta,\beta,\beta}(b_{T})}
\end{align*}
with any $\beta^{*}\in(\beta,1)$.
\end{lemma}
\begin{proof}
First of all, assumptions on $f_{1}$ along with Corollary \ref{c2.2}
provide the regularity $f\in \C^{1+\beta,\beta,\beta}(b_{T})$ and
the bound
\[
\|f\|_{\C^{1+\beta,\beta,\beta}(b_{T})}\leq
C\|f_{1}\|_{E_{s+1}^{1+\beta,\beta,\beta}(g_{T})}.
\]
Next, to estimate $v_{1}$ and $v_{2}$, we will exploit the arguments
of \cite[Sections 3-6]{V2}. To this end, we need in the asymptotic
behavior of the functions $\mathcal{B}_{i},$ $\mathcal{A}_{i},$
$\mathcal{H}$. Performing technical computations, we deduce that
\begin{align*}
\mathcal{B}_{i}(\zeta,\delta)&=1,\quad
\mathcal{B}_{1}(\zeta,x_{2})\approx\begin{cases}
\text{const.}\qquad\qquad\qquad\,\text{ if}\quad |Re\,\zeta|<1,\\
\text{const.}e^{-(\delta-x_{2})|Re\,\zeta|}\quad\text{if}\quad
|Re\,\zeta|\to+\infty,
\end{cases}\\
\mathcal{B}_{2}(\zeta,x_{2})&\approx\begin{cases}
\text{const.}\qquad\qquad\qquad\,\text{ if}\quad |Re\,\zeta|<1,\\
\text{const.}e^{-(x_{2}-\delta)|Re\,\zeta|}\quad\text{if}\quad
|Re\,\zeta|\to+\infty,
\end{cases}\\
\mathcal{H}(z,t,x_{1})&\approx
\begin{cases}
\text{const.}\qquad\qquad\qquad\,\text{ if}\quad |z|<1,\\
\text{const.}\exp\{iz\ln|z|+d_{0}-iz-iz
x_{1}-iz\ln(s^{*}-2)t-\tfrac{5\pi|z|}{2}\}\quad\text{if}\quad
|z|\to+\infty.
\end{cases}
\end{align*}
Here, we utilized the well-known asymptotic
\begin{align*}
\Gamma(iz_{1}+z_{2})&\approx const.
\exp\{-\frac{\pi|z_{1}|}{2}+iz_{1}\ln|iz_{1}+z_{2}|-iz_{1}+O(1)\}|iz_{1}+z_{2}|^{z_{2}-1/2},\\
\sin\pi(iz_{1}+z_{2})&\approx e^{\pi|z_{1}|}
\end{align*}
as $|z_{1}|\to+\infty$ and bounded $z_{2}$.

At this point, we aim to get the asymptotic to the functions
$\mathcal{A}_{1}$ and $\mathcal{A}_{2}$. To this end, collecting
Proposition \ref{p4.1} with the easily verified  relations:
\begin{equation}\label{4.17}
\mathcal{G}(\zeta)\approx\frac{q_{2}}{q^{*}\sin\theta_{1}}\begin{cases}
e^{i\theta_{2}}\quad\, Re\,\zeta\to+\infty,\\
e^{-i\theta_{2}}\quad Re\,\zeta\to-\infty,
\end{cases}\,
\mathcal{N}_{2}(i\zeta+s+2),\frac{\mathcal{N}_{2}(i\zeta+s+2)}{1+k\mathcal{N}_{2}(i\zeta+s+2)}\approx\text{const.,}\quad
|Re\,\zeta|\to +\infty
\end{equation}
with $Im\,\zeta$ meeting requirements of Lemma \ref{l4.3}, we end up
with the asymptotic
\begin{align}\label{4.18}\notag
&\mathcal{A}(\nu)=(\mathfrak{d})^{\nu-1/2}\mathcal{L}_{0}(\nu)\mathcal{L}(\nu)\mathcal{P}(\nu)
\approx[(Im\,\nu)^{2}+(\hat{s}+Re\,\nu)^{2}]^{\frac{Re\,\nu-1/2}{2}}
[(Im\,\nu)^{2}+(Re\,\nu)^{2}]^{-\frac{C_{1}}{2}}\\
& \times \exp\{(\tfrac{3\pi}{2}-\theta_{2})|Im\,\nu|+\tfrac{i
Im\,\nu}{2}\ln[(Im\,\nu)^{2}+(\hat{s}+Re\,\nu)^{2}]-i
C_{3}Im\,\nu+O(1)\}\quad|Im\,\nu|\to+\infty,
\end{align}
where
\[
C_{3}=C_{2}+\ln\frac{q_{2}}{q^{*}\sin\theta_{1}},\quad
Im\,\nu=\frac{Re\,\zeta}{s^{*}-2},\quad
Re\,\nu=-\frac{Im\,\zeta}{s^{*}-2}.
\]
After that, performing the change of variable
\[
\mathfrak{K}=\mathfrak{K}_{1}+i\mathfrak{K}_{2}=z+\frac{\zeta}{s^{*}-2}
\]
in the integrals in \eqref{4.14}, we have
\begin{align}\label{4.19}\notag
\mathcal{A}_{1}((s^{*}-2)(\mathfrak{K}-z),\varsigma,z)&=\frac{k\mathcal{N}_{2}(s+2+i(s^{*}-2)
(\mathfrak{K}-z))\mathcal{A}(i(\mathfrak{K}-z))}{[1+k\mathcal{N}_{2}(s+2+i(s^{*}-2)
(\mathfrak{K}-z))]\mathcal{A}(i\mathfrak{K}+1-d_{0})}\exp\{i\varsigma\mathfrak{K}(s^{*}-2)\}\\
\mathcal{A}_{2}((s^{*}-2)(\mathfrak{K}-z),\varsigma,z)&=\frac{\exp\{i\varsigma\mathfrak{K}(s^{*}-2)\}\mathcal{A}(i(\mathfrak{K}-z))}{[1+k\mathcal{N}_{2}(s+2+i(s^{*}-2)
(\mathfrak{K}-z))]\mathcal{A}(i\mathfrak{K}+1-d_{0})}.
\end{align}
Then, following \cite[(4.4)]{V2}, we decompose the plane
$(\mathfrak{K}_{1},z)$ in the subdomains
\[
(\mathfrak{K}_{1},z)=(\cup_{i=1}^{8}\mathcal{D}_{i}^{+})\cup
(\cup_{i=1}^{8}\mathcal{D}_{i}^{-}),
\]
where for sufficiently large  value $K$ we set
\begin{align*}
&\mathcal{D}_{1}^{+}
=\{(\mathfrak{K}_{1},z):z\in(0,2K),\mathfrak{K}_{1}\geq3K\},\qquad\qquad
\mathcal{D}_{2}^{+}    =\{(\mathfrak{K}_{1},z):z\geq2K,\mathfrak{K}_{1}\geq z+K\},\\
&\mathcal{D}_{3}^{+}    =\{(\mathfrak{K}_{1},z):z\geq2K,z-K\leq
\mathfrak{K}_{1}<z+K\},\quad
\mathcal{D}_{4}^{+}   =\{(\mathfrak{K}_{1},z):z\geq2K,K<\mathfrak{K}_{1}\leq z-K\},\\
&\mathcal{D}_{5}^{+}    =\{(\mathfrak{K}_{1},z):z\geq2K,-K\leq
\mathfrak{K}_{1}<K\},\qquad\qquad
\mathcal{D}_{6}^{+}    =\{(\mathfrak{K}_{1},z):z\geq2K,\mathfrak{K}_{1}<-K\},\\
&\mathcal{D}_{7}^{+}
=\{(\mathfrak{K}_{1},z):z\in(0,2K),\mathfrak{K}_{1}\leq-3K\},\qquad\qquad
\mathcal{D}_{8}^{+}    =\{(\mathfrak{K}_{1},z):z\in(0,2K),-3K\leq
\mathfrak{K}_{1}\leq3K\}.
\end{align*}
Replacing $z$ on $-z$ in $\mathcal{D}_{j}^{+}$ arrives at
$\mathcal{D}_{j}^{-},$ $j=1,2,...,8$.

At last, collecting \eqref{4.17}-\eqref{4.19} end up with  the
following asymptotic
\begin{equation*}
\mathcal{A}_{j}e^{-i\varsigma(s^{*}-2)\mathfrak{K}}\sim
\left\{%
\begin{array}{ll}
   \mathfrak{R}_{j,1}(\mathfrak{K},z)\frac{e^{-(\frac{3\pi}{2}-\theta_{2})z+iz[C_{3}-\ln(\mathfrak{K}_{1}-z)]
   +i\mathfrak{K}_{1}\ln\frac{\mathfrak{K}_{1}-z}{\mathfrak{K}_{1}}}}
    {\mathfrak{K}_{1}^{1-d_{0}}}\left(\frac{\mathfrak{K}_{1}-z}{\mathfrak{K}_{1}}\right)^{-\mathfrak{K}_{2}-C_{1}-1/2}\quad\quad \text{in}\ \mathcal{D}_{1}^{+},  & \hbox{} \\
   \mathfrak{R}_{j,2}(\mathfrak{K},z)
   \frac{e^{-(\frac{3\pi}{2}-\theta_{2})z+iz[C_{3}-\ln(\mathfrak{K}_{1}-z)]+i\mathfrak{K}_{1}\ln\frac{\mathfrak{K}_{1}-z}{\mathfrak{K}_{1}}}}
    {\mathfrak{K}_{1}^{1-d_{0}}}\left(\frac{\mathfrak{K}_{1}-z}{\mathfrak{K}_{1}}\right)^{-\mathfrak{K}_{2}-C_{1}-1/2}\quad\quad\, \text{ in}\ \mathcal{D}_{2}^{+},  & \hbox{} \\
   \mathfrak{R}_{j,3}(\mathfrak{K},z)
  e^{-(\frac{3\pi}{2}-\theta_{2})\mathfrak{K}_{1}+i\mathfrak{K}_{1}[C_{3}-\ln\mathfrak{K}_{1}]}
 \mathfrak{K}_{1}^{-1/2+d_{0}+k_{2}+C_{1}}\qquad\qquad\qquad\quad\quad\text{ in}\ \mathcal{D}_{3}^{+},  & \hbox{} \\
   \mathfrak{R}_{j,4}(\mathfrak{K},z)
   \frac{e^{(\frac{3\pi}{2}-\theta_{2})(z-2\mathfrak{K}_{1})+iz[C_{3}-\ln|\mathfrak{K}_{1}-z|]+i\mathfrak{K}_{1}\ln\frac{|\mathfrak{K}_{1}-z|}
{\mathfrak{K}_{1}}}}{\mathfrak{K}_{1}^{1-d_{0}}}\left(\frac{|\mathfrak{K}_{1}-z|}{\mathfrak{K}_{1}}\right)^{-\mathfrak{K}_{2}-C_{1}-\frac{1}{2}}\
 \text{ in}\ \mathcal{D}_{4}^{+},  & \hbox{} \\
   \mathfrak{R}_{j,5}(\mathfrak{K},z)
e^{(\frac{3\pi}{2}-\theta_{2})(z-\mathfrak{K}_{1})+i(\mathfrak{K}_{1}-z)[C_{3}+\ln|\mathfrak{K}_{1}-z|]}
|\mathfrak{K}_{1}-z|^{-\mathfrak{K}_{2}-C_{1}-1/2}\quad\qquad \text{ in}\ \mathcal{D}_{5}^{+},  & \hbox{} \\
   \mathfrak{R}_{j,6}(\mathfrak{K},z)
\frac{e^{(\frac{3\pi}{2}-\theta_{2})z+iz[C_{3}-\ln|\mathfrak{K}_{1}-z|]+i\mathfrak{K}_{1}\ln\frac{|\mathfrak{K}_{1}-z|}{\mathfrak{K}_{1}}}}
    {|\mathfrak{K}_{1}|^{1-d_{0}}}\left|\frac{\mathfrak{K}_{1}-z}{\mathfrak{K}_{1}}\right|^{-\mathfrak{K}_{2}-C_{1}-1/2}\qquad\qquad \text{in}\ \mathcal{D}_{6}^{+},
    & \hbox{} \\
\mathfrak{R}_{j,7}(\mathfrak{K},z)
\frac{e^{(\frac{3\pi}{2}-\theta_{2})z+iz[C_{3}-\ln|\mathfrak{K}_{1}-z|]+i\mathfrak{K}_{1}\ln\frac{|\mathfrak{K}_{1}-z|}{\mathfrak{K}_{1}}}}
    {|\mathfrak{K}_{1}|^{1-d_{0}}}\left|\frac{\mathfrak{K}_{1}-z}{\mathfrak{K}_{1}}\right|^{-\mathfrak{K}_{2}-C_{1}-1/2}\qquad\qquad \text{in}\  \mathcal{D}_{7}^{+},  & \hbox{} \\
   \mathfrak{R}_{j,8}(\mathfrak{K},z),\qquad\qquad\qquad\qquad \qquad\qquad\text{in}\quad \mathcal{D}_{8}^{+},  & \hbox{} \\
    \end{array}%
    \right.
\end{equation*}
where the functions $\mathfrak{R}_{j,i}(\mathfrak{K},z)$,
$i=1,2,...,8,$ $j=1,2,$ along with $\frac{\partial
\mathfrak{R}_{j,i}}{\partial z}$ are uniformly bounded in
$\mathfrak{K}_{1}$ and $z$. Namely, there are the following
estimates
\begin{align*}
&\sum_{i=1,i\neq 3,5,8}^{7}|\mathfrak{R}_{j,i}(\mathfrak{K},z)|\leq
C[1+\mathfrak{K}_{1}^{-2}+(\mathfrak{K}_{1}-z)^{-2}],\\
&|\mathfrak{R}_{j,3}(\mathfrak{K},z)|\leq
C[1+\mathfrak{K}_{1}^{-2}],\quad
|\mathfrak{R}_{j,5}(\mathfrak{K},z)|\leq
C[1+(\mathfrak{K}_{1}-z)^{-2}],\quad
|\mathfrak{R}_{j,8}(\mathfrak{K},z)|\leq C,\\
&
\Big|\frac{\partial\mathfrak{R}_{j,1}}{\partial\mathfrak{K}_{1}}\Big|+
\Big|\frac{\partial\mathfrak{R}_{j,2}}{\partial\mathfrak{K}_{1}}\Big|+
\Big|\frac{\partial\mathfrak{R}_{j,6}}{\partial\mathfrak{K}_{1}}\Big|
+
\Big|\frac{\partial\mathfrak{R}_{j,7}}{\partial\mathfrak{K}_{1}}\Big|\leq
C|\mathfrak{K}_{1}|^{-3}.
\end{align*}
We notice that the similar asymptotic and estimates are true in
$\mathcal{D}^{-}_{i},$ $i=1,2,...,8$.

In fine, collecting these asymptotic with the corresponding behavior
of $\mathcal{B}_{i},$ $\mathcal{H}$ and recasting step by step the
arguments of \cite[Sections 3-6]{V2}, we end up with the desired
estimates of $v_{i}.$ It is worth noting that, these arguments
contain a lot of technical calculations but do not bring additional
conceptual difficulties. Therefore, we  omit them here and only
express two notions.

\noindent$\bullet$ In order to evaluate the senior derivatives of
$v_{i}$, we should chose $d_{0}=1$ in the representation
\eqref{4.14} to $v_{i}$ and in the corresponding asymptotic, while
to verify the second estimate in this lemma,  we select any
$d_{0}\in(\beta,1)$.

\noindent$\bullet$ Thanks to the appropriate selection of the value
$\mathfrak{K}_{2}$ (satisfying requirements of Lemma \ref{l4.3}),
 the factor
$\exp\{-\varsigma\mathfrak{K}_{2}(s^{*}-2)\}$ along with the
properties of the function $f$ ensure the convergence of the
corresponding  integrals in $\varsigma$ (in \eqref{4.14}).
\end{proof}


\subsection{The Proof of Theorem \ref{t4.2}}\label{s4.4}
The verification of this claim is based on  \cite[Theorem 6,
Proposition 2]{V1} and recast (with very minor modifications) the
arguments of Section \ref{s4.3}.

Indeed, we look for a solution of \eqref{4.2} in the form
\[
u_{1}=u_{1,1}+u_{1,2}\qquad u_{2}=u_{2,1}+u_{2,2},
\]
where $(u_{1,1}, u_{2,1})$ solves the inhomogeneous transmission
problem
\begin{equation}\label{4.21}
\begin{cases}
\Delta u_{i,1}=f_{0,i},\quad\qquad\qquad y\in G_{i},\quad
t\in(0,T),\quad
i=1,2,\\
\frac{\partial u_{1,1}}{\partial n}-k\frac{\partial
u_{2,1}}{\partial n}=f_{2},\quad\quad y\in g,\quad t\in[0,T]\\
u_{1,1}-u_{2,1}=0,\qquad\quad\quad y\in g,\quad t\in[0,T],\\
u_{1,1}=0,\qquad\qquad\qquad\quad y\in\partial G_{1}\backslash g, \,
t\in[0,T],\\
u_{2,1}=0,\qquad\qquad\qquad\quad y\in\partial G_{2}\backslash g, \,
t\in[0,T],\\
 u_{i,1}(y,0)=0,\qquad\qquad\quad\text{in}\quad \bar{G}_{i},\quad i=1,2,
\end{cases}
\end{equation}
while $(u_{1,2},u_{2,2})$ is a unique classical solution of
\eqref{4.7} with new right-hand side
\[
\bar{f}_{1}=f_{1}+\frac{\partial u_{1,1}}{\partial n}-\frac{\partial
u_{2,1}}{\partial n}.
\]
Utilizing  \cite[Theorem 6 and Proposition 2]{V1} to \eqref{4.21},
we end up with the one-valued classical solvability of \eqref{4.21}
and, besides,
\[
u_{1,1}\in\underset{0}{E}\,_{s+2}^{2+\beta,\beta,\beta}(\bar{G}_{1,T})\quad\text{and}\quad
u_{2,1}\in\underset{0}{E}\,_{s+2}^{2+\beta,\beta,\beta}(\bar{G}_{2,T})
\]
with $s+2\in(2,\min\{3,\pi/\delta\})\backslash\{\pi/2\delta\}$ and
$\delta\in(\pi/4,\pi/2)$.

Hence, we are left to prove the uniqueness and existence of
\eqref{4.7} in unknowns $(u_{1,2},u_{2,2})$. To this end, we only
need to show that the  $\bar{f}_{1}$ satisfies assumption (h8) with
the weight $s$ meeting requirements of (h9). After that, recasting
all arguments of Section \ref{s4.3} completes the proof of Theorem
\ref{t4.2}.

At this point, coming to $\bar{f}_{1}$, we verify each conditions in
(h8) separately.

\noindent$\bullet$ Clearly, the smoothness of $u_{i,1}$, $i=1,2,$
and properties of $f_{1}$ (see (h9)) provides the desired regularity
of $\bar{f}_{1}$. Namely, $\bar{f}_{1}\in
E_{s+1}^{1+\beta,\beta,\beta}(g_{T})$.

\noindent$\bullet$ Assumption (h9) on the right-hand sides in
\eqref{4.21} along with Corollary \ref{c2.1} suggest that
\[
f_{0,i}\in E_{s_{1}}^{\beta,\beta,\beta}(\bar{G}_{i,T}),\quad
f_{2}\in E_{s_{1}+1}^{1+\beta,\beta,\beta}(g_{T})
\]
with any
$s_{1}\in(\max\{-4,-\pi/2\delta-2\},-3)\backslash\{-\pi/2\delta-2\}$.
After that, recasting the arguments leading to  \cite[Theorem
1]{V1}, we conclude  that the obtained solution $(u_{1,1},u_{2,1})$
belongs also $E_{s_{1}+2}^{2+\beta,\beta,\beta}(\bar{G}_{1,T})\times
E_{s_{1}+2}^{2+\beta,\beta,\beta}(\bar{G}_{2,T})$. In particulary,
this means that the solution of \eqref{4.21} together with all its
derivatives vanish as $|y|\to+\infty$ for each $t\in[0,T]$.

Collecting these statements and bearing in mind assumption on
$f_{1}$ (see (h9)), we deduce that the right-hand side meets all
requirements in (h8).\qed


\section{Proof of Theorem \ref{t3.1}}
\label{s5}

\noindent In this section, we will follow the strategy containing
the $4^{\text{th}}$ main steps. First, we demonstrate that, under
the assumptions (h1)-(h4), (h6), the initial distribution of the
pressure $(\mathcal{U}_{1,0},\mathcal{U}_{2,0})$ is a unique
classical solution of \eqref{3.4}. Then, we linearize the nonlinear
system \eqref{3.3} on the initial data and rewrite it in the form
\begin{equation}\label{5.1}
\mathfrak{L}\mathrm{z}=\mathfrak{N}(\mathrm{z})\quad\text{with the
element}\quad\mathrm{ z}=(\mathcal{U}_{1},\mathcal{U}_{2},\sigma).
\end{equation}
Here, $\mathfrak{L}$ denotes  a linear operator, while the symbol
$\mathfrak{N}$ stands for a nonlinear perturbation of \eqref{3.3}.
On the step 3, collecting the continuation approach with the results
of Section \ref{s4}, we claim the boundedness of the linear operator
$\mathfrak{L}^{-1}$ in the corresponding functional space. In fine,
utilizing these outcomes, we rewrite the nonlinear problem
\eqref{5.1} in the form
\[
\mathrm{z}=\mathfrak{L}^{-1}\mathfrak{N}(\mathrm{z})
\]
and prove that the mapping $\mathfrak{L}^{-1}\mathfrak{N}$ is
contraction for sufficiently small $T=T^{*}$. Then, appealing to the
contraction mapping principle ends up with the existence of a unique
fixed point in \eqref{5.1}, that completes the proof of Theorem
\ref{t3.1}.

\subsection{Solvability of \eqref{3.4}}\label{s5.1}
First, we discuss the solvability of the transmission problem having
more general form than \eqref{3.4}. Indeed, for any fixed $T>0$, we
look for the unknown $(W_{1}, W_{2})$ (depending on time $t\in[0,T]$
as a parameter) satisfying the relations
\begin{equation}\label{7.1}
\begin{cases}
\Delta_{x}W_{i}=\phi_{0,i}(x,t)\qquad\qquad\qquad\text{in}\quad \Omega_{i,T},\, i=1,2,\\
W_{i}(x,0)=0\qquad\qquad\qquad\qquad\text{ in}\quad\bar{\Omega}_{i},\, i=1,2,\\
W_{1}-W_{2}=\phi_{1}(x,t)\qquad\qquad\quad\text{on}\quad\Gamma_{T},\\
k_{1}\frac{\partial W_{1}}{\partial n}=k_{2}\frac{\partial
W_{2}}{\partial n}+\phi_{2}(x,t)\qquad\text{ on}\quad\Gamma_{T},\\
W_{1}=\phi_{3}(x,t)\qquad\qquad\qquad\qquad\text{on}\quad\Gamma_{1,T},\\
W_{2}=\phi_{4}(x,t)\qquad\qquad\qquad\qquad\text{on}\quad\Gamma_{2,T}.
\end{cases}
\end{equation}
\begin{lemma}\label{l5.0}
Let $\delta_{0},\delta_{1}\in(0,\pi/2),$ $0<k_{2}<k_{1}$, and let
(h1) and (h2) hold. We assume that consistency condition is
fulfilled, i.e.
\[
\phi_{0,i}(x,0)=0\quad \text{in }\, \bar{\Omega}_{i},\quad
\phi_{1}(x,0),\phi_{2}(x,0)=0\quad\text{on }\,\Gamma
,\quad\phi_{3}(x,0)=0\quad\text{on }\, \Gamma_{1},\quad
\phi_{4}(x,0)=0\quad\text{on }\, \Gamma_{2}.
\]

\noindent(i) If
\begin{align}\label{7.2}\notag
\phi_{0,i}&\in E_{s}^{\beta,\beta,\beta}(\bar{\Omega}_{i,T}),\quad
i=1,2,\quad \phi_{1}\in
E^{2+\beta,\beta,\beta}_{s+2}(\Gamma_{T}),\quad
\phi_{2}\in E^{1+\beta,\beta,\beta}_{s+1}(\Gamma_{T}),\\
\phi_{3}&\in
E^{2+\beta,\beta,\beta}_{s+2}(\Gamma_{1,T}),\quad\phi_{4}\in
E^{2+\beta,\beta,\beta}_{s+2}(\Gamma_{2,T}),
\end{align}
where
$s+2\in(\max\{3;\tfrac{2\pi}{\pi-2\delta_{0}},\tfrac{2\pi}{\pi-2\delta_{1}}\};4)\backslash\{\pi/2\delta_{0},\pi/2\delta_{1}\},$
then problem \eqref{7.1} admits a unique classical solution
satisfying the regularity
\[
W_{1}\in
E_{2+s}^{2+\beta,\beta,\beta}(\bar{\Omega}_{1,T})\qquad\text{and}\qquad
W_{2}\in E_{2+s}^{2+\beta,\beta,\beta}(\bar{\Omega}_{2,T}).
\]
Besides,
\begin{align}\label{7.3}\notag
\sum_{i=1}^{2}\|W_{i}\|_{E_{2+s}^{2+\beta,\beta,\beta}(\bar{\Omega}_{i,T})}&\leq
C\Big[\sum_{i=1}^{2}\|\phi_{0,i}\|_{E_{s}^{\beta,\beta,\beta}(\bar{\Omega}_{i,T})}
+ \|\phi_{1}\|_{E_{2+s}^{2+\beta,\beta,\beta}(\Gamma_{T})} +
\|\phi_{2}\|_{E_{1+s}^{1+\beta,\beta,\beta}(\Gamma_{T})}\\
& + \|\phi_{3}\|_{E_{2+s}^{2+\beta,\beta,\beta}(\Gamma_{1,T})} +
\|\phi_{4}\|_{E_{2+s}^{2+\beta,\beta,\beta}(\Gamma_{2,T})}
 \Big].
\end{align}
If, additionally,
\[
\phi_{0,i},\phi_{1},\phi_{2}\equiv 0,\quad \text{and}\quad
\phi_{3}\in E_{2+s}^{3+\beta,\beta,\beta}(\Gamma_{1,T}),\,\,
\phi_{4}\in E_{2+s}^{3+\beta,\beta,\beta}(\Gamma_{2,T}),
\]
then this solution has a higher regularity, $W_{1}\in
E_{2+s}^{3+\beta,\beta,\beta}(\bar{\Omega}_{1,T}),$ $W_{2}\in
E_{2+s}^{3+\beta,\beta,\beta}(\bar{\Omega}_{2,T}),$ and the estimate
similar to \eqref{7.3} holds.

\noindent(ii) If $\delta_{0},\delta_{1}\in(0,\pi/4)$ and \eqref{7.2}
holds with
\begin{equation}\label{7.3*}
s+2\in(2;\max\{3;\tfrac{2\pi}{\pi-2\delta_{0}},\tfrac{2\pi}{\pi-2\delta_{1}})\backslash\{\pi/2\delta_{0},\pi/2\delta_{1}\},
\end{equation}
then the results stated in (i) hold.

\noindent(iii) If
\[
\phi_{0,i},\phi_{2},\phi_{3},\phi_{4}\equiv 0,\quad \text{and}\quad
\phi_{1}\in E_{2+s}^{2+\beta,\beta,\beta}(\Gamma_{T}),\,\,
\frac{\partial \phi_{1}}{\partial t}\in
E_{1+s}^{1+\beta,\beta,\beta}(\Gamma_{T})
\]
with $s$ satisfying \eqref{7.3*}, then there is the bound
\begin{equation}\label{7.4}
\sum_{i=1}^{2}\underset{\bar{\Omega}_{i,T}}{\sup}\,
r^{-s}(x)|\tfrac{\partial W_{i}}{\partial t}(x,t)|\leq C_{1}T
\underset{\Gamma_{T}}{\sup}\, r^{-s-1}(x)|\tfrac{\partial
\phi_{1}}{\partial t}(x,t)|.
\end{equation}
\end{lemma}
\begin{proof}
The verification of this claim in the case of $\Gamma$ being a
smooth closed curve are carried out in \cite[Proposition 2.3]{V5}.
The proof of this lemma in the case of singular $\partial\Omega_{i}$
is simple consequence of results in \cite{V1} and arguments leading
to \cite[Proposition 2.3]{V5}. Indeed, if assumptions stated in (i)
hold, then \cite[Theorem 6]{V1} provides the one-valued solvability
of \eqref{7.1} in $E_{2+s}^{2+\beta}$ for all $t\in[0,T]$. After
that, computing the left-hand sides of \eqref{7.1} on the difference
$W_{i}(x,t_{1})-W_{i}(x,t_{2})$ with any $t_{1},t_{2}\in[0,T]$, and
exploiting the smoothness of the right-hand sides, we easily arrive
at (i) in this lemma. The point (ii) in this claim is verified with
the similar arguments where we appeal  to \cite[Theorem 6,
Proposition 2]{V1}. Coming to (iii) and differentiating \eqref{7.1}
 with respect to time, we end up with the  transmission problem
 similar to \eqref{7.1} with unknowns $\bar{W}_{i}=\frac{\partial W_{i}}{\partial
 t},$ $i=1,2,$ and the new right-hand side $\frac{\partial \phi_{2}}{\partial t}.$
After that, recasting the arguments leading to \cite[Lemma 4]{V5}
arrives at the desired bound. We notice that, this arguments
utilize the embedding theorem and the solvability of \eqref{7.1}
 in the weighted Sobolev spaces $W_{\gamma}^{2,p}$ with some
 $\gamma\in(0,1)$ and $p>1$.
\end{proof}
\begin{remark}\label{r5.0*}
It is apparent that statements in (i) and (ii) of Lemma \ref{l5.0}
hold if the right-hand sides are time-independent. In this case, the
constructed solution belongs to $E_{s+2}^{l+\beta}$ with $l=2$ or
$3$.
\end{remark}

Exploiting Lemma \ref{l5.0} and Remark \ref{r5.0*} to problem
\eqref{3.4}, we  achieve to the following results.
\begin{lemma}\label{l5.1}
Let assumptions (h1)-(h3) and (h6) hold, $k\in(0,1),$
$\delta_{0},\delta_{1}\in(0,\pi/4)$.
Then transmission problem \eqref{3.4} admits a unique classical
solution $(\mathcal{U}_{1,0},\mathcal{U}_{2,0})$ having regularity
\[
\mathcal{U}_{1,0}\in E_{s^{*}}^{3+\beta}(\bar{\Omega}_{1}),\quad
\mathcal{U}_{2,0}\in E_{s^{*}}^{3+\beta}(\bar{\Omega}_{2}).
\]
Besides,
\[
\|\mathcal{U}_{1,0}\|_{ E_{s^{*}}^{3+\beta}(\bar{\Omega}_{1})}+
\|\mathcal{U}_{2,0}\|_{ E_{s^{*}}^{3+\beta}(\bar{\Omega}_{2})} \leq
C(\|\mathrm{p}_{1}\|_{ E_{s^{*}}^{3+\beta}(\bar{\Gamma}_{1})} +
\|\mathrm{p}_{2}\|_{ E_{s^{*}}^{3+\beta}(\bar{\Gamma}_{2})}).
\]
\end{lemma}
In the following section, we will need in the new function
$\rho(\omega,t)\in\C([0,T],E^{2+\beta}_{s^{*}-1}(\Gamma))$
constructed by the initial data.
\begin{corollary}\label{c5.1}
Let assumptions of Lemma \ref{l5.1} hold, then the function
\[
\rho=\rho(\omega,t)=t\frac{\partial\s(\omega,0)}{\partial
t}\Big|_{\Gamma}
\]
satisfies the equalities
\begin{equation}\label{5.4}
\rho(\omega,0)=0\qquad\text{and}\quad\frac{\partial\rho}{\partial
t}(\omega,0)=\frac{\partial\s}{\partial
t}(\omega,0)\quad\text{on}\quad\Gamma.
\end{equation}
Besides, the inequalities are fulfilled
\[
\sum_{l=0}^{2}\|\tfrac{\partial^{l}\rho}{\partial
t^{l}}\|_{\C([0,T],E^{2+\beta}_{s^{*}-1}(\Gamma))}\leq
C(1+T)\sum_{i=1}^{2}\|\mathcal{U}_{i,0}\|_{E^{3+\beta}_{s^{*}}(\bar{\Omega}_{i})}\leq
C(1+T)\sum_{i=1}^{2}\|\mathrm{p}_{i}\|_{E^{3+\beta}_{s^{*}}(\bar{\Gamma}_{i})}.
\]
\end{corollary}
\begin{proof}
It is apparent that equalities in \eqref{5.4} are simple consequence
of the representation of $\rho(\omega,t)$. Thus, we are left only to
verify the smoothness of $\rho(\omega,t)$ providing with the
estimate in this claim. To this end, we notice that the boundary
condition on $\Gamma_{T}$ in \eqref{3.3} arrives at the equalities
\begin{align}\label{5.5}\notag
\tfrac{\partial\s}{\partial
t}(\omega,0)\Big|_{\Gamma}&=-k_{1}[\S(\omega,0,0,0)\tfrac{\partial\mathcal{U}_{1,0}}{\partial\lambda}
+\S_{1}(\omega,0,0,0)\tfrac{\partial\mathcal{U}_{1,0}}{\partial\omega}]\\
& =
-k_{2}[\S(\omega,0,0,0)\tfrac{\partial\mathcal{U}_{2,0}}{\partial\lambda}
+\S_{1}(\omega,0,0,0)\tfrac{\partial\mathcal{U}_{2,0}}{\partial\omega}].
\end{align}
It is worth noting that the explicit forms of the function $\S$ and
$\S_{1}$ beyond
$\overline{\Gamma\cap\mathcal{O}_{2\varepsilon}(\mathbf{A}_{i})},$
$i=0,1,$ for each $t\in[0,T],$ are given in \eqref{3.12}, which tell
that these functions belong to $\C^{2+\beta}$. Moreover, in the
proof of Lemma \ref{l5.2}, we will demonstrate that this regularity
holds in
$(\overline{\Gamma\cap\mathcal{O}_{2\varepsilon}(\mathbf{A}_{i})})_{T}.$

In fine, collecting \eqref{5.5} with Lemma \ref{l5.1}, we
immediately end up with the desired estimate which completes the
proof of this claim.
\end{proof}

\subsection{A perturbation form of system \eqref{3.3}}\label{s5.2}
To linearize system \eqref{3.3}, we introduce the new unknown
functions
\begin{equation}\label{5.6}
\sigma=\sigma(\omega,t)=\s(\omega,t)-\rho(\omega,t),\quad
U_{i}=U_{i}(x,t)=\mathcal{U}_{i}(x,t)-\mathcal{U}_{i,0}(x)-\langle\nabla_{x}\mathcal{U}_{i,0},\mathbf{e}\rangle,\quad
i=1,2
\end{equation}
with
\[
\mathbf{e}=(\tfrac{\partial
x_{1}}{\partial\lambda}\chi(\lambda)\sigma(\omega,t);\tfrac{\partial
x_{2}}{\partial\lambda}\chi(\lambda)\sigma(\omega,t)),
\]
and rewrite system \eqref{3.3} in the form
\begin{equation}\label{5.7}
\begin{cases}
\Delta_{x}U_{1}=\mathfrak{N}_{0,1}(U_{1},\sigma)\qquad\qquad\text{in}\quad\Omega_{1,T},\\
\Delta_{x}U_{2}=\mathfrak{N}_{0,2}(U_{2},\sigma)\qquad\qquad\text{in}\quad\Omega_{2,T},\\
U_{1}-U_{2}=A_{0}(x)\sigma\qquad\qquad\quad\text{on}\quad\Gamma_{T},\\
\frac{\partial\sigma}{\partial t}=A_{1}(x)[\tfrac{\partial
U_{1}}{\partial n}-\tfrac{\partial U_{2}}{\partial
n}]+A_{2}(x)[\tfrac{\partial U_{1}}{\partial \omega}-\tfrac{\partial
U_{2}}{\partial \omega}]+\mathfrak{N}_{1}(U_{1},U_{2},\sigma)\quad\text{on}\quad\Gamma_{T},\\
\tfrac{\partial U_{1}}{\partial n}-k\tfrac{\partial U_{2}}{\partial
n}-kA_{3}(x)[\tfrac{\partial U_{1}}{\partial \omega}-\tfrac{\partial
U_{2}}{\partial \omega}]=\mathfrak{N}_{2}(U_{1},U_{2},\sigma)\qquad\qquad\quad\text{ on}\quad\Gamma_{T},\\
U_{1}=\mathfrak{N}_{3}(\sigma)\qquad\qquad\text{on}\quad\Gamma_{1,T},\\
U_{2}=\mathfrak{N}_{4}(\sigma)\qquad\qquad\text{on}\quad\Gamma_{2,T},\\
\sigma(\omega,0)=0\qquad\qquad\qquad\text{in}\quad\Gamma.
\end{cases}
\end{equation}
Comparing system \eqref{5.7} with \eqref{5.1}, we easily conclude
that the linear operator $\mathfrak{L}$ is constructed by the
left-hand side of \eqref{5.7}, while the nonlinear operator
$\mathfrak{N}$ is defined by $\mathfrak{N}_{0,i},$ $i=1,2,$
$\mathfrak{N}_{j},$ $j=1,2,3,4$ (other words by the right-hand side
of \eqref{5.7}).

We notice that, since we discuss the classical solution in
\eqref{5.7}, then collecting the last condition in \eqref{5.7} with
the representation \eqref{5.6} immediately provides
\[
U_{i}(x,0)=0\quad\text{in}\quad \bar{\Omega}_{i}.
\]

The properties of the coefficients $A_{l}(x),$ and the functions
$\mathfrak{N}_{0,i},$  $\mathfrak{N}_{j}$ are stated in the
following claim.
\begin{lemma}\label{l5.2}
Let assumptions (h1)--(h7) hold, and $\mathcal{U}_{i,0}$ and $\rho$
be described in Lemma \ref{l5.1} and Corollary \ref{c5.1},
respectively. If $U_{i}\in
E_{s+2}^{2+\beta,\beta,\beta}(\bar{\Omega}_{i,T})$ and
$\sigma\in\underset{0}{\mathcal{E}}^{2+\beta,\beta,\beta}_{s+2,s^{*}-1}(\Gamma_{T})$,
then there are the following relations:

\noindent(i) $A_{1}(x)$ and $-A_{0}(x)$ are  positive functions, and
$A_{0}\in E_{s^{*}-1}^{2+\beta}(\Gamma),$
$A_{j}\in\C^{2+\beta}(\Gamma),$ $j=1,2,3.$ Besides,
\begin{align*}
A_{0}(x)&=
\begin{cases}
\tfrac{1-k}{k}[1+(\varphi_{0}')^{2}]^{-1/2}\tfrac{\partial\mathcal{U}_{1,0}}{\partial
n}\qquad\qquad\text{on}\quad\Gamma\cap\mathcal{O}_{\varepsilon}(\mathbf{A}_{0}),\\
\tfrac{1-k}{k}[1+(\varphi_{1}')^{2}]^{-1/2}\tfrac{\partial\mathcal{U}_{1,0}}{\partial
n}\qquad\qquad\text{on}\quad\Gamma\cap\mathcal{O}_{\varepsilon}(\mathbf{A}_{1
}),\\
\tfrac{1-k}{k}\tfrac{\partial\mathcal{U}_{1,0}}{\partial
n}\qquad\qquad\qquad\qquad\qquad \text{
on}\quad\Gamma\backslash\{(\Gamma\cap\mathcal{O}_{2\varepsilon}(\mathbf{A}_{1
}))\cup(\Gamma\cap\mathcal{O}_{2\varepsilon}(\mathbf{A}_{0}))\},
\end{cases}\\
A_{1}(x)&=
\begin{cases}
\tfrac{k_{2}}{1-k}[1+(\varphi_{0}')^{2}]^{1/2}\qquad\qquad\qquad\quad\text{on}\quad\Gamma\cap\mathcal{O}_{\varepsilon}(\mathbf{A}_{0}),\\
\tfrac{k_{2}}{1-k}[1+(\varphi_{1}')^{2}]^{1/2}\qquad\qquad\qquad\quad\text{on}\quad\Gamma\cap\mathcal{O}_{\varepsilon}(\mathbf{A}_{1}),\\
\tfrac{k_{2}}{1-k}\qquad\qquad\qquad\qquad\qquad\quad\qquad\text{on}\quad\Gamma\backslash\{(\Gamma\cap\mathcal{O}_{2\varepsilon}(\mathbf{A}_{1
}))\cup(\Gamma\cap\mathcal{O}_{2\varepsilon}(\mathbf{A}_{0}))\},
\end{cases}\\
A_{2}(x)&=
\begin{cases}
\tfrac{k_{2}}{1-k}\varphi_{0}'[1+(\varphi_{0}')^{2}]^{1/2}\qquad\qquad\qquad\text{on}\quad\Gamma\cap\mathcal{O}_{\varepsilon}(\mathbf{A}_{0}),\\
-\tfrac{k_{2}}{1-k}\varphi_{1}'[1+(\varphi_{1}')^{2}]^{1/2}\qquad\quad\qquad\text{on}\quad\Gamma\cap\mathcal{O}_{\varepsilon}(\mathbf{A}_{1}),\\
0\qquad\qquad\qquad\qquad\qquad\qquad\qquad\text{on}\quad\Gamma\backslash\{(\Gamma\cap\mathcal{O}_{2\varepsilon}(\mathbf{A}_{1
}))\cup(\Gamma\cap\mathcal{O}_{2\varepsilon}(\mathbf{A}_{0}))\},
\end{cases}\\
A_{3}(x)&=
\begin{cases}
\tfrac{\partial\mathcal{U}_{1,0}}
{\partial r_{0}}/\tfrac{\partial\mathcal{U}_{1,0}}{\partial n}\qquad\qquad\qquad\qquad\qquad\text{on}\quad\Gamma\cap\mathcal{O}_{\varepsilon}(\mathbf{A}_{0}),\\
-\tfrac{\partial\mathcal{U}_{1,0}}{\partial
r_{1}}/\tfrac{\partial\mathcal{U}_{1,0}}{\partial
n}\qquad\quad\qquad\qquad\qquad\text{on}
\quad\Gamma\cap\mathcal{O}_{\varepsilon}(\mathbf{A}_{1}),\\
(k-1)^{-1}\tfrac{\partial\S_{1}(\omega,0,0)}{\partial \rho_{\omega}}
\tfrac{\partial\mathcal{U}_{1,0}}{\partial
\omega}/\tfrac{\partial\mathcal{U}_{1,0}}{\partial n}
\qquad\text{on}\quad\Gamma\backslash\{(\Gamma\cap\mathcal{O}_{2\varepsilon}(\mathbf{A}_{1
}))\cup(\Upsilon\cap\mathcal{O}_{2\varepsilon}(\mathbf{A}_{0}))\},
\end{cases}
\end{align*}
where $r_{0}=r_{0}(y)$ and $r_{1}=r_{1}(y)$.

 \noindent(ii) The following equalities hold:
\[
\mathfrak{N}_{0,i}(U_{i},\sigma)|_{t=0}=0,
\quad\mathfrak{N}_{i}(U_{1},U_{2},\sigma)|_{t=0}=0,\,
i=1,2,\quad\mathfrak{N}_{j}(\sigma)|_{t=0}=0,\, j=3,4.
\]
\noindent(iii) The functions $ \mathfrak{N}_{0,i}(U_{i},\sigma)$,
$\mathfrak{N}_{i}(U_{1},U_{2},\sigma)$, $i=1,2$,
$\mathfrak{N}_{j}(\sigma),$ $j=3,4,$ consist in the following terms:

\noindent$\bullet$ higher derivatives of $U_{i},$ and $\sigma$ with
coefficients vanishing as $t$ goes to zero,

\noindent$\bullet$ the "quadratic" terms with respect to
$U_{i},\sigma$ and their corresponding derivatives,

\noindent$\bullet$  minor derivatives of unknown functions.

\noindent(iv) The inequality holds
\begin{align*}
&\|\mathfrak{N}_{0,1}(0,0)\|_{E_{s}^{\beta,\beta,\beta}(\bar{\Omega}_{1,T})}
+\|\mathfrak{N}_{0,2}(0,0)\|_{E_{s}^{\beta,\beta,\beta}(\bar{\Omega}_{2,T})}
+\|\mathfrak{N}_{1}(0,0,0)\|_{E_{s+1}^{1+\beta,\beta,\beta}(\Gamma_{T})}\\
&+
\|\mathfrak{N}_{2}(0,0,0)\|_{E_{s+1}^{1+\beta,\beta,\beta}(\Gamma_{T})}
+\|\mathfrak{N}_{3}(0)\|_{E_{s+2}^{2+\beta,\beta,\beta}(\bar{\Gamma}_{1,T})}+
\|\mathfrak{N}_{4}(0)\|_{E_{s+2}^{2+\beta,\beta,\beta}(\bar{\Gamma}_{2,T})}
\\&\leq C
T^{1-\beta}\Big[\|\rho\|_{C^{2}([0,T],E^{2+\beta}_{s^{*}-1}(\Gamma))}+\|\mathcal{U}_{1,0}\|_{E_{s^{*}}^{3+\beta}(\bar{\Omega}_{1})}
+\|\mathcal{U}_{2,0}\|_{E_{s^{*}}^{3+\beta}(\bar{\Omega}_{2})}\Big
].
\end{align*}
\end{lemma}
\begin{proof}
It is worth noting that if
$x\in\bar{\Omega}\backslash\{\overline{\Omega\cap\mathcal{O}_{\varepsilon}(\mathbf{A}_{0})}
\cup\overline{\Omega\cap\mathcal{O}_{\varepsilon}(\mathbf{A}_{1})}\}$,
then this claim follows from the arguments of \cite[Sections
3.1-3.2]{V5}, where the explicit forms of the coefficients and the
right-hand sides in \eqref{5.7} are given if $\Gamma$ is a  smooth
closed curve. Thus, we are left to verify this lemma only in the
case of
$x\in\overline{\Omega\cap\mathcal{O}_{\varepsilon}(\mathbf{A}_{0})}$
 $
\cup\overline{\Omega\cap\mathcal{O}_{\varepsilon}(\mathbf{A}_{1})}$.
Here, we provide the detailed proof if $x$ is of  the
$\varepsilon-$neighborhood of the corner point $\mathbf{A}_{0}$. The
 remaining case is analyzed with the similar arguments and left to
 the interested readers. In the further consideration, we utilized
 the following strategy. On the first step, appealing to assumption
 (h3) and recasting the arguments of Section \ref{s3.1}, we obtain
 the explicit forms of $A_{l}(x)$ and $\mathfrak{N}_{0,i},$
 $\mathfrak{N}_{j}$ in the $\varepsilon-$ neighborhood of $\mathbf{A}_{0}$. After that, utilizing Lemma \ref{l5.1} and Corollary
 \ref{c5.1} ends up with the desired statements.
Throughout this proof, we consider that $x$ or $y$ belongs to the
set
$\overline{\Omega\cap\mathcal{O}_{\varepsilon}(\mathbf{A}_{0})}$.

Coming to the change of variables \eqref{3.2}, in this particular
case, it has a simple form
\[
\omega=\tfrac{x_{1}}{\sin\delta_{0}}=r_{0},\qquad
\lambda=x_{2}-\varphi_{0}(x_{1})
\]
and accordingly,
\[
\begin{cases}
y_{1}=x_{1},\\
y_{2}=x_{2}-\chi(\lambda)\s(\omega,t).
\end{cases}
\]
Differentiating these relations with respect to $y_{1}$ and $y_{2}$,
we arrive at
\[
\begin{cases}
\tfrac{\partial x_{1}}{\partial y_{1}}=1,\qquad\qquad
\tfrac{\partial
x_{1}}{\partial y_{2}}=0,\\
\tfrac{\partial x_{2}}{\partial
y_{1}}=\tfrac{\chi\s_{x_{1}}-\chi'\s}{1-\chi'\s},\quad\tfrac{\partial
x_{2}}{\partial y_{2}}=\tfrac{1}{1-\chi'\s}.
\end{cases}
\]
Here, we note that the restriction on $\s$ (\eqref{3.0*},
\eqref{3.1*} and \eqref{3.10}) ensures the well-posedness of these
relations.

Differentiating once more the last two equalities, we have
\[
\begin{cases}
\tfrac{\partial^{2}x_{2}}{\partial
y_{1}^{2}}=\tfrac{\chi\s_{x_{1}x_{1}}-\varphi'_{0}\chi'\s_{x_{1}}+\varphi'_{0}\chi''\s-\chi'\s_{x_{1}}}{1-\chi'\s}
+\tfrac{[\chi'\s_{x_{1}}-\chi''\s][\chi\s_{x_{1}}-\chi'\s]}{[1-\chi'\s]^{2}}+\tfrac{\chi''\s[\chi\s_{x_{1}}-\chi'\s]^{2}}{[1-\chi's]^{3}}
\\
\qquad-\tfrac{[\chi\s_{x_{1}}-\chi'\s][\varphi'_{0}\chi''\s-\chi'\s_{x_{1}}]}{[1-\chi'\s]^{2}},\\
\tfrac{\partial^{2}x_{2}}{\partial
y_{2}^{2}}=\tfrac{\chi''\s}{[1-\chi'\s]^{3}}.
\end{cases}
\]
Performing standard calculations and taking into account these
relations, we get the following representations to $\nabla^{2}_{\s}$
and the vector of the normal in the neighborhood of
$\mathbf{A}_{0}$:
\begin{align*}
\nabla_{\s}^{2}&=\Delta_{x}+2\frac{\partial x_{1}}{\partial
y_{1}}\frac{\partial^{2}}{\partial x_{1}\partial
x_{2}}+\Big[\Big(\frac{\partial x_{2}}{\partial
y_{1}}\Big)^{2}+\Big(\frac{\partial x_{2}}{\partial
y_{2}}\Big)^{2}-1\Big]\frac{\partial^{2}}{\partial x_{2}^{2}}
+\Big(\frac{\partial^{2}x_{2}}{\partial
y_{1}^{2}}+\frac{\partial^{2}x_{2}}{\partial
y_{2}^{2}}\Big)\frac{\partial}{\partial x_{2}},\\
n_{\tau}&=\Big([\varphi'_{0}-\s_{x_{1}}][(1+(\varphi'_{0}-\s_{x_{1}})^{2}]^{-1/2};-(1+[\varphi'_{0}-\s_{x_{1}}]^{2})^{-1/2}\Big).
\end{align*}
Besides, appealing to the definition of $\Phi_{\s}$, we have
\[
\Phi_{\s}=-y_{2}+\varphi_{0}(y_{1})-\s(x_{1},t),
\]
and, accordingly, we easily conclude that
\[
\nabla_{\s}\Phi_{\s}=(\varphi'_{0}(x_{1})-\s_{x_{1}};-1).
\]
At this point, setting
\[
\widetilde{\mathfrak{N}}_{0,i}(\s,U_{i})=-2\frac{\partial
x_{1}}{\partial y_{1}}\frac{\partial^{2}U_{i}}{\partial
x_{1}\partial x_{2}}-\Big[\Big(\frac{\partial x_{2}}{\partial
y_{1}}\Big)^{2}+\Big(\frac{\partial x_{2}}{\partial
y_{2}}\Big)^{2}-1\Big]\frac{\partial^{2}U_{i}}{\partial x_{2}^{2}} -
\Big(\frac{\partial^{2}x_{2}}{\partial
y_{1}^{2}}+\frac{\partial^{2}x_{2}}{\partial
y_{2}^{2}}\Big)\frac{\partial U_{i}}{\partial x_{2}},
\]
and appealing to the equalities above, we rewrite \eqref{3.3} in the
neighborhood of $\mathbf{A}_{0}$
\begin{equation}\label{5.8}
\begin{cases}
\Delta_{x}
U_{i}=\widetilde{\mathfrak{N}}_{0,i}(\s,U_{i})\quad\qquad\qquad\text{in}\quad(\Omega_{i}\cap\mathcal{O}_{\varepsilon}(\mathbf{A}_{0}))_{T},\,
i=1,2,\\
U_{1}-U_{2}=0\qquad\qquad\qquad\qquad\text{on}\quad\overline{(\Gamma\cap\mathcal{O}_{\varepsilon}(\mathbf{A}_{0}))}_{T},\\
-\frac{\partial\s}{\partial
t}=k_{1}\sqrt{1+(\varphi'_{0})^{2}}\frac{\partial U_{1}}{\partial
n}-k_{1}\frac{\partial\s}{\partial r_{0}}\Big[\frac{\partial
U_{1}}{\partial x_{1}}+\frac{\partial\s}{\partial
x_{1}}\frac{\partial U_{1}}{\partial x_{2}}\Big] \qquad\qquad\qquad\text{on}\quad\overline{(\Gamma\cap\mathcal{O}_{\varepsilon}(\mathbf{A}_{0}))}_{T},\\
\frac{\partial U_{1}}{\partial n}-k\frac{\partial U_{2}}{\partial
n}-\frac{\partial\s}{\partial r_{0}}\Big[\frac{\partial
U_{1}}{\partial x_{1}}-k\frac{\partial U_{2}}{\partial
x_{1}}+\frac{\partial\s}{\partial x_{1}}\Big(\frac{\partial
U_{1}}{\partial x_{2}}-k\frac{\partial U_{2}}{\partial
x_{2}}\Big)\Big]=0 \qquad\text{on}\quad\overline{(\Gamma\cap\mathcal{O}_{\varepsilon}(\mathbf{A}_{0}))}_{T},\\
U_{i}=\mathrm{p}_{i}(x_{1},x_{2}-\chi\s)\qquad\text{ on}\qquad
\overline{(\Gamma_{i}\cap
\mathcal{O}_{\varepsilon}(\mathbf{A}_{0}))}_{T},\quad i=1,2,\\
\s(x_{1},0)=0\qquad\qquad\qquad\text{in}\quad
\overline{(\Gamma\cap\mathcal{O}_{\varepsilon}(\mathbf{A}_{0}))}.
\end{cases}
\end{equation}
Here, we utilized the easily verified relations in
$\overline{\Omega\cap\mathcal{O}_{\varepsilon}(\mathbf{A}_{0})}$:
\begin{align*}
\frac{\partial\s}{\partial
x_{1}}&=\sqrt{1+(\varphi'_{0})^{2}}\frac{\partial\s}{\partial
r_{0}}(r_{0},t),\\
\frac{\partial}{\partial
x_{1}}&=\frac{1}{\sqrt{1+(\varphi'_{0})^{2}}}\frac{\partial}{\partial
r_{0}}+\frac{\varphi'_{0}}{\sqrt{1+(\varphi'_{0})^{2}}}\frac{\partial}{\partial
n},\\
\frac{\partial}{\partial
x_{2}}&=\frac{\varphi'_{0}}{\sqrt{1+(\varphi'_{0})^{2}}}\frac{\partial}{\partial
r_{0}}-\frac{1}{\sqrt{1+(\varphi'_{0})^{2}}}\frac{\partial}{\partial
n}.
\end{align*}
Next, bearing in mind these equalities, we arrive at
\begin{equation*}\label{5.9}
\begin{cases}
\langle\nabla_{x}\mathcal{U}_{1,0},\mathbf{e}\rangle-\langle\nabla_{x}\mathcal{U}_{2,0},\mathbf{e}\rangle
=\frac{1-k}{k\sqrt{1+(\varphi'_{0})^{2}}}\frac{\partial
\mathcal{U}_{1,0}}{\partial n}\sigma,\\
\frac{\partial \mathcal{U}_{1,0}}{\partial x_{1}}-k\frac{\partial
\mathcal{U}_{1,0}}{\partial
x_{1}}=\frac{1-k}{\sqrt{1+(\varphi'_{0})^{2}}}\frac{\partial
\mathcal{U}_{1,0}}{\partial r_{0}}.
\end{cases}
\end{equation*}
In fine, substituting \eqref{5.6} to \eqref{5.8}, we end up with the
system
\[
\begin{cases}
\Delta
U_{i}=\mathfrak{N}_{0,i}(U_{i},\sigma)\qquad\qquad\qquad\text{in}\quad
(\Omega_{i}\cap\mathcal{O}_{\varepsilon}(\mathbf{A}_{0}))_{T},\quad
i=1,2,\\
U_{1}-U_{2}=\frac{1-k}{k\sqrt{1+(\varphi'_{0})^{2}}}\frac{\partial\mathcal{U}_{1,0}}{\partial
n}\sigma\qquad\text{on}\qquad
\overline{(\Gamma\cap\mathcal{O}_{\varepsilon}(\mathbf{A}_{0}))}_{T},\\
\frac{\partial\sigma}{\partial
t}=-k_{1}\sqrt{1+(\varphi'_{0})^{2}}\frac{\partial U_{1}}{\partial
n}+k_{1}\frac{\partial\sigma}{\partial
r_{0}}\Big[\frac{\partial\mathcal{U}_{1,0}}{\partial
r_{0}}+\varphi'_{0}\frac{\partial\mathcal{U}_{1,0}}{\partial
n}\Big]+\widetilde{\mathfrak{N}}_{1}(U_{1},U_{2},\sigma)\qquad\text{on}\qquad
\overline{(\Gamma\cap\mathcal{O}_{\varepsilon}(\mathbf{A}_{0}))}_{T},\\
\frac{\partial U_{1}}{\partial n}-k\frac{\partial U_{2}}{\partial
n}-\frac{1-k}{\sqrt{1+(\varphi'_{0})^{2}}}\frac{\partial\mathcal{U}_{1,0}}{\partial
r_{0}}\frac{\partial \sigma}{\partial
r_{0}}+\widetilde{\mathfrak{N}}_{2}(U_{1},U_{2},\sigma)=0\qquad\qquad\qquad\qquad\text{on}\qquad
\overline{(\Gamma\cap\mathcal{O}_{\varepsilon}(\mathbf{A}_{0}))}_{T},\\
U_{1}=\mathfrak{N}_{3}(\sigma)\quad\text{on}\quad\overline{(\Gamma_{1}\cap\mathcal{O}_{\varepsilon}(\mathbf{A}_{0}))},\qquad
U_{2}=\mathfrak{N}_{4}(\sigma)\quad\text{on}\quad\overline{(\Gamma_{2}\cap\mathcal{O}_{\varepsilon}(\mathbf{A}_{0}))},\\
\sigma(r_{0},0)=0\quad\text{on}\quad\overline{(\Gamma\cap\mathcal{O}_{\varepsilon}(\mathbf{A}_{0}))},\qquad
U_{i}(x,0)=0\quad\text{in}\quad\overline{(\Omega_{i}\cap\mathcal{O}_{\varepsilon}(\mathbf{A}_{0}))},\quad
i=1,2,
\end{cases}
\]
where we set
\begin{align*}
&\mathfrak{N}_{0,i}(U_{i},\sigma)=\widetilde{\mathfrak{N}}_{0,i}\Big(\sigma+\rho,
U_{i}+\mathcal{U}_{i,0}-\frac{\partial\mathcal{U}_{i,0}}{\partial
x_{2}}\sigma\Big)+\chi\sigma_{x_1x_1}\frac{\partial\mathcal{U}_{i,0}}{\partial
x_{2}}\\
&+\sigma\Delta_{x}\Big(\chi
\frac{\partial\mathcal{U}_{i,0}}{\partial
x_{2}}\Big)+2\frac{\partial\sigma}{\partial
x_{1}}\frac{\partial}{\partial x_{1}}\Big(\chi
\frac{\partial\mathcal{U}_{i,0}}{\partial x_{2}}\Big),\qquad
i=1,2,\\
&\widetilde{\mathfrak{N}}_{1}(U_{1},U_{2},\sigma)=k_{1}\sqrt{1+(\varphi'_{0})^{2}}\frac{\partial^{2}\mathcal{U}_{1,0}}{\partial
n\partial x_{2}}\sigma+
k_{1}\sqrt{1+(\varphi'_{0})^{2}}\frac{\partial\mathcal{U}_{1,0}}{\partial
x_{1}}\frac{\partial\rho}{\partial r_{0}}\\
& +k_{1}\sqrt{1+(\varphi'_{0})^{2}}\Big[\frac{\partial\rho}{\partial
r_{0}}+\frac{\partial\sigma}{\partial r_{0}}\Big]\Big[
\frac{\partial
U_{1}}{x_{1}}-\frac{\partial^{2}\mathcal{U}_{1,0}}{\partial
x_{1}\partial x_{2}}\sigma-\frac{\partial\mathcal{U}_{1,0}}{\partial
x_{2}}\frac{\partial\sigma}{\partial x_{1}}\\
& + \Big(\frac{\partial\rho}{\partial
x_{1}}+\frac{\partial\sigma}{\partial x_{1}}\Big)\Big(
\frac{\partial U_{1}}{\partial
x_{2}}+\frac{\partial\mathcal{U}_{1,0}}{\partial
x_{2}}-\frac{\partial\mathcal{U}_{1,0}}{\partial x_{2}}\sigma \Big)
 \Big],\\
&\widetilde{
\mathfrak{N}}_{2}(U_{1},U_{2},\sigma)=-\sigma\Big[\frac{\partial^{2}\mathcal{U}_{1,0}}{\partial
x_{2}\partial n}-k\frac{\partial^{2}\mathcal{U}_{2,0}}{\partial
x_{2}\partial n}\Big]+\Big[\frac{\partial\rho}{\partial
r_{0}}+\frac{\partial\sigma}{\partial r_{0}}\Big]\Big[
\frac{\partial U_{1}}{\partial x_{1}}-k\frac{\partial
U_{2}}{\partial x_{1}}
-\frac{\partial^{2}\mathcal{U}_{1,0}}{\partial x_{2}\partial
r_{0}}\frac{(1-k)\sigma}{\sqrt{1+(\varphi'_{0})^{2}}}\\&
+\Big[\frac{\partial\rho}{\partial
x_{1}}+\frac{\partial\sigma}{\partial x_{1}}\Big]\Big[\frac{\partial
U_{1}}{x_{2}}-k\frac{\partial U_{2}}{\partial
x_{2}}+\frac{(1-k)\varphi'_{0}}{\sqrt{1+(\varphi'_{0})^{2}}}\frac{\partial
\mathcal{U}_{1,0}}{\partial
r_{0}}-\sigma\Big(\frac{\partial^{2}\mathcal{U}_{1,0}}{\partial
x_{2}^{2}}-k\frac{\partial^{2}\mathcal{U}_{2,0}}{\partial
x_{2}^{2}}\Big)\Big]
 \Big];\\
&
\mathfrak{N}_{3}(\sigma)=\mathrm{p}_{1}(x_{1},x_{2}-\chi(\rho+\sigma))-\mathrm{p}_{1}(x_{1},x_{2})+
\chi\sigma\frac{\partial\mathrm{p}_{1}}{\partial
x_{2}}(x_{1},x_{2}),\\
&
\mathfrak{N}_{4}(\sigma)=\mathrm{p}_{2}(x_{1},x_{2}-\chi(\rho+\sigma))-\mathrm{p}_{2}(x_{1},x_{2})+
\chi\sigma\frac{\partial\mathrm{p}_{2}}{\partial
x_{2}}(x_{1},x_{2}).
\end{align*}
To obtain the terms $\mathfrak{N}_{3}$ and $\mathfrak{N}_{4}$, we
appealed that $(\mathcal{U}_{1,0},\mathcal{U}_{2,0})$ is a classical
solution of \eqref{3.4}.

At last, differentiating the first condition on
$\overline{(\Gamma\cap\mathcal{O}_{\varepsilon}(\mathbf{A}_{0}))}_{T}$
with respect to $r_{0}$, we deduce that
\[
\frac{\partial\sigma}{\partial r_{0}}=
\frac{k\sqrt{1+(\varphi'_{0})^{2}}}
{(1-k)\frac{\partial\mathcal{U}_{1,0}}{\partial n}}
\frac{\partial}{\partial r_{0}}(U_{1}-U_{2}) -\sigma
\frac{\sqrt{1+(\varphi'_{0})^{2}}}{\frac{\partial\mathcal{U}_{1,0}}{\partial
n}} \frac{\partial}{\partial r_{0}}
\Big(\frac{\partial\mathcal{U}_{1,0}}{\partial n}
(1+(\varphi'_{0})^{2})^{-1/2}\Big)\frac{k}{1-k}.
\]
Then, taking into account these equalities and performing technical
calculations, we end up with the system
\begin{equation}\label{5.10*}
\begin{cases}
\Delta
U_{i}=\mathfrak{N}_{0,i}(U_{i},\sigma)\qquad\qquad\qquad\text{in}\quad
(\Omega_{i}\cap\mathcal{O}_{\varepsilon}(\mathbf{A}_{0}))_{T},\quad
i=1,2,\\
U_{1}-U_{2}=\frac{1-k}{k\sqrt{1+(\varphi'_{0})^{2}}}\frac{\partial\mathcal{U}_{1,0}}{\partial
n}\sigma\qquad\text{on}\qquad
\overline{(\Gamma\cap\mathcal{O}_{\varepsilon}(\mathbf{A}_{0}))}_{T},\\
\frac{\partial\sigma}{\partial
t}=\frac{k_{2}\sqrt{1+(\varphi'_{0})^{2}}}{1-k}\Big[\frac{\partial
U_{1}}{\partial n}- \frac{\partial U_{2}}{\partial n}\Big] +
\frac{k_{2}\varphi'_{0}\sqrt{1+(\varphi'_{0})^{2}}}{1-k}\Big[\frac{\partial
U_{1}}{\partial r_{0}}- \frac{\partial U_{2}}{\partial r_{0}}\Big]
+\mathfrak{N}_{1}(U_{1},U_{2},\sigma) \qquad\text{on}\qquad
\overline{(\Gamma\cap\mathcal{O}_{\varepsilon}(\mathbf{A}_{0}))}_{T},\\
\frac{\partial U_{1}}{\partial n}-k\frac{\partial U_{2}}{\partial
n}-k\frac{\frac{\partial\mathcal{U}_{1,0}}{\partial
r_{0}}}{\frac{\partial\mathcal{U}_{1,0}}{\partial
n}}\Big[\frac{\partial U_{1}}{\partial r_{0}}-\frac{\partial
U_{2}}{\partial
r_{0}}\Big]=\mathfrak{N}_{2}(U_{1},U_{2},\sigma)\qquad\qquad\qquad\qquad\,
\qquad\qquad\text{on}\qquad
\overline{(\Gamma\cap\mathcal{O}_{\varepsilon}(\mathbf{A}_{0}))}_{T},\\
U_{1}=\mathfrak{N}_{3}(\sigma)\quad\text{on}\quad\overline{(\Gamma_{1}\cap\mathcal{O}_{\varepsilon}(\mathbf{A}_{0}))},\qquad
U_{2}=\mathfrak{N}_{4}(\sigma)\quad\text{on}\quad\overline{(\Gamma_{2}\cap\mathcal{O}_{\varepsilon}(\mathbf{A}_{0}))},\\
\sigma(r_{0},0)=0\quad\text{on}\quad\overline{(\Gamma\cap\mathcal{O}_{\varepsilon}(\mathbf{A}_{0}))},\qquad
U_{i}(x,0)=0\quad\text{in}\quad\overline{(\Omega_{i}\cap\mathcal{O}_{\varepsilon}(\mathbf{A}_{0}))},\quad
i=1,2.
\end{cases}
\end{equation}
Here, we put
\begin{align*}
\mathfrak{N}_{1}(U_{1},U_{2},\sigma)&=\widetilde{\mathfrak{N}}_{1}(U_{1},U_{2},\sigma)+\frac{k_{1}\sqrt{1+(\varphi'_{0})^{2}}}{1-k}
\widetilde{\mathfrak{N}}_{2}(U_{1},U_{2},\sigma)\\&
-\sigma\frac{k_{2}\varphi'_{0}\sqrt{1+(\varphi'_{0})^{2}}}{1-k}\frac{\partial}{\partial
r_{0}} \Big(\frac{\partial\mathcal{U}_{1,0}}{\partial n}
(1+(\varphi'_{0})^{2})^{-1/2}\Big),\\
\mathfrak{N}_{2}(U_{1},U_{2},\sigma)&=-\widetilde{\mathfrak{N}}_{2}(U_{1},U_{2},\sigma)+\sigma
\frac{\frac{\partial\mathcal{U}_{1,0}}{\partial
r_{0}}}{\frac{\partial\mathcal{U}_{1,0}}{\partial
n}}\frac{\partial}{\partial
r_{0}}\Big(\frac{\partial\mathcal{U}_{1,0}}{\partial n}
(1+(\varphi'_{0})^{2})^{-1/2}\Big).
\end{align*}
At last, collecting \eqref{5.10*} with Lemma \ref{l5.1} and
Corollary \ref{c5.1} arrives at  (i)-(iii) of this lemma. Thus, we
are left only to verify the bound in (iv) in
$\overline{(\Omega\cap\mathcal{O}_{\varepsilon}(\mathbf{A}_{0}))}$.
Appealing to explicit form of $\mathfrak{N}_{0,i}$ and
$\mathfrak{N}_{j}$, we easily compute  $\mathfrak{N}_{0,i}(0,0)$ and
$\mathfrak{N}_{i}(0,0,0)$, $i=1,2,$ $\mathfrak{N}_{j}(0)$ $j=3,4$.
In particulary, we have
\begin{align*}
&\mathfrak{N}_{2}(0,0,0)=(k-1)\varphi'_{0}\Big(\frac{\partial\rho}{\partial
r_{0}}\Big)^{2}\frac{\partial\mathcal{U}_{1,0}}{\partial r_{0}},\\
&\mathfrak{N}_{3}(0)=-\chi\rho\int_{0}^{1}\frac{\partial\mathrm{p}_{1}(x_{1},x_{2}-z\chi(\rho+\sigma))}{\partial(x_{2}-z\chi(\rho+\sigma))}dz,\quad
\mathfrak{N}_{4}(0)=-\chi\rho\int_{0}^{1}\frac{\partial\mathrm{p}_{2}(x_{1},x_{2}-z\chi(\rho+\sigma))}{\partial(x_{2}-z\chi(\rho+\sigma))}dz,\\
&\mathfrak{N}_{1}(0,0,0)=k_{1}(1+(\varphi'_{0})^{2})\Big(\frac{\partial\rho}{\partial
r_{0}}\frac{\partial \mathcal{U}_{1,0}}{\partial r_{0}}+
\Big(\frac{\partial\rho}{\partial r_{0}}\Big)^{2}\frac{\partial
\mathcal{U}_{1,0}}{\partial x_{2}}
+\frac{\varphi'_{0}}{\sqrt{1+(\varphi'_{0})^{2}}}\Big(\frac{\partial\rho}{\partial
r_{0}}\Big)^{2}\frac{\partial \mathcal{U}_{1,0}}{\partial r_{0}}\Big),\\
&\mathfrak{N}_{0,i}(0,0,0)=-2\frac{\chi\rho_{x_{1}}-\chi'\rho}{1-\chi'\rho}\frac{\partial^{2}\mathcal{U}_{i,0}}{\partial
x_{1}\partial x_{2}}
-\frac{(\chi\rho_{x_{1}}-\chi'\rho)^{2}+\chi'\rho(2-\chi'\rho)}{(1-\chi'\rho)^{2}}\frac{\partial^{2}\mathcal{U}_{i,0}}{\partial
x_{2}^{2}}\\
& -\frac{\partial\mathcal{U}_{i,0}}{\partial
x_{2}}\Big[\frac{[\chi'\rho_{x_{1}}-\chi''\rho][(\chi\rho_{x_{1}}-\chi'\rho)(1-\chi'\rho)^{-1}-\varphi'_{0}]+\chi\rho_{x_{1}x_{1}}-\chi'\rho_{x_{1}}}{1-\chi'\rho}
+\frac{\chi''\rho}{(1-\chi'\rho)^{2}}
\\
&+
\frac{[\chi\rho_{x_{1}}-\chi'\rho][\chi'\rho_{x_{1}}+\chi''\rho(\frac{\chi\rho_{x_{1}}-\chi'\rho}{1-\chi'\rho}-\varphi'_{0})]}{(1-\chi'\rho)^{2}}
 \Big].
\end{align*}
After that, these explicit representations together with the
smoothness of the functions $\rho,\mathcal{U}_{i,0}$ and
$\mathrm{p}_{i}$ (see assumption (h6), Lemma \ref{l5.1} and
Corollary \ref{c5.1}) provide the desired estimates.

\end{proof}


\subsection{Solvability of the linear problem corresponding to
\eqref{5.7}}\label{s5.3}
 Relations in \eqref{5.7} suggest that the linear system
 corresponding to the linear operator $\mathfrak{L}$ in
 \eqref{5.1} has the form
 \begin{equation}\label{5.10}
\begin{cases}
\Delta_{x}w_{1}=F_{0,1}(x,t)\qquad\qquad\text{in}\quad\Omega_{1,T},\\
\Delta_{x}w_{2}=F_{0,2}(x,t)\qquad\qquad\text{in}\quad\Omega_{2,T},\\
w_{1}-w_{2}=A_{0}(x)\sigma\qquad\qquad\text{on}\quad\Gamma_{T},\\
\frac{\partial\sigma}{\partial t}=A_{1}(x)[\tfrac{\partial
w_{1}}{\partial n}-\tfrac{\partial w_{2}}{\partial
n}]+A_{2}(x)[\tfrac{\partial w_{1}}{\partial \omega}-\tfrac{\partial
w_{2}}{\partial \omega}]+F_{1}(x,t)\quad\text{on}\quad\Gamma_{T},\\
\tfrac{\partial w_{1}}{\partial n}-k\tfrac{\partial w_{2}}{\partial
n}-kA_{3}(x)[\tfrac{\partial w_{1}}{\partial \omega}-\tfrac{\partial
w_{2}}{\partial \omega}]=F_{2}(x,t)\qquad\qquad\quad\text{ on}\quad\Gamma_{T},\\
w_{1}=F_{3}(x,t)\qquad\qquad\text{on}\quad\Gamma_{1,T},\\
w_{2}=F_{4}(x,t)\qquad\qquad\text{on}\quad\Gamma_{2,T},\\
\sigma(\omega,0)=0\qquad\text{in}\quad\Gamma,\quad
w_{i}(x,0)=0\quad\text{in}\quad\bar{\Omega}_{i},\,\,  i=1,2,
\end{cases}
\end{equation}
where $F_{0,i},$ $i=1,2,$ and $F_{j}$ are given functions specified
below. As for  the coefficients $A_{l},$ $l=0,1,2,3,$ they are
described in point (i) of Lemma \ref{l5.2}, in particulary, the
asymptotic holds
\begin{equation}\label{5.11}
A_{0}(x)\sim\begin{cases}
\mathrm{A}_{0,1}r^{s*-1}_{0}(x)\quad\text{as}\quad
x\to\mathbf{A}_{0},\\
\mathrm{A}_{0,0}r^{s*-1}_{1}(x)\quad\text{as}\quad
x\to\mathbf{A}_{1}
\end{cases}
\end{equation}
with negative constants $\mathrm{A}_{0,0}$ and $\mathrm{A}_{0,1}$.
\begin{theorem}\label{t5.1}
Let $k\in(0,1)$ and assumptions (h1)--(h3), (h5) and (h7) hold. We
assume that
\[
F_{0,i}\in\underset{0}{E}\,_{s}^{\beta,\beta,\beta}(\bar{\Omega}_{i,T}),\quad
F_{1},F_{2}\in\underset{0}{E}\,_{s+1}^{1+\beta,\beta,\beta}(\Gamma_{T}),
\,
F_{3}\in\underset{0}{E}\,_{2+s}^{2+\beta,\beta,\beta}(\Gamma_{1,T}),\,
F_{4}\in\underset{0}{E}\,_{2+s}^{2+\beta,\beta,\beta}(\Gamma_{2,T}).
\]
Then problem \eqref{5.10} admits a unique local classical solution
\[
w_{1}\in
\underset{0}{E}\,_{s+2}^{2+\beta,\beta,\beta}(\bar{\Omega}_{1,T}),\quad
w_{2}\in
\underset{0}{E}\,_{s+2}^{2+\beta,\beta,\beta}(\bar{\Omega}_{2,T}),\quad
\sigma\in
\underset{0}{\mathcal{E}}\,_{s+2,s^{*}-1}^{2+\beta,\beta,\beta}(\Gamma_{T}).
\]
Besides, the estimate holds
\begin{align}\label{5.11*}\notag
&\|w_{1}\|_{E_{s+2}^{2+\beta,\beta,\beta}(\bar{\Omega}_{1,T})}+\|w_{2}\|_{E_{s+2}^{2+\beta,\beta,\beta}(\bar{\Omega}_{2,T})}
+
\|\sigma\|_{\mathcal{E}_{s+2,s^{*}-1}^{2+\beta,\beta,\beta}(\Gamma_{T})}\\
&\leq
C[\sum_{i=1}^{2}\|F_{0,i}\|_{E_{s}^{\beta,\beta,\beta}(\bar{\Omega}_{i,T})}+\|F_{1}\|_{E_{s+1}^{1+\beta,\beta,\beta}(\Gamma_{T})}
+\|F_{2}\|_{E_{s+1}^{1+\beta,\beta,\beta}(\Gamma_{T})}
\\\notag
& + \|F_{3}\|_{E_{s+2}^{2+\beta,\beta,\beta}(\Gamma_{1,T})} +
\|F_{4}\|_{E_{s+2}^{2+\beta,\beta,\beta}(\Gamma_{2,T})}]
\end{align}
with the positive constant $C$ being independent of the right-hand
sides.
\end{theorem}
\begin{proof}
At first, we prove this claim in the case of the special right-hand
sides
\begin{equation}\label{5.12}
F_{0,1},\,F_{0,2},\, F_{3},\, F_{4}\equiv 0.
\end{equation}
After that, we discuss how this restriction may be removed.

To verify the one-valued local solvability  of \eqref{5.10} in the
case of \eqref{5.12}, we exploit the continuation method. To this
end, for each $\epsilon\in[0,1]$, we consider the family of problems
\begin{equation}\label{5.13}
\begin{cases}
\Delta_{x}w_{1}=0\qquad\text{in}\quad\Omega_{1,T},\quad w_{1}=0\qquad\quad\text{ on}\quad\Gamma_{1,T},\\
\Delta_{x}w_{2}=0\qquad\text{in}\quad\Omega_{2,T},\quad
w_{2}=0\qquad\quad\text{ on}\quad\Gamma_{2,T},\\
\sigma(\omega,0)=0\quad\text{ in}\quad\Gamma,\quad
w_{i}(x,0)=0\quad\text{in}\quad\bar{\Omega}_{i},\,\,  i=1,2,\\
w_{1}-w_{2}=A_{0}(x)\sigma\qquad\text{on}\quad\Gamma_{T},\\
\tfrac{\partial w_{1}}{\partial n}-k\epsilon\tfrac{\partial
w_{2}}{\partial n}-k\epsilon A_{3}(x)[\tfrac{\partial
w_{1}}{\partial \omega}-\tfrac{\partial
w_{2}}{\partial \omega}]=0\qquad\qquad\qquad\text{ on}\quad\Gamma_{T},\\
\frac{\partial\sigma}{\partial t}=A_{1}(x)[\tfrac{\partial
w_{1}}{\partial n}-\tfrac{\partial w_{2}}{\partial
n}]+A_{2}(x)[\tfrac{\partial w_{1}}{\partial \omega}-\tfrac{\partial
w_{2}}{\partial \omega}]+F_{1}(x,t)\quad\text{on}\quad\Gamma_{T},\\
\end{cases}
\end{equation}
The continuation approach  (in the case of linear problem) tells
that  the one-to-one solvability of \eqref{5.10} is provided by

\noindent (I)  the one-valued solvability in the case of
$\epsilon=0$ in \eqref{5.13};

\noindent(II) the a priory estimates of a solution to \eqref{5.13},
which will be uniform in $\epsilon\in[0,1]$.

At this point, we discuss each stage, separately.

\noindent(I) Clearly, if $\epsilon=1$, then problem \eqref{5.13}
boils down with \eqref{5.10}. The case of $\epsilon=0$ splits
\eqref{5.13} into two problems, the first concerning to $w_{1}$  is
the Dirichlet-Neumann problem for Laplace equation, while the second
dealing with $w_{2}$ and $\sigma$ is the boundary value problem
subject to a dynamic boundary condition.
 Indeed, in $\Omega_{1,T}$ we have
\begin{equation}\label{5.14}
\Delta_{x} w_{1}=0\qquad\text{in}\quad\Omega_{1,T},\quad
w_{1}(x,0)=0\quad \text{in}\quad\Omega_{1},\quad \frac{\partial
w_{1}}{\partial n}=0\quad\text{on}\quad \Gamma_{T},\quad
w_{1}=0\quad\text{on}\quad\Gamma_{1,T},
\end{equation}
where the standard theory to elliptic linear equations provides
$w_{1}\equiv 0$.

The second problem concerns with the finding $w_{2}(x,t)$ and
$\sigma(\omega,t)$ by conditions
\begin{equation}\label{5.15}
\begin{cases}
\Delta_{x}w_{2}=0\qquad\text{in}\quad\Omega_{2,T},\quad
w_{2}=0\qquad\quad\text{ on}\quad\Gamma_{2,T},\\
\sigma(\omega,0)=0\quad\text{ in}\quad\Gamma,\quad
w_{2}(x,0)=0\quad\text{in}\quad\bar{\Omega}_{2},\\
w_{2}=-A_{0}(x)\sigma\qquad\text{on}\quad\Gamma_{T},\\
\frac{\partial\sigma}{\partial t}=-A_{1}(x)\tfrac{\partial
w_{2}}{\partial n}-A_{2}(x)\tfrac{\partial w_{2}}{\partial
\omega}+F_{1}(x,t)\quad\text{on}\quad\Gamma_{T},
\end{cases}
\end{equation}
The one-valued classical solvability of this problem is proved in
\cite[Section 4]{V3}. In particulary, the following regularity:
$w_{2}\in
\underset{0}{E}\,_{s+2}^{2+\beta,\beta,\beta}(\bar{\Omega}_{2,T}),$
$ \sigma\in
\underset{0}{\mathcal{E}}\,_{s+2,s^{*}-1}^{2+\beta,\beta,\beta}(\Gamma_{T}),
$  is established as well the corresponding estimate is obtained.

In fine, we end up with the unique solution $(w_{1},w_{2},\sigma)$
of \eqref{5.13} in the case $\epsilon=0$, and, besides, the desire
bound holds for this solution. This completes the verification of
 (I).

\noindent(II) Coming to the a priori estimates in \eqref{5.13}, we
utilize the standard Schauder technique (see e.g., \cite[Sections
4.1 and 6.3]{K} and \cite[Section 6]{GT}) accounting in the case of
problem \eqref{5.13} three (pretty standard) steps:

\noindent(i) building a so-called partition of unity in
$\overline{\Omega_{1}\cup\Omega_{2}}_{T}$ (see e.g. \cite[Section
6.3]{K});

\noindent(ii) "freezing" the coefficients in \eqref{5.13} via a
standard technique (see e.g. \cite[Section 6]{GT});

\noindent(iii) obtaining the one-to-one classical  solvability of
the corresponding model problems in $E_{s+2}^{2+\beta,\beta,\beta}$.

It is worth noting that, stages (i) and (ii) recast literally
(almost) proofs from \cite[Section 6.3]{K}, and we omit them here.
Coming to the last step, the only difference compared to the
arguments from \cite[Section 6]{GT} is concerned with the discussion
of the nonclassical transmission problems with a dynamic boundary
condition stated in $G_{1,T}\cup G_{2,T}$ and
$\R_{T}^{+}\cup\R^{-}_{T}$. This problem (see asymptotic
\eqref{5.11}) in the case of $G_{1,T}\cup G_{2,T}$ is analyzed in
Section \ref{s4} (with $\delta=\pi/2-\delta_{i},$ $i=0,1$), while
the same problem in the case of $\R_{T}^{+}\cup\R^{-}_{T}$ is
studied in \cite[Section 3.2]{BV2}. Thus, collecting all the
obtained results, we end up with the bound
\begin{align*}
&\|w_{1}\|_{E_{s+2}^{2+\beta,\beta,\beta}(\bar{\Omega}_{1,T})}+\|w_{2}\|_{E_{s+2}^{2+\beta,\beta,\beta}(\bar{\Omega}_{2,T})}
+
\|\sigma\|_{\mathcal{E}_{s+2,s^{*}-1}^{2+\beta,\beta,\beta}(\Gamma_{T})}\\
& \leq C_{2}[\|F_{1}\|_{E_{s+1}^{1+\beta,\beta,\beta}(\Gamma_{T})}
+\langle w_{1}\rangle^{(\beta)}_{t,s,\Omega_{1,T}} + \langle
w_{2}\rangle^{(\beta)}_{t,s,\Omega_{2,T}}]
\end{align*}
with the constant being independent of $\epsilon$.

To manage the last two terms in the right-hand sides, we utilize
(iii) in Lemma \ref{l5.0} to \eqref{5.13} (with excepting the last
condition on $\Gamma_{T}$), where we set $\phi_{2}=A_{0}\sigma$ and
$W_{i}=w_{i},$ $i=1,2,$ and arrive at the inequalities
\begin{align*}
\langle w_{1}\rangle^{(\beta)}_{t,s,\Omega_{1,T}} + \langle
w_{2}\rangle^{(\beta)}_{t,s,\Omega_{2,T}}&\leq
T^{1-\beta}[\underset{\bar{\Omega}_{1,T}}{\sup}\,
r^{-s}(x)|w_{1}|+\underset{\bar{\Omega}_{2,T}}{\sup}\,
r^{-s}(x)|w_{2}|]\\
& \leq
C_{1}T^{1-\beta}\|\sigma_{t}r^{s^{*}-1}(y)\|_{\C([0,T],E^{1+\beta}_{1+s}(\Gamma))}\\
& \leq
C_{1}|\Upsilon|^{s^{*}-1}T^{1-\beta}\|\sigma_{t}\|_{E^{1+\beta,\beta,\beta}_{1+s}(\Gamma_{T})}.
\end{align*}
Here we used the positivity of $s^{*}-1$.

In fine, collecting all estimates and selecting positive time $T$
satisfying the inequality
\[
C_{1}C_{2}|\Upsilon|^{s^{*}-1}T^{1-\beta}<\frac{1}{2},
\]
we end up with \eqref{5.11*} in the case of \eqref{5.12} with the
constant being independent of $\epsilon$ and the right-hand sides.
This completes the proof of Theorem \ref{t5.1} in the special case
\eqref{5.12}.

In order to reach this claim in the general case, we look for a
solution to the original problem \eqref{5.10} with inhomogeneous
right-hand sides in the form
\[
w_{1}=W_{1}+\mathcal{W}_{1},\qquad w_{2}=W_{2}+\mathcal{W}_{2},
\]
where $(W_{1},W_{2})$ solves transmission problem \eqref{7.1} with
$\phi_{0,i}=F_{0,i},$ $\phi_{1}=0,$ $\phi_{2}=F_{2},$
$\phi_{3}=F_{3},$ $\phi_{4}=F_{4}$, while
$(\mathcal{W}_{1},\mathcal{W}_{2},\sigma)$ solves \eqref{5.13} with
$\epsilon=1$ and new
\[
F_{1}:=F_{1}+A_{1}\Big[\frac{\partial W_{1}}{\partial
n}-\frac{\partial W_{2}}{\partial n}\Big].
\]
Thus, exploiting Lemma \ref{l5.0} and Theorem \ref{t5.1} with
assumption \eqref{5.12} completes the proof of this claim in the
general case.
\end{proof}


\subsection{Completion of the proof of Theorem \ref{t3.1}}\label{s5.4}

Coming to the nonlinear problem \eqref{5.1} with
$\mathrm{z}=(U_{1},U_{2},\sigma)$ (see also \eqref{5.7}), we first
introduce the functional spaces $\mathrm{H}_{1}$ and
$\mathrm{H}_{2}$, $\mathrm{z}\in\mathrm{H}_{1}$ and
$\mathfrak{N}(\mathrm{z})\in\mathrm{H}_{2},$
\begin{align*}
\mathrm{H}_{1}&=
\underset{0}{E}\,_{s+2}^{2+\beta,\beta,\beta}(\bar{\Omega}_{1,T})\times
\underset{0}{E}\,_{s+2}^{2+\beta,\beta,\beta}(\bar{\Omega}_{2,T})\times
\underset{0}{\mathcal{E}}\,_{s+2,s^{*}-1}^{2+\beta,\beta,\beta}(\Gamma_{T}),\\
\mathrm{H}_{2}&=\underset{0}{E}\,_{s}^{\beta,\beta,\beta}(\bar{\Omega}_{1,T})\times
\underset{0}{E}\,_{s}^{\beta,\beta,\beta}(\bar{\Omega}_{1,T})\times
\underset{0}{E}\,_{s+1}^{1+\beta,\beta,\beta}(\Gamma_{T})\times
\underset{0}{E}\,_{s+1}^{1+\beta,\beta,\beta}(\Gamma_{T})\\
&
\times\underset{0}{E}\,_{2+s}^{2+\beta,\beta,\beta}(\Gamma_{1,T})\times\underset{0}{E}\,_{2+s}^{2+\beta,\beta,\beta}(\Gamma_{2,T})
\end{align*}
endowed with the product norms
\begin{align*}
\|\mathrm{z}\|_{\mathrm{H}_{1}}&=\|U_{1}\|_{E_{s+2}^{2+\beta,\beta,\beta}(\bar{\Omega}_{1,T})}
+ \|U_{2}\|_{E_{s+2}^{2+\beta,\beta,\beta}(\bar{\Omega}_{2,T})} +
\|\sigma\|_{E_{s+2,s^{*}-1}^{2+\beta,\beta,\beta}(\Gamma_{T})},\\
\|\mathfrak{N}(\mathrm{z})\|_{\mathrm{H}_{2}}&=\sum_{i=1}^{2}[\|\mathfrak{N}_{0,i}(\mathrm{z})\|_{E_{s}^{\beta,\beta,\beta}(\bar{\Omega}_{i,T})}
+\|\mathfrak{N}_{i}(\mathrm{z})\|_{E_{s+1}^{1+\beta,\beta,\beta}(\Gamma_{T})}]\\
& +
\|\mathfrak{N}_{3}(\mathrm{z})\|_{E_{s+2}^{2+\beta,\beta,\beta}(\Gamma_{1,T})}
+
\|\mathfrak{N}_{4}(\mathrm{z})\|_{E_{s+2}^{2+\beta,\beta,\beta}(\Gamma_{2,T})}.
\end{align*}
Taking into account \eqref{5.7} and Lemma \ref{l5.2}, we rewrite
equation \eqref{5.1} in the form
\[
\mathfrak{L}\mathrm{z}=\mathrm{F}(x,t)+\overline{\mathfrak{N}}(\mathrm{z}),
\]
where $\mathfrak{L}:\mathrm{H}_{1}\to\mathrm{H}_{2}$ is the linear
operator whose properties are described in   Section \ref{s5.3}; the
vector $\mathrm{F}(x,t)$ is constructed via initial data, and
$\overline{\mathfrak{N}}(\mathrm{z})$ being similar to
$\mathfrak{N}(\mathrm{z})$ is described by Lemma \ref{l5.2}.

Utilizing Theorem \ref{t5.1}, we conclude that
\begin{equation}\label{5.16}
\mathrm{z}=\mathfrak{L}^{-1}\mathrm{F}(x,t)+\mathfrak{L}^{-1}\overline{\mathfrak{N}}(\mathrm{z}):=\mathfrak{M}(\mathrm{z}).
\end{equation}
\begin{lemma}\ref{l5.1}
Let $\mathrm{B}_{\mathrm{R}}\subset\mathrm{H}_{1}$ be a ball
centered in the origin and having the radius $\mathrm{R}$. For each
$\mathrm{z}_{1}, \mathrm{z}_{2}\in \mathrm{B}_{\mathrm{R}}$, the
following estimates hold
\[
\|\overline{\mathfrak{N}}(0)\|_{\mathrm{H}_{2}}\leq C_{3}(T),\qquad
\|\overline{\mathfrak{N}}(\mathrm{z}_{1})-\overline{\mathfrak{N}}(\mathrm{z}_{2})\|_{\mathrm{H}_{2}}\leq
C_{4}(T,\mathrm{R})\|\mathrm{z}_{1}-\mathrm{z}_{2}\|_{\mathrm{H}_{1}}
\]
with quantities $C_{3}(T)$ and $C_{4}(T,\mathrm{R})$ vanishes if $T,
\mathrm{R}$ tend to zero.
\end{lemma}
The proof of this claim is verified with recasting the arguments of
\cite[Section 5]{V2} and utilizing Lemmas \ref{l5.0}-\ref{l5.2},
Theorem \ref{t5.1} and Corollary \ref{c5.1}.

In fine, Lemma \ref{l5.1} tells that for sufficiently small
$T=T^{*}$ and $\mathrm{R}$, the nonlinear operator $\mathfrak{M}$
meets the requirements of the fixed point theorem for a contraction
operator. This, in turn, arrives at  a unique fixed point of
\eqref{5.16}, which will be a unique local solution of nonlinear
problem \eqref{5.7}. This completes the proof of Theorem
\ref{t3.1}.\qed


\section{Solvability of \eqref{3.3} in the case of arbitrary
$\delta_{i}\in(0,\pi/4)$}\label{s6}

\noindent In this section we aim to obtain the results similar to
Theorem \ref{t3.1} if $\delta_{0},\delta_{1}\in(0,\pi/4)$ do not
obey  assumption (h5), that is we not assume that these angles are
rational part of $\pi$. The arguments describing in Sections
\ref{s3}-\ref{s5} tell that this assumption is exploited only to
verify Theorem \ref{t4.2} if
\begin{equation}\label{6.1}
f_{0,i},f_{2}\equiv 0,
\end{equation}
while the remaining parts of the proof to Theorem \ref{t3.1} work in
the case of arbitrary $\delta_{i}\in(0,\pi/4)$. Thus, we are left to
remove restriction \eqref{4.1} in the arguments leading to  Theorem
\ref{t4.2} if \eqref{6.1} holds. To this end, we utilize the
technique proposed in \cite[Section 3.3]{BV1} and modified here
below to our target. Namely, we construct the solution
$(u_{1},u_{2})$ of \eqref{4.7} corresponding to any irrational
$\delta\in(\pi/4,\pi/2)$ as a limit of the approximating solutions
$(u^{m}_{1},u^{m}_{2})$ corresponding to $\delta^{m}$, which
satisfies  \eqref{4.1} and converges to $\delta$ as $m\to+\infty$.

At first, for any irrational $\delta$, we introduce an infinite
sequence of rational $c_{m}$ such that
\[
\delta^{m}=c_{m}\pi\to\delta\qquad\text{as}\quad m\to+\infty.
\]
Then, for each $m$, we define $G_{i,T}^{m}$ and $g_{m,T}$ via
replacing $\delta$ by $\delta^{m}$ in the relations of $G_{i,T}$ and
$g_{T}$. In fine, for each $m$, we consider \eqref{4.7} in the
unknown $(u^{m}_{1},u^{m}_{2})$ defined in $G_{1,T}\times G_{2,T}$
and the same right-hand side $f_{1}$. It is apparent that, the
assumptions of Theorem \ref{t4.2} hold in this case and, hence,  we
end up with the unique solution $(u^{m}_{1},u^{m}_{2})$ satisfying
the estimates
\begin{align}\label{6.2}\notag
&
\sum_{i=1}^{2}\Big(\|u_{i}^{m}\|_{E_{s_{m}+2}^{2+\beta,\beta,\beta}(\bar{G}_{i,T}^{m})}+
\Big\|r_{0}(y)^{1-s^{*}}\frac{\partial u_{i}^{m}}{\partial
t}\Big\|_{E_{s_{m}+1}^{1+\beta,\beta,\beta}(\bar{g}_{T}^{m})}
\Big)\leq
C\|f_{1}\|_{_{E_{s_{m}+1}^{1+\beta,\beta,\beta}(\bar{g}_{T}^{m})}},\\
&
\sum_{i=1}^{2}\|u_{i}^{m}\|_{E_{s_{m}+2}^{1+\beta,\beta,\beta}(\bar{G}_{i,T}^{m})}
\leq
CT^{\beta^{*}-\beta}\|f_{1}\|_{_{E_{s_{m}+1}^{1+\beta,\beta,\beta}(\bar{g}_{T}^{m})}}
\end{align}
with the constants being independent of $m$ and
\[
s_{m}+2\in(\max\{2,\z^{*}_{2,m}\},\max\{3,\pi/\delta^{m}\},
\]
where
$\z^{*}_{2,m}=\max\{\underline{\z_{2,m}},\overline{\z_{2,m}}\}$ with
$\underline{\z_{2,m}},\,\overline{\z_{2,m}}$ satisfying \eqref{4.5}.

Actually, in these estimates,  we can select the same value $s$ for
all $m$. Indeed, setting
\[
\z^{*}=\max\{\underset{m\to+\infty}{\overline{\lim}}\underline{\z_{2,m}},\underset{m\to+\infty}{\overline{\lim}}\overline{\z_{2,m}}
\}
\]
and bearing in mind the definition of $\delta^{m},$
$\underline{\z_{2,m}},\,\overline{\z_{2,m}}$, we may set $s_{m}=s$
in \eqref{6.2}, where
\begin{equation}\label{6.0}
s+2\in (\max\{2,\z^{*}\},\max\{3,\pi/\delta\}.
\end{equation}
Then, performing the change of variables \eqref{2.2} in \eqref{4.7},
we arrive at the problem
\begin{equation}\label{6.3}
\begin{cases}
\Delta_{x} u_{i}^{m}=0\qquad \text{in}\quad B^{m}_{i,T},\\
u_{i}^{m}(x,0)=0\quad\text{in}\quad \bar{B}_{i}^{m},\, i=1,2,
\\
e^{-s^{*}x_1}\frac{\partial(u_1^{m}-u_2^{m})}{\partial t} -
\frac{\partial(u_1^{m}-u_2^{m})}{\partial x_2}+a_2
\frac{\partial (u_1^{m}-u_2^{m})}{\partial x_1}=-e^{x_{1}}f_{1}(x,t)\quad\text{on}\quad b^{m}_{T},\\
\frac{\partial u_1^{m}}{\partial x_2}-k\frac{\partial
u_2^{m}}{\partial x_2}+ka_3
\frac{\partial (u_1^{m}-u_2^{m})}{\partial x_1}=0\qquad\qquad\text{on}\quad b^{m}_{T},\\
u_1^{m}(x_1,-\pi/2,t)=0\qquad\text{and}\qquad
u^{m}_2(x_1,\pi/2,t)=0,\quad \quad x_1\in\R,\, t\in[0,T],
\end{cases}
\end{equation}
where $B_{i}^{m}$ and $b^{m}$ are obtained from $B_{i}$ and $b$ via
replacing $\delta$ by $\delta^{m}$.

After that,  the domains $B_{i}^{m}$ and $b^{m}$ are transformed to
$B_{i}$ and $b$ by the change of variables:
\[
\bar{x}_{1}=x_{1}\qquad\text{and}\qquad
\bar{x}_{2}=x_{2}-\chi(x_{2})(\delta^{m}-\delta),
\]
where $\chi\in\C_{0}^{\infty}$ and $\chi\in[0,1],$
\[
\chi(x_{2})=\begin{cases} 1\qquad\text{if}\quad
x_{2}\in(\delta^{m}-\epsilon,\delta^{m}+\epsilon),\\
0\qquad\text{if}\quad
x_{2}\notin(\delta^{m}-2\epsilon,\delta^{m}+2\epsilon)
\end{cases}
\]
with the same $\epsilon$ for all $m$,
$0<\epsilon<\frac{\pi}{4}(\frac{1}{2}-\delta^{m})$.

\noindent In sum, saving the previous notations for $u_{i}^{m}$ and
$f_{1}$, we rewrite \eqref{6.3} in the form
\begin{equation}\label{6.4}
\begin{cases}
\Delta_{\bar{x}}
u_{i}^{m}=(\delta^{m}-\delta)[2-(\delta^{m}-\delta)\chi']\chi'\frac{\partial^{2}u_{i}^{m}}{\partial\bar{x}_{2}^{2}}+
(\delta^{m}-\delta)\chi''\frac{\partial
u_{i}^{m}}{\partial\bar{x}_{2}}\equiv f_{0,i}^{m}
\quad \text{in}\quad B_{i,T},\\
u_{i}^{m}(\bar{x},0)=0\quad\text{in}\quad \bar{B}_{i},\, i=1,2,
\\
e^{-s^{*}\bar{x}_1}\frac{\partial(u_1^{m}-u_2^{m})}{\partial t} -
\frac{\partial(u_1^{m}-u_2^{m})}{\partial \bar{x}_2}+a_2
\frac{\partial (u_1^{m}-u_2^{m})}{\partial \bar{x}_1}=-e^{\bar{x}_1}f_{1}(\bar{x},t)\qquad\text{on}\quad b_{T},\\
\frac{\partial u_1^{m}}{\partial \bar{x}_2}-k\frac{\partial
u_2^{m}}{\partial \bar{x}_2}+ka_3
\frac{\partial (u_1^{m}-u_2^{m})}{\partial \bar{x}_1}=0\qquad\qquad\text{on}\quad b_{T},\\
u_1^{m}(\bar{x}_1,-\pi/2,t)=0\qquad\text{and}\qquad
u^{m}_2(\bar{x}_1,\pi/2,t)=0,\quad \quad \bar{x}_1\in\R,\,
t\in[0,T].
\end{cases}
\end{equation}
Taking into account \eqref{6.2} and performing the straightforward
calculations, we yield
\begin{equation}\label{6.5}
\|f_{0,i}^{m}e^{-s\bar{x}_{1}}\|_{\C^{\beta,\beta,\beta}(\bar{B}_{i,T})}\to
0\qquad \text{as}\quad m\to+\infty.
\end{equation}
Arguments of Section \ref{s4} ensure the one-valued classical
solvability of \eqref{6.4} for each $m$ and, besides, estimates
\eqref{6.2} and Corollary \ref{c2.2} derive the inequalities
\begin{align}\label{6.6}\notag
&
\sum_{i=1}^{2}(\|e^{-(s+2)\bar{x}_{1}}u_{i}^{m}\|_{\C^{2+\beta,\beta,\beta}(\bar{B}_{i,T})}+
\|e^{-(s+s^{*})\bar{x}_{1}}\tfrac{\partial u_{i}^{m}}{\partial
t}\|_{\C^{1+\beta,\beta,\beta}(\bar{b}_{T})} )\leq
C\|f_{1}e^{-(s+1)\bar{x}_{1}}\|_{\C^{1+\beta,\beta,\beta}(\bar{b}_{T})},\\
&
\sum_{i=1}^{2}\|e^{-(s+2)\bar{x}_{1}}u_{i}^{m}\|_{\C^{1+\beta,\beta,\beta}(\bar{B}_{i,T})}
\leq
CT^{\beta^{*}-\beta}\|f_{1}e^{-(s+1)\bar{x}_{1}}\|_{\C^{1+\beta,\beta,\beta}(\bar{b}_{T})}.
\end{align}
Finally, since each bounded subset of $\C^{1+\beta,\beta,\beta}$ is
a compact, we exploit \eqref{6.5}, \eqref{6.6} and pass to limit in
\eqref{6.4}. As a result, we end up with a unique solution
$(u_{1},u_{2})$ of \eqref{4.7} for any $\delta\in(\pi/4,\pi/2)$.
Besides, \eqref{6.6} ensures the desired bounds and the regularity
of the constructed solution.
\begin{theorem}\label{t6.1}
Let $\delta\in(\pi/4,\pi/2)$ be irrational and $s$ satisfy
\eqref{6.0}. Then, under assumptions of Theorem \ref{t4.2}, problem
\eqref{4.2} admits a unique classical solution having the regularity
established by Theorem \ref{t4.2}.
\end{theorem}
\begin{remark}\label{r6.1}
Recasting the arguments above provides  Theorem \ref{t4.3}
 in the case of an arbitrary
$\delta\in(\pi/4,\pi/2)$.
\end{remark}

Now, we are ready to prove Theorem \ref{t3.1} in the case of
arbitrary $\delta_{0},\delta_{1}\in(0,\pi/4)$. To this end, for each
irrational $\delta_{i}$ we introduce
\[
\delta_{0}^{m}=\pi(\tfrac{1}{2}-c^{0}_{m})\qquad \text{and}\quad
\delta_{1}^{m}=\pi(\tfrac{1}{2}-c^{1}_{m})
\]
with $c^{0}_{m}$ and $c^{1}_{m}$ being rational sequences such that
\[
\delta_{0}^{m}\to\delta_{0}\qquad\text{and}\qquad
\delta_{1}^{m}\to\delta_{1}\quad\text{as}\quad m\to+\infty.
\]
After that, we set
\begin{equation}\label{6.7}
\mathfrak{h}^{*}=\max\{\underset{m\to+\infty}{\overline{\lim}}\underline{\mathfrak{h}_{m}},
\underset{m\to+\infty}{\overline{\lim}}\overline{\mathfrak{h}_{m}}\},\qquad
\mathfrak{f}^{*}=\max\{\underset{m\to+\infty}{\overline{\lim}}\underline{\mathfrak{f}_{m}},
\underset{m\to+\infty}{\overline{\lim}}\overline{\mathfrak{f}_{m}}\},
\end{equation}
where $\overline{\mathfrak{f}_{m}},\overline{\mathfrak{h}_{m}},$
$\underline{\mathfrak{f}_{m}},\underline{\mathfrak{h}_{m}}$ are
constructed via \eqref{0.1} with $\delta_{0}^{m}$ and
$\delta_{1}^{m}$ instead of $\delta_{0}$ and $\delta_{1}$. Then,
recasting the arguments leading to Theorem 3.1 (where Theorems
\ref{t4.2} and \ref{t4.3} are replaced by Theorem \ref{t6.1} and
Remark \ref{r6.1}), we claim.
\begin{theorem}\label{t6.2}
Let $\delta_{0},\delta_{1}\in(0,\pi/4)$ be irrational. Then, under
assumptions (h1)--(h4), (h6) and (h7) with $\mathfrak{h}^{*}$ and
$\mathfrak{f}^{*}$ given by \eqref{6.7},  the results of Theorem
\ref{t3.1} hold.
\end{theorem}





\begin{thebibliography}{60}


\bibitem{Ad}
R.A. Adams, Sobolev spaces, Academic Press, New York,
1975.

\bibitem{APW}
S. Agrawal, N. Patel, S. Wu, Regidity of acute angled corners for on
phase Muskat interfaces, Adv. Math. \textbf{412}(1) 108801 (2023).

\bibitem{Am}
D. Ambrose, Well-posedness of two-phase Hele-Shaw flow without surface tension, European J. Appl. Math. \textbf{15} (2004) 597--607.

\bibitem{BF}
B.V. Bazaliy, A. Friedman, The Hele-Shaw problem with surface
tension in a half plane, J. Differ. Equa. \textbf{216} (2005)
439--469.

\bibitem{BV1}
B.V. Bazaliy, N. Vasylyeva, The Muskat problem with surface tension and a nonregular initial interface, Nonlinear Anal. \textbf{74} (2011) 6074--6096.

\bibitem{BV2}
B.V. Bazaliy, N. Vasylyeva, The two-phase Hele-Shaw problem with a nonregular initial interface and without surface tension,
 J. Math. Phys. Anal. Geometry \textbf{10}(1) (2014) 3--43.

\bibitem{BV4}
B.V. Bazaliy, N. Vasylyeva, On the solvability of the Hele-Shaw
model problem in weighted H\"{o}lder spaces in a plane angle,
Ukrainian Math. J. \textbf{52} (2000) 1647--1660.

\bibitem{BCG}
L.C. Berselli, D. C\'{o}rdoba, R. Granero-Belinch\'{o}n, Local
solvability and turning for the inhomogeneous Muskat problem,
Interfaces Free Bound. \textbf{16} (2014) 175--213.

\bibitem{CCFG}
A. Castro, D. C\'{o}rdoba, C. Fefferman, F. Gancedo, Breakdown of
smoothness for the Muskat problem, Arch. Rational Mech. Anal.
\textbf{208} (2013) 805--909.

\bibitem{CCG}
A. C\'{o}rdoba, D. C\'{o}rdoba, F. Gancedo, Interface evolution: the
Hele-Shaw and Muskat problems, Ann. Math. \textbf{173} (2011)
477--542.

\bibitem{CCGS}
P. Constantin, D. C\'{o}rdoba, F. Gancedo, R.M. Strain, On the
global existence for the Muskat problem, J. Eur. Math. Soc.
\textbf{15} (2013) 201--227.



\bibitem{CGZ}
D. C\'{o}rdoba, J. G\'{o}mez-Serrano, A. Zlato\v{s}, A note on
stability shifting for the Muskat problem II: stable to unstable and
back to stable, Anal. PDE \textbf{10}(2) (2017) 367--378.

\bibitem{CGO}
D. C\'{o}rdoba Gazolaz, R. Granero-Belinch\'{o}n, R. Orive-Illera,
The confined Muskat problem: differences with the deep water regime,
Commun. Math. Sci. \textbf{12} (2014) 4223--455.

\bibitem{CGS}
C.H.A. Cheng, R. Granero-Belinch\'{o}n, S. Shkoller, Well-posedness
of the Muskat problem with $H^{2}$ initial data, Adv. Math.
\textbf{286} (2016) 32--104.


\bibitem{EEM}
M. Ehrustr\"{o}m,  J. Escher,  B-V. Matioc, Steady-state fingering
patterns for a periodic Muskat problem, Methods Appl. Anal.
\textbf{20} (2013) 33--46.

\bibitem{EMM}
J. Escher, A-V. Matioc, B-V. Matioc, A generalized Rayleigh-Taylor
condition for the Muskat problem, Nonlinearity \textbf{25} (2012)
73--92.

\bibitem{EMM2}
J. Escher, A-V. Matioc, B-V. Matioc, Modelling and analysis of the
Muskat problem for thin fluid layers, J. Math. Fluid Mech.
\textbf{14} (2012) 267--277.

\bibitem{EMW}
J. Escher, B-V. Matioc, C. Walker, The domain of parabolicity for
the Muskat problem, Indiana Univ. Math. J. \textbf{2} (2018)
679--737.

\bibitem{FT}
A. Friedman, Y. Tao, Nonlinear stability of the Muskat problem with
capillary pressure at the free boundary, Nonlinear Anal. \textbf{53}
(2003) 45--80.

\bibitem{GGHP}
E. Garcia-Ju\'{a}rez, J. G\'{o}mez-Serrano, S.V. Hazios, B.
Pausader, Desingularization of small moving corners for the Muskat
equation, Annal.  PDE, \textbf{10}(17) (2024) 1--71.

\bibitem{GGNP}
E. Garcia-Ju\'{a}rez, J. G\'{o}mez-Serrano, H.Q. Nguyen, B.
Pausader, Self-similar solutions for the Muskat equation, Adv. Math.
\textbf{399} 108294 (2022).

\bibitem{GT}
D. Gilbarg, N.S. Trudinger, Elliptic partial differential equations
of second order, 2 nd. edn. Springer, Berlin (1983).


\bibitem{H}
E.I. Hanzawa, Classical solution of the Stefan problem, Tohoku Math. J. \textbf{33} (1981) 297--335.

\bibitem{HTY}
J. Hong, Y. Tao, F. Yi, Muskat problem with surface tension, J.
Partial Diff. Equa. \textbf{10} (1997) 213--231.

\bibitem{K}
N.V. Krylov, Lecture on elliptic and parabolic equations in
H\"{o}lder spaces, Graduate Studies in Mathematics, v. 12, AMS
(1996).

\bibitem{MM}
 A-V. Matioc, B-V. Matioc, The Muskat problem with surface tension
 and equal viscosities in subcritical $L_{p}-$Sobolev spaces, J.
 Ellipt. Parabolic. Equa. \textbf{7} (2021) 635--670.


\bibitem{Mat}
 B-V. Matioc, The Muskat problem in two dimensions: equivalence of
 formulations, well-posedness, and regularity results, Anal.  PDE
 \textbf{12}(2) (2019) 281--332.

 \bibitem{Mat2}
 B-V. Matioc, Viscous displacement in porous media: the Muskat
 problem in $2$D, Transact. American Math. Soc. \textbf{370}(10)
 (2018) 7511--7556.

\bibitem{Mu}
M. Muskat, Two fluid systems in porous media: the encoroachment of
water into an oil sand, Physics \textbf{5} (1934) 250--264.

\bibitem{MW}
M. Muskat, R.D. Wyckoff, The flow of homogeneous fluids through
porous media, McGraw-Hill, New York, London (1937).

\bibitem{MV}
T. Mel'nyk, N. Vasylyeva, Asymptotic analysis of a contact Hele-Shaw
problem in a thin domain, Nonlinear Differ. Equa. Appl.
\textbf{27}(48) (2020) 1--37.

\bibitem{PS}
J. Pr\"{u}ss, G. Simonett, Moving interfaces and quasilineaar
parabolic evolution equations, Monographs in Mathematics,
\textbf{105}, Birkh\"{a}user, (2016).

\bibitem{SCH}
M. Soegel, R.E. Caflish, S. Howison, Global existence, singular
solutions and ill-posedness for the Muskat problem, Comm. Pure Appl.
Math. \textbf{57} (2004) 1374--1411.

\bibitem{V1}
N. Vasylyeva, Mixed Dirichlet-transmission problems in non-smooth
domains, In: Sadovnichiy, V.A., Zgurovsky, M.Z. (eds) Contemporary
Approaches and Methods in Fundamental Mathematics and Mechanics.
Understanding Complex Systems. \textbf{9}  Springer, Cham, (2021)
195--229. doi.org/10.1007/978-3-030-50302-4\underline{\,\,\,}9

\bibitem{V2}
N. Vasylyeva, On the solvability of some nonclassical boundary-value problem for the Laplace equation in the plane corner,
 Advan. Diff. Equa. \textbf{2}(10) (2007) 1167--1200.

\bibitem{V3}
N. Vasylyeva, Existence of smooth solutions of the Hele-Shaw problem
in a nonregular domain, Nonlinear Boundary Value Problems
\textbf{19} (2009) 12--28.

\bibitem{V4}
N. Vasylyeva, On a class of functional difference equations:
explicit solutions, asymptotic behavior and applications, Aequations
Mathematicae \textbf{98} (2024) 99--171.

\bibitem{V5}
N. Vasylyeva,  On a local solvability of the multidimensional Muskat
problem with a fractional derivative in time on the boundary
condition, Fract. Diff. Calculus, \textbf{4}(2) (2014) 89--124.

\bibitem{Y}
F. Yi, Local classical solution of Muskat free boundary problem, J.
Partial Diff. Equa. \textbf{9}(1) (1996) 84--96.

\bibitem{Y1}
F. Yi, Global classical solution of Muskat free boundary problem, J.
Math. Anal. Appl. \textbf{288} (2003) 442--461.


\end{thebibliography}
\end{document}